\documentclass[11pt,reqno,a4paper]{amsart}
\usepackage[utf8]{inputenc}
\usepackage{mathtools,amsmath,amssymb}
\usepackage[T1]{fontenc} \usepackage{ae} \usepackage{aecompl}
\usepackage{lmodern}
\usepackage{microtype}
\usepackage{hyperref}
\usepackage{xspace}
\usepackage{bibentry}
\usepackage{easybmat}
\usepackage{mathdots}
\usepackage[margin=2.32cm]{geometry}

\usepackage{cite}
\usepackage{graphicx}

\usepackage{subcaption}

\usepackage{bibentry}

\usepackage[usenames,dvipsnames,svgnames,table]{xcolor}
\definecolor{todocolor}{HTML}{D7E1E5}
\usepackage[textsize=tiny,
            backgroundcolor=todocolor,
            bordercolor=todocolor,
            linecolor=todocolor,
            textwidth=3.9cm]{todonotes}

\usepackage{comment}
\usepackage{rotating}

\usepackage{dynkin-diagrams}

\usepackage{tikz}
\usetikzlibrary{decorations.pathreplacing}

\numberwithin{equation}{section}

\usepackage{amsmath}
\usepackage{amsthm}

\makeatletter
\let\@wraptoccontribs\wraptoccontribs
\makeatother

\theoremstyle{plain}
\newtheorem{Thm}[equation]{Theorem}
\newtheorem{Cor}[equation]{Corollary}
\newtheorem{Lem}[equation]{Lemma}
\newtheorem{Prop}[equation]{Proposition}

\newtheorem{MainTh}{Main Theorem}
\newtheorem*{Thm-non}{Theorem}

\theoremstyle{definition}

\theoremstyle{remark}

\newtheorem{Rem}[equation]{Remark}

\newcommand{\osc}[1]{\mathbf{#1}}

\newcommand{\oa}{\osc{a}}
\newcommand{\ob}{\osc{b}}

\newcommand{\oad}{\osc{\bar{a}}}

\newcommand{\obdag}{\ob^\dagger}

\renewcommand{\wp}{\bar{\mathbf{w}}}
\newcommand{\wm}{\mathbf{w}}

\newcommand{\ap}{\bar{\mathbf{A}}}
\newcommand{\am}{{\mathbf{A}}}

\newcommand{\idb}{\JD }
\newcommand{\id}{{\ID }}

\newcommand{\ol}{\overline}
\newcommand{\iso}{{\stackrel{\sim}{\longrightarrow}}}

\DeclareMathOperator{\diag}{diag}

\DeclareMathOperator{\ch}{ch}

\newmuskip\pFqmuskip
\newcommand*\pFq[6][8]{%
    \begingroup
    \pFqmuskip=#1mu\relax
    \mathcode`\,=\string"8000
    \begingroup\lccode`\~=`\,
    \lowercase{\endgroup\let~}\pFqcomma
    {}_{#2}F_{#3}{\left(\genfrac..{0pt}{}{#4}{#5};#6\right)}%
    \endgroup}
\newcommand{\pFqcomma}{\mskip\pFqmuskip}

\def\be{\begin{eqnarray}}
\def\ee{\end{eqnarray}}

\DeclareMathOperator{\tr}{tr}

\newcommand{\sfrac}[2]{{\textstyle\frac{#1}{#2}}}

\newcommand{\ID}{\mathrm{I}}
\newcommand{\JD}{\mathrm{J}}

\newcommand{\Gop}{\mathrm{G}}
\newcommand{\rtt}{\mathrm{rtt}}
\newcommand{\End}{\mathrm{End}}
\newcommand{\Lie}{\mathrm{Lie}}
\newcommand{\coker}{\mathrm{coker}}
\newcommand{\codim}{\mathrm{codim}}
\newcommand{\Fock}{\mathsf{F}}
\newcommand{\opp}{\mathrm{op}}

\newcommand{\bN}{\mathbf{N}}

\newcommand{\fg}{\mathfrak{g}}
\newcommand{\fh}{\mathfrak{h}}
\newcommand{\fl}{\mathfrak{l}}

\newcommand{\fp}{\mathfrak{p}}
\newcommand{\fu}{\mathfrak{u}}
\newcommand{\fb}{\mathfrak{b}}
\newcommand{\gl}{\mathfrak{gl}}
\newcommand{\ssl}{\mathfrak{sl}}
\newcommand{\ssp}{\mathfrak{sp}}
\newcommand{\sso}{\mathfrak{so}}
\newcommand{\CA}{\mathcal{A}}
\newcommand{\CE}{\mathcal{E}}
\newcommand{\CD}{\mathcal{D}}
\newcommand{\CF}{\mathcal{F}}

\newcommand{\CL}{\mathcal{L}}
\newcommand{\CK}{\mathcal{K}}

\newcommand{\CO}{\mathcal{O}}
\newcommand{\CR}{\mathcal{R}}
\newcommand{\CS}{\mathcal{S}}
\newcommand{\CT}{\mathcal{T}}
\newcommand{\fL}{\mathfrak{L}}
\newcommand{\bfS}{\mathbf{S}}
\newcommand{\BZ}{\mathbb{Z}}
\newcommand{\BC}{\mathbb{C}}
\newcommand{\BN}{\mathbb{N}}

\newcommand{\st}{t}
\newcommand{\NK}{\mathsf{K}}
\newcommand{\sQ}{\mathsf{Q}}
\newcommand{\sT}{\mathsf{T}}

\newcommand{\wtr}{\widehat\tr}

\newcommand{\bL}{\bar{L}}

\makeatletter
\g@addto@macro\bfseries{\boldmath}
\makeatother

\makeatletter
\newcases{rrcases}{\quad}{%
  \hfil$\m@th\displaystyle{##}$}{$\m@th\displaystyle{##}$\hfil}{\lbrace}{.}
\makeatother

\begin{document}
\title[Transfer matrices of rational spin chains via novel BGG-type resolutions]
      {Transfer matrices of rational spin chains \\ via novel BGG-type resolutions}

\author{Rouven Frassek}
 \address{R.F.: University of Modena and Reggio Emilia, FIM, Via G.~Campi~213/b, 41125 Modena, Italy}
 \email{rouven.frassek@unimore.it}

\author{Ivan Karpov}
 \address{I.K.: MIT, Department of Mathematics, Cambridge, MA 02139, USA;
 on leave from National Research University Higher School of Economics, Department of Mathematics}
 \email{karpov57@mit.edu}

\author{Alexander Tsymbaliuk}
 \address{A.T.:  Purdue University, Department of Mathematics, West Lafayette, IN 47907, USA}
 \email{sashikts@gmail.com}

\begin{abstract}
We obtain BGG-type formulas for transfer matrices of irreducible finite-dimensional representations of the classical
Lie algebras $\fg$, whose highest weight is a multiple of a fundamental one and which can be lifted to the representations
over the Yangian $Y(\fg)$. These transfer matrices are expressed in terms of transfer matrices of certain
infinite-dimensional highest weight representations (such as parabolic Verma modules and their generalizations)
in the auxiliary space. We further factorise the corresponding infinite-dimensional transfer matrices into
the products of two Baxter $Q$-operators, arising from our previous study~\cite{fpt,Frassek:2021ogy} of the
degenerate Lax matrices. Our approach is crucially based on the new BGG-type resolutions of the finite-dimensional
$\fg$-modules, which naturally arise geometrically as the restricted duals of the Cousin complexes of relative
local cohomology groups of ample line bundles on the partial flag variety $G/P$ stratified by $B_{-}$-orbits.
\end{abstract}

\maketitle
\tableofcontents


\medskip

\section{Introduction}


\subsection{Summary}
$\ $

The main results of the present paper are:
\begin{enumerate}

\item[$\bullet$]
The construction of new BGG-type resolutions of finite-dimensional $\fg$-modules by infinite-dimensional
highest weight modules such as parabolic Verma and their ``$W$-translations''. These resolutions
admit an elegant geometric interpretation via the Cousin complexes of relative local cohomology, similar
to the parabolic BGG resolutions of~\cite{l}. The crucial difference from the latter setup, however, is
that our modules are also defined when the ``dominant integral'' condition is lifted, and become generically
irreducible in this setting.

\medskip

\item[$\bullet$]
The explicit expression of the finite-dimensional transfer matrices $T_{i,\st}(x)$  (with $i$ as
in~\eqref{eq:KR-classification} and $\st\in \BN$) via the infinite-dimensional transfer matrices.
This allows to analytically continue $T_{i,\st}(x)$ from $\st\in \BN$ to the entire complex plane
$\st\in \BC$ and study its $\st$-symmetries.

\medskip

\item[$\bullet$]
The extra symmetry of the rational $ABCD$-type $R$-matrices in the first fundamental representations gives
rise to explicit realizations of the aforementioned infinite-dimensional irreducible highest weight
$\fg$-representations in the Fock spaces of oscillator algebras. This also immediately extends the
action of $\fg$ on these modules to that of the Yangian $Y(\fg)$.

\medskip

\item[$\bullet$]
The factorisation of the above infinite-dimensional transfer matrices into two commuting Baxter
$Q$-operators, which arise by realizing the corresponding non-degenerate Lax matrices as fusion of two
\emph{degenerate} Lax matrices. While the latter can be recovered as \emph{renormalized limits}
of the former, they also can be viewed as ``$W$-translations'' of the single one arising from the
antidominantly shifted Yangians $Y_{-\omega_i}(\fg)$ following our earlier work~\cite{fpt,Frassek:2021ogy}.

\end{enumerate}

\medskip


\subsection{Overview of type A}\label{ssec A-overview}
$\ $

For $\gl_n$-type rational spin chains, both $T$- and $Q$-operators can be built within the framework of
the quantum inverse scattering method~\cite{fad} from appropriate solutions of the \emph{RTT relation}
\begin{equation}\label{eq:RTT}
  R_{12}(z-w)L_1(z)L_2(w)=L_2(w)L_1(z)R_{12}(z-w)
\end{equation}
involving the rational $R$-matrix
\begin{equation}\label{eq:R-matrix-A}
  R(z)=z\id_n+{\rm P} \,, \qquad
  \id_n=\sum_{i,j=1}^n\, e_{ii}\otimes e_{jj} \,, \qquad {\rm P}=\sum_{i,j=1}^n\, e_{ij}\otimes e_{ji} \,,
\end{equation}
where $\left(e_{ij}\right)_{k\ell}=\delta_{i}^{k}\delta_{j}^{\ell}$ is the standard basis of $\gl_n$,
satisfying the \emph{quantum Yang-Baxter equation}:
\begin{equation}\label{eq:YBe}
  R_{12}(z-w)R_{13}(z)R_{23}(w) = R_{23}(w)R_{13}(z)R_{12}(z-w) \,.
\end{equation}
More precisely, the operators $\sT(z),\sQ(z)\in \End(\BC^n)^{\otimes N}$ (with $N$ being the length of the spin chain)
are then defined as traces in the auxiliary space of the \emph{monodromy matrices} (followed by a twist)
\begin{equation}\label{eq:monodromy}
  M(z)=\underbrace{L(z)\otimes \cdots \otimes L(z)}_{N}
\end{equation}
built from the \emph{Lax matrices} $L(z)$, i.e.\ solutions of~\eqref{eq:RTT}, with the product $\otimes$ denoting
the tensor product in the quantum space and the usual multiplication in the auxiliary space. The key difference
between $\sT(z)$ and $\sQ(z)$, however, is that the former correspond to \emph{non-degenerate} Lax matrices
(with a non-degenerate coefficient of the leading $z$-power) while the latter ones are built from the
\emph{degenerate} Lax matrices.

\medskip
\noindent
Using the obvious \emph{invariance} of the $R$-matrix~\eqref{eq:R-matrix-A} under the action of $GL_n$:
\begin{equation}\label{eq:R-inv}
  [R(z),G\otimes G]=0 \qquad \forall\ G\in GL_n \,,
\end{equation}
one can multiply the non-degenerate Lax matrix $L(z)$ by a $z$-independent $G\in GL_n$ to bring the coefficient of
its leading $z$-power to the identity matrix $\id_n$. Such Lax matrices are governed by the \emph{RTT Yangian} $Y^\rtt(\gl_n)$
(the explicit identification of which with the Drinfeld Yangian $Y(\gl_n)$ was carried out in~\cite{bk1}).
Here, $Y^\rtt(\gl_n)$ is the associative algebra generated by $\{t_{ij}^{(r)}\}_{1\leq i,j\leq n}^{r\geq 1}$ subject to
the defining relation~\eqref{eq:RTT}, where $L(z)$ has been replaced by $T(z)\in \End(\BC^n)\otimes Y^\rtt(\gl_n)[[z^{-1}]]$:
\begin{equation}\label{eq:A-Yangian T-matrix}
  T(z)=\sum_{i,j=1}^n e_{ij}\otimes t_{ij}(z) \qquad \mathrm{with} \qquad
  t_{ij}(z)=\delta_i^j+\sum_{r\geq 1} t^{(r)}_{ij} z^{-r} \,.
\end{equation}
In particular, any $Y^\rtt(\gl_n)$-representation $V$ gives rise to the $\gl_n$-type $\End(V)$-valued Lax matrix.

\medskip
\noindent
The special feature of type $A$ is the presence of the \emph{evaluation} homomorphism to the universal enveloping
algebra of $\gl_n$:
\begin{equation}\label{eq:A-evaluation}
  \mathrm{ev}\colon Y^\rtt(\gl_n)\to U(\gl_n)
  \quad \mathrm{given\ by} \quad t^{(r)}_{ij}\mapsto \delta_r^1 E_{ji} \,.
\end{equation}
In particular, given any $\gl_n$-module $\sigma\colon \gl_n\to \End(V)$, the matrix
$L(z)=z\id_n+\sum_{i,j=1}^n e_{ij}\sigma(E_{ji})$, with $E_{ji}$ denoting the generators of $\gl_n$,
is an $\End(V)$-valued Lax matrix, i.e.\ satisfies~\eqref{eq:RTT}. Conversely any such Lax matrix
endows $V$ with a $\gl_n$-action, since $\mathrm{ev}$ of~\eqref{eq:A-evaluation} admits a right inverse:
\begin{equation}\label{eq:embedding}
  \iota\colon U(\gl_n)\hookrightarrow Y^\rtt(\gl_n) \quad \mathrm{given\ by} \quad  E_{ij}\mapsto t^{(1)}_{ji} \,.
\end{equation}

\medskip
\noindent
In contrast, the $Q$-operators are known to arise from degenerate Lax matrices. As has been recently
realized in~\cite{fpt} (cf.~\cite{cgy} for an interpretation via the $4d$ Chern-Simons theory), under certain
\emph{antidominance} condition, the corresponding Lax matrices are governed by the \emph{antidominantly shifted RTT Yangians}
$Y^\rtt_{-\mu}(\gl_n)$. The identification of the latter with the shifted Drinfeld Yangian $Y_{-\mu}(\gl_n)$
defined as in~\cite[Appendix B]{bfnb} goes through the Gauss decomposition of the generating matrix $T(z)$ as
in the non-shifted case~\cite{bk1}, see~\cite[Theorem 2.54]{fpt}. Using the $GL_n$-invariance~\eqref{eq:R-inv}
one can further obtain other degenerate Lax matrices which do not admit a Gauss decomposition
(and thus are no longer directly related to the \emph{shifted Drinfeld Yangians}).

\medskip
\noindent
According to~\cite{Bazhanov:2010jq}, see also \cite{Tsuboi:2009ud,klt}, there is a whole family of $2^n$ $Q$-operators $\sQ_I(z)$, labelled by all subsets
$I\subseteq \{1,2,\ldots,n\}$. However, all of those can be expressed through $n$ fundamental ones $\{\sQ_i(z)\}_{i=1}^n$,
due to so-called \emph{QQ-relations}:
\begin{equation}\label{eq:A-Hirota}
  \Delta_{j,i}\cdot \sQ_{I\sqcup i\sqcup j}(z) \sQ_I(z) =
  \sQ_{I\sqcup i}(z-\tfrac{1}{2})\sQ_{I\sqcup j}(z+\tfrac{1}{2}) -
  \sQ_{I\sqcup j}(z-\tfrac{1}{2})\sQ_{I\sqcup i}(z+\tfrac{1}{2}) \,,
\end{equation}
where the scalar factor $\Delta_{j,i}=\frac{\tau_j-\tau_i}{\sqrt{\tau_i\tau_j}}$ depends only on the twist parameters
and the brackets have been suppressed for the one element sets. The two key components in the proof of~\eqref{eq:A-Hirota} are:
the fusion of the corresponding Lax matrices and the \emph{Desnanot-Jacobi-Dodgson-Sylvester theorem}
from linear algebra. In particular, fusing all $n$ fundamental $Q$-operators, one obtains the transfer matrix
$\sT^+_\lambda(z)$ associated with the non-degenerate linear Lax matrix corresponding to a $\gl_n$ Verma~module:
\begin{equation}\label{eq:Verma-through-Q}
  \Delta_{\{1,\ldots,n\}}\cdot \sT^+_{\lambda}(z) =
  \sQ_1(z+\lambda'_1) \sQ_2(z+\lambda'_2) \cdots \sQ_n(z+\lambda'_n) \,,
\end{equation}
where $\lambda\in \BC^n$ is the highest weight of the Verma module,
$\lambda'_i$ are the components of the shifted weight $\lambda'=(\lambda'_1,\ldots,\lambda'_n)=\lambda+\rho$ with
$\rho=\left(\frac{n-1}{2},\frac{n-3}{2},\ldots,\frac{1-n}{2}\right)$, and
$\Delta_{\{1,\ldots,n\}}=\prod_{1\leq i<j\leq n} \Delta_{i,j}$.

\medskip
\noindent
Finally, for a dominant integral $\gl_n$-weight $\lambda$, the finite-dimensional transfer matrix $\sT_\lambda(z)$
can be expressed as an alternating sum of the infinite-dimensional $\sT^+_\mu(z)$, due to so-called \emph{BGG-relations}:
\begin{equation}\label{eq:transfer-via-BGG}
  \sT_\lambda(z)=\sum_{\sigma\in S_n} (-1)^{l(\sigma)} \sT^+_{\sigma(\lambda+\rho)-\rho}(z) \,.
\end{equation}
Combining~\eqref{eq:transfer-via-BGG} with~\eqref{eq:Verma-through-Q}, one finally gets the
\emph{determinant formula} for type $A$ transfer matrices:
\begin{equation}\label{eq:det-intro}
  \Delta_{\{1,\ldots,n\}}\cdot \sT_\lambda(z) =
  \det \left\Vert \sQ_i(z+\lambda_j') \right\Vert_{1\leq i,j\leq n} \,.
\end{equation}

\medskip


\subsection{BGG resolutions}\label{ssec BGG resol}
$\ $

Formula~\eqref{eq:transfer-via-BGG} follows directly from the famous
\emph{Bernstein-Gelfand-Gelfand resolution}~\cite{bgg} of the finite-dimensional
$\fg$-module $L_\lambda$ by means of the infinite-dimensional Verma modules $M_{\mu}$:
\begin{equation}\label{eq:general-BGG}
  0\to M_{w_0\, \cdot \lambda}\to \cdots\, \to \bigoplus_{w\in W}^{l(w)=2} M_{w\, \cdot \lambda}\ \to
  \bigoplus_{w\in W}^{l(w)=1} M_{w\, \cdot \lambda}\to M_{\lambda}\to L_{\lambda}\to 0 \,.
\end{equation}
Here, $W$ is the Weyl group of $\fg$, equipped with the length function $l\colon W\to \BZ_{\geq 0}$ as an abstract
Coxeter group. Furthermore, we use the \textbf{dot action} of $W$ on the space of weights, defined via:
\begin{equation}\label{eq:dot action}
  w\, \cdot \lambda = w(\lambda+\rho)-\rho \qquad \forall\, w\in W \,,\, \lambda\in\fh^* \,,
\end{equation}
with the weight $\rho\in \fh^*$ defined in the standard way:
\begin{equation}\label{eq:rho-weight}
  \rho=\sfrac{1}{2}\sum_{\alpha\in \Delta^+} \alpha \,,
\end{equation}
where $\Delta^+$ denotes the set of positive roots of $\fg$.
The resolution~\eqref{eq:general-BGG} involves the total of $|W|$ Verma modules and has a length equal to $|\Delta^+|$,
with the leftmost nontrivial term corresponding~to
\begin{equation}\label{eq:longest element}
  w_0 =  \mathrm{the\ longest\ element\ of}\ W \,.
\end{equation}

\medskip
\noindent
The BGG resolution~\eqref{eq:general-BGG} can be thought of as a \emph{categorification} of the
\emph{Weyl character formula} (expressing the character of the finite-dimensional $\fg$-module $L_\lambda$
via those of Verma modules $M_{\mu}$):
\begin{equation}\label{eq:Weyl character}
  \ch_{L_\lambda} = \sum_{w\in W} (-1)^{l(w)}
  \frac{e^{w(\lambda+\rho)-\rho}}{\prod_{\alpha\in \Delta^+} (1-e^{-\alpha})} =
  \sum_{w\in W} (-1)^{l(w)} \ch_{M_{w\, \cdot \lambda}} \,,
\end{equation}
the \emph{character limit} of~\eqref{eq:general-BGG}.
For $\lambda=0$, the formula~\eqref{eq:Weyl character} recovers the \emph{Weyl denominator formula}:
\begin{equation}\label{eq:Weyl denominator}
  \sum_{w\in W} (-1)^{l(w)} e^{w(\rho)-\rho}\ = \prod_{\alpha\in \Delta^+} (1-e^{-\alpha}) \,.
\end{equation}

\medskip


\subsection{Generalization to other classical types}\label{ssec BCD-overview}
$\ $

Let us now consider the rational spin chains of types $B_r,C_r,D_r$. The corresponding rational $R$-matrices $R(z)$
were first discovered in~\cite{zz} and take the following form (we use the conventions in mathematics literature,
related to the original formulas of~\cite{zz} via similarity transformations):
\begin{equation}\label{eq:BCD-Rmatrix}
  R(z)=z(z+\kappa)\id_{\NK} + (z+\kappa){\rm P} - z{\rm Q}
\end{equation}
with $\kappa=r+1$ for $\ssp_{2r}$, $\kappa=r-1=\sfrac{\NK}{2}-1$ for $\sso_{2r}$, $\kappa=r-\sfrac{1}{2}=\sfrac{\NK}{2}-1$ for $\sso_{2r+1}$,
where $\NK=2r$ for $\ssp_{2r}$ and $\sso_{2r}$ while $\NK=2r+1$ for $\sso_{2r+1}$,
and the linear operators ${\rm P}, {\rm Q}\in \End(\BC^{\NK})$ defined by:
\begin{equation}\label{eq:PQ}
  {\rm P}=\sum_{i,j=1}^{\NK} e_{ij}\otimes e_{ji} \,, \qquad
  {\rm Q}=\sum_{i,j=1}^{\NK} \epsilon_i\epsilon_j \, e_{ij}\otimes e_{i'j'} \,,
\end{equation}
where
\begin{equation}\label{eq:prime-index}
  i'=\NK+1-i  \qquad \mathrm{for} \qquad 1\leq i\leq \NK
\end{equation}
and
$\epsilon_1=\cdots=\epsilon_\NK=1$ for $\sso_{\NK}$ while
$\epsilon_1=\cdots=\epsilon_r=1,\ \epsilon_{r+1}=\cdots=\epsilon_{2r}=-1$ for $\ssp_{2r}$.

\medskip
\noindent
Similarly to type $A$, the non-degenerate Lax matrices $L(z)$ of types $BCD$ with the coefficient of the leading $z$-power
equal to $\id_{\NK}$ are governed by the corresponding \emph{extended RTT Yangians} $X^\rtt(\fg)$ (whose explicit relation
to the Drinfeld Yangian $Y(\fg)$ was obtained quite recently in~\cite{jlm, grw}).

\medskip
\noindent
However, the key difference from type $A$ is in the absence of the evaluation homomorphism~\eqref{eq:A-evaluation}.
In particular, the action of $\fg$ on $L_\lambda$ in general cannot be extended to that of $Y(\fg)$,
cf.~\eqref{eq:embedding}. Nonetheless, extending the earlier works~\cite{sw,r}, a certain family of linear and quadratic
oscillator-type non-degenerate Lax matrices $\CL(z)$ is known~\cite{f,Frassek:2021ogy,Karakhanyan:2020jtq}.
These $\CL(z)$ depend on a parameter $\st\in \BC$ and give rise to an action of $X^\rtt(\fg)$ on the parabolic Verma
$\fg$-modules.
That way we also obtain an action of $Y(\fg)$ on $L_{t\omega_i}$ for the fundamental weights $\omega_i$ classified
in~\eqref{eq:KR-classification} and all $t\in \BN$.

\medskip
\noindent
Given a simple Lie algebra $\fg$ of rank $r$, one may ask for which indices $i\in \{1,\ldots,r\}$ do the finite-dimensional
irreducible $\fg$-modules $L_{t\omega_i}$ (where $\omega_i$ denotes the $i$-th fundamental weight of $\fg$) admit a compatible
(through the embedding $U(\fg)\hookrightarrow Y(\fg)$, cf.~\eqref{eq:embedding}) action of $Y(\fg)$ for all $t\in \BN$.
In the classical $ABCD$ types, the answer to this question has been provided long ago in~\cite[\S2]{kr}:
\begin{equation}\label{eq:KR-classification}
  \Big\{ 1\leq i\leq r \,\Big|\, \fg \curvearrowright L_{\st\omega_i} \ \mathrm{extends\ to}\
         Y(\fg) \curvearrowright L_{\st\omega_i}\ \forall\, \st\in \BN \Big\} =
  \begin{cases}
     i=1,\ldots,r & \text{for } A_r \\
     i=1  & \text{for } B_r \\
     i=r  & \text{for } C_r \\
     i=1,r-1,r & \text{for } D_r
  \end{cases} \,.
\end{equation}
Let us emphasize that the above Lax matrix approach provides a \underline{constructive existence proof} for all these cases.
Furthermore, we note that the corresponding values of the index $i$ can be characterized by either of the following two equivalent
conditions:
\begin{enumerate}

\item [$\bullet$]
  the label of the vertex $i$ in the Dynkin diagram of $\fg$ is equal to $1$

\item [$\bullet$]
  the $i$-th fundamental coweight $\omega^\vee_i$ is minuscule (minuscule weight of the Langlands dual~$\fg^L$)

\end{enumerate}
It has been recently shown in~\cite[Appendix B]{cgy} by cohomological arguments that these conditions
indeed guarantee the positive answer to the above question in all types. This also adds Lax matrices in
the exceptional types $E_6, E_7$, but we presently do not have explicit formulas for those.

\medskip
\noindent
Yet another obstacle to generalize the results of Section~\ref{ssec A-overview} to other types is
that even if $L_\lambda$ extends to a module over $Y(\fg)$, it may not be the case for the Verma modules
$M_{w\, \cdot\lambda}$ featuring in the BGG resolution~\eqref{eq:general-BGG}.
However, as explained above, the Lax matrix approach does provide a natural $Y(\fg)$-action on
the corresponding parabolic Verma modules. We shall now discuss \underline{new} BGG-type resolutions of
$L_\lambda$ which involve only those parabolic Verma modules and their ``$W$-translations''.

\medskip


\subsection{New BGG-type resolutions}\label{ssec truncated-BGG}
$\ $

To state the main results of this subsection, let us first introduce some more notation.
Consider the root decomposition $\fg=\fh\oplus \bigoplus_{\alpha\in \Delta} \BC e_{\alpha}$.
Let  $\{\alpha_i\}_{i=1}^r\subset \Delta^+$ be the positive simple roots of $\fg$.
Given a subset $S\subseteq \{1,\ldots,r\}$, one defines the standard parabolic Lie algebra $\fp_S\subseteq \fg$ via:
\begin{equation}\label{eq:parabolic subalg}
  \fp_S = \fh\oplus \bigoplus_{\alpha\in \Delta^+} \BC e_{\alpha} \oplus
  \bigoplus_{\alpha\in \Delta^+_{S}} \BC e_{-\alpha} \qquad \mathrm{with} \qquad
  \Delta^+_S=\Delta^+\cap \bigoplus_{i\in S}\, \BZ\alpha_i \,.
\end{equation}
Let us note that $\fp_{\{1,\ldots,r\}}=\fg$, while $\fp_{\emptyset}$ coincides with the Borel subalgebra
$\fb=\fh\oplus \bigoplus_{\alpha\in \Delta^+} \BC e_{\alpha}$. It is well-known that all parabolic subalgebras
$\fp$ satisfying $\fb\subseteq \fp\subseteq \fg$ are necessarily of that form. Such $\fp_S$ further decomposes
into the semidirect product:
\begin{equation}\label{eq:semidirect}
  \fp_S=\fl\ltimes \fu  \qquad \mathrm{with} \qquad
  \fl=\fh \oplus \bigoplus_{\alpha\in \Delta^+_S} \BC e_\alpha \oplus \bigoplus_{\alpha\in \Delta^+_S} \BC e_{-\alpha}
  \qquad \mathrm{and} \qquad
  \fu\, =\bigoplus_{\alpha\in \Delta^+ \setminus \Delta^+_S} \BC e_\alpha \,,
\end{equation}
of the semisimple part $\fl$ (the \emph{Levi subalgebra}) and the nilpotent part $\fu$ (the \emph{nilpotent radical}).
Then
\begin{equation}\label{eq:Levi-roots}
  \Delta_\fl=\Delta^+_\fl\sqcup (-\Delta^+_\fl) \quad \mathrm{with} \quad \Delta^+_{\fl}=\Delta^+_S
\end{equation}
is the root system of $\fl$, and
the subgroup $W_{\fl}$ of $W$ generated by the simple reflections $\{s_{\alpha_j}\}_{j\in S}$ is the Weyl group of $\fl$.
In analogy with~\eqref{eq:rho-weight}, we also define
\begin{equation}\label{eq:rho-Levi}
  \rho_\fl=\sfrac{1}{2}\sum_{\alpha\in \Delta^+_\fl} \alpha \,.
\end{equation}

\medskip
\noindent
The finite-dimensional irreducible $\fl$-modules are indexed by the dominant integral weights of $\fl$:
\begin{equation}\label{eq:Levi dominant}
  P^+_\fl=\left\{\lambda\in \fh^* \, |\, \lambda(h_{\alpha_j})\in \BN\ \ \forall j\in S \right\} \,.
\end{equation}
For $\lambda\in P^+_{\fl}$, let $L^\fl_\lambda$ be the corresponding $\fl$-module.
One makes it into a $\fp_S$-module by letting the nilpotent radical $\fu\subset \fp_S$ act trivially.
Then, the \textbf{parabolic Verma module} $M^{\fp_S}_\lambda$ is defined via:
\begin{equation}\label{eq:parabolic def}
  M^{\fp_S}_{\lambda}=\mathrm{Ind}^{\fg}_{\fp_S}(L^{\fl}_\lambda) = U(\fg)\otimes_{U(\fp_S)} L^{\fl}_\lambda \,.
\end{equation}
Let $P^+_{\fg}$ be the set of dominant integral weights of $\fg$ (apply~\eqref{eq:Levi dominant} to $S=\{1,\ldots,r\}$),
and define
\begin{equation}\label{eq:Levi-trivial}
   P^+_{\fg/\fl} = \left\{\lambda\in P^+_{\fg} \, |\, \lambda(h_{\alpha_j})=0\ \ \forall j\in S \right\} \,,
\end{equation}
that is, $P^+_{\fg/\fl}$ is the set of dominant integral weights $\lambda$ of $\fg$ such that $\dim(L^{\fl}_\lambda)=1$.

\medskip
\noindent
Let us now state our \underline{key mathematical construction}: for any parabolic $\fp\subset\fg$ and
$\lambda\in P^+_{\fg/\fl}$~\eqref{eq:Levi-trivial}, the irreducible finite-dimensional $\fg$-module $L_{\lambda}$
admits the following \emph{truncated BGG resolution}:
\begin{equation}\label{eq:conjectured resolution 1}
  0\to M'_{{}^{\fl,0}w\, \cdot \lambda}\to
  \cdots\, \to \bigoplus_{w\in {}^{\fl}W}^{l(w)=2} M'_{w\, \cdot \lambda}\ \to
  \bigoplus_{w\in {}^{\fl}W}^{l(w)=1} M'_{w\, \cdot \lambda}\to
  M'_{\lambda}\to L_{\lambda}\to 0
\end{equation}
with each term admitting further a resolution by the usual Verma modules:
\begin{equation}\label{eq:conjectured resolution 2}
  0\to M_{ww_{\fl,0}\, \cdot \lambda} \to
  \cdots\, \to \bigoplus_{v\in W_{\fl}}^{l(v)=2} M_{wv\, \cdot \lambda}\ \to
  \bigoplus_{v\in W_{\fl}}^{l(v)=1} M_{wv\, \cdot \lambda}\to M_{w\, \cdot \lambda}\to
  M'_{w\, \cdot \lambda}\to 0 \,.
\end{equation}
Here, ${}^{\fl}W$ is defined via:
\begin{equation}\label{eq:shortest left}
  {}^{\fl}W=\{w\in W \, |\, w(\Delta^+_\fl)\subseteq \Delta^+\} \,.
\end{equation}
According to~\cite[\S5.13]{kos}, the set ${}^{\fl}W$ can be equivalently described as:
\begin{equation}\label{eq:shortest left equiv}
  {}^{\fl}W=\{ \mathrm{shortest\ representatives\ of\ the\ left\ cosets}\ W/W_{\fl} \} \,,
\end{equation}
and each element $w\in W$ admits a unique factorisation:
\begin{equation}\label{eq:left factorization}
  w={}^{\fl}w\, w_{\fl} \qquad \mathrm{with} \qquad w_{\fl}\in W_{\fl} \, ,\, {}^{\fl}w\in {}^{\fl}W \,.
\end{equation}
In particular, the longest elements ${}^{\fl,0}w\in {}^{\fl}W$ and $w_{\fl,0}\in W_{\fl}$
featuring in~(\ref{eq:conjectured resolution 1}) and~(\ref{eq:conjectured resolution 2}) arise from
the decomposition~\eqref{eq:left factorization} applied to the longest element $w_0\in W$
of~\eqref{eq:longest element}:
\begin{equation}\label{eq:longest components}
  w_0 = {}^{\fl,0}w\, w_{\fl,0 } \,.
\end{equation}

\noindent
The above modules $\{M'_{w\, \cdot \lambda}\}_{w\in {}^{\fl}W}$ are defined as explicit quotients of
the Verma modules $M_{w\, \cdot \lambda}$:
\begin{equation}\label{eq:hard modules}
   M'_{w\, \cdot \lambda} = M_{w\, \cdot \lambda}/M_{w\, \cdot \lambda}^{\mathrm{sing}}
\end{equation}
with $M_{w\, \cdot \lambda}^{\mathrm{sing}}$ being the $\fg$-submodule of $M_{w\, \cdot \lambda}$
generated by the singular vectors of weights
\begin{equation}\label{eq:critical vectors}
  s_{w(\alpha)}\left(w(\lambda+\rho)\right)-\rho   \,, \qquad \alpha\in \Delta^+_{\fl} \,.
\end{equation}
Here, $s_{w(\alpha)}$ denotes a reflection in the positive root $w(\alpha)$, see~\eqref{eq:shortest left},
while the existence (note that uniqueness is standard) of such singular vectors is guaranteed by
$\left( w(\lambda+\rho), w(\alpha) \right) = (\rho,\alpha)\in \BZ_{>0}$.

\medskip
\noindent
One can think of~(\ref{eq:conjectured resolution 1},~\ref{eq:conjectured resolution 2}) as a \emph{categorification}
of the following character equality expressing the character of the finite-dimensional $\fg$-module $L_\lambda$
via those of the modules $\{M'_{w\, \cdot \lambda}\}$, cf.~\eqref{eq:Weyl character}:
\begin{equation}\label{eq:character identity}
  \ch_{L_{\lambda}}=
  \sum_{w\in {}^{\fl}W}
    \frac{(-1)^{l(w)} e^{w(\lambda+\rho)-\rho}}
         {\prod_{\alpha\in\Delta^+\backslash w(\Delta^+_{\fl})} (1-e^{-\alpha})} \, =
  \sum_{w\in {}^{\fl}W}
    (-1)^{l(w)} \ch_{M'_{w\, \cdot \lambda}} \,.
\end{equation}

\medskip

\begin{MainTh}\label{mthm 1}
For $\lambda\in P^+_{\fg/\fl}$, the irreducible finite-dimensional $\fg$-module $L_\lambda$
has a finite length resolution~\eqref{eq:conjectured resolution 1}, with each term admitting a finite length
resolution~\eqref{eq:conjectured resolution 2} by Verma modules.
\end{MainTh}

\medskip
\noindent
The resolutions~(\ref{eq:conjectured resolution 1},~\ref{eq:conjectured resolution 2}) are reminiscent of
the well-known Lepowsky's \emph{parabolic BGG resolutions}~\cite{l} of any irreducible finite-dimensional
$\fg$-module $L_\lambda$ by parabolic Verma modules~$\{M^{\fp}_\mu\}_{\mu\in P^+_{\fl}}$:
\begin{equation}\label{eq:Lepowsky resolution 1}
  0\to \cdots\, \to \bigoplus_{w\in W^{\fl}}^{l(w)=2} M^{\fp}_{w\, \cdot \lambda}\ \to
  \bigoplus_{w\in W^{\fl}}^{l(w)=1} M^{\fp}_{w\, \cdot \lambda}\to M^{\fp}_{\lambda}\to L_{\lambda}\to 0
\end{equation}
with each term admitting further a resolution by the usual Verma modules:
\begin{equation}\label{eq:Lepowsky resolution 2}
  0\to M_{w_{\fl,0}w\,\cdot \lambda} \to \cdots\, \to \bigoplus_{v\in W_{\fl}}^{l(v)=2} M_{vw\, \cdot \lambda}\ \to
  \bigoplus_{v\in W_{\fl}}^{l(v)=1} M_{vw\, \cdot \lambda}\to M_{w\, \cdot \lambda}\to
  M^{\fp}_{w\, \cdot \lambda}\to 0 \,.
\end{equation}
Here, the indexing subset $W^{\fl}$ of $W$ is defined as:
\begin{equation}\label{eq:shortest right}
  W^{\fl} = \{w\in W \, |\, \Delta^+_\fl\subseteq w(\Delta^+)\} =
  \{ \mathrm{shortest\ representatives\ of\ the\ right\ cosets}\ W_{\fl}\backslash W \} \,.
\end{equation}
Note that $W^{\fl}=\{w^{-1} \, |\, w\in {}^{\fl}W\}$ and each element $w\in W$ admits a unique factorisation:
\begin{equation}\label{eq:right factorization}
  w=w_{\fl} w^{\fl} \qquad \mathrm{with} \qquad w_{\fl}\in W_{\fl} \,,\, w^{\fl}\in W^{\fl} \,.
\end{equation}

\medskip
\noindent
While the BGG resolution~\eqref{eq:general-BGG} and consequently the Lepowsky-BGG resolution~\eqref{eq:Lepowsky resolution 1}
were originally constructed algebraically, we are presently not aware of the algebraic construction of~\eqref{eq:conjectured resolution 1}:
the key difficulty is that non-simple reflections are involved in the definition of $M'_{w\, \cdot \lambda}$. Instead,
we shall construct~\eqref{eq:conjectured resolution 1}  by interpreting its \emph{restricted dual} as a \emph{Cousin complex}
of relative local cohomology groups of the corresponding line bundle on the partial flag variety $X=G/P$ stratified by
$B_{-}$-orbits. Here, $B\subset P$ are the Borel and parabolic subgroups of the Lie group $G$ with
$\Lie(B)=\fb$, $\Lie(P)=\fp$, $\Lie(G)=\fg$, and $B_{-}$ is the opposite Borel subgroup of $G$.
A similar geometric interpretation of~\eqref{eq:general-BGG} goes back to~\cite{k,b}, while the analogous treatment
of the parabolic
BGG resolutions~\eqref{eq:Lepowsky resolution 1} was presented in~\cite{mr}
by considering instead the complete flag variety $Y=G/B$ stratified by $P$-orbits.
However, let us point out that~\cite{mr} was not self-contained as it used some algebraic properties
established in~\cite{l} (and also contained two substantial inaccuracies which we fix in our Remark~\ref{rem:MR gaps}).

\medskip


\subsection{BGG-relations for transfer matrices of classical types}\label{ssec analytic continuation}
$\ $

Despite the aforementioned similarity between our resolution~(\ref{eq:conjectured resolution 1}) and the
Lepowsky-BGG resolution~\eqref{eq:Lepowsky resolution 1}, they have several major differences. First of all,
the latter is constructed for any choice of the parabolic subalgebra $\fp\subset \fg$ independent of
$\lambda\in P^+_{\fg}$ and consists only of the parabolic Verma modules $\{M^{\fp}_{\mu}\}_{\mu\in P^+_{\fl}}$
(though $\dim(L^{\fl}_{\mu})>1$ in general). But an even more striking  difference is
the fact that our modules $M'_{w\, \cdot \lambda}$ admit an \underline{analytic continuation} in $\lambda$
to the domain
\begin{equation}\label{eq:Levi-trivial non-positive}
  P_{\fg/\fl} = \left\{\lambda\in \fh^* \, |\, \lambda(h_{\alpha_j})=0\ \ \forall j\in S \right\} \,.
\end{equation}

\medskip
\noindent
Indeed, one may apply the construction~(\ref{eq:hard modules},~\ref{eq:critical vectors}) for any
$\lambda\in P_{\fg/\fl}$ to define $\fg$-modules $M'_{w\, \cdot \lambda}$ of the highest weight $w\, \cdot \lambda$,
see~\eqref{eq:dot action}. Moreover, the resulting highest-weight $\fg$-modules
$\{M'_{w\, \cdot \lambda}\}_{w\in {}^{\fl}W}$ are \underline{generically irreducible} for $\lambda\in P_{\fg/\fl}$
(in particular, they are irreducible for $\lambda\in -P^+_{\fg/\fl}\subset P_{\fg/\fl}$) as follows from the classical
description~\cite{j,s} of the weights of singular vectors in the Verma modules.

\medskip
\noindent
We note that the module $M'_{\lambda}=M'_{\mathrm{id}\cdot \lambda}$ coincides with the parabolic Verma
module~$M^{\fp}_{\lambda}$, and all other modules $\{M'_{w\, \cdot \lambda}\}_{w\in {}^{\fl}W}$ can be thought of as
``$W$-translations'' of $M'_{\lambda}=M^{\fp}_{\lambda}$, as follows from~\eqref{eq:character identity}.
In the particular case $\lambda=\st\omega_i$ with the label of vertex $i$ equal to $1$
(equivalent to~\eqref{eq:KR-classification} in the classical types, see Section~\ref{ssec BCD-overview}),
the corresponding parabolic Verma $\fg$-modules $M^{\fp_{\{1,\ldots,r\}\setminus \{i\}}}_{\st\omega_i},\ \st\in \BC$,
can be extended to the modules over $Y(\fg)$, cf.~\cite[Appendix B]{cgy}, and so should all other
$M'_{w\cdot \st\omega_i}$. Thus, our resolution~\eqref{eq:conjectured resolution 1} can be actually regarded
as a resolution of $Y(\fg)$-modules, giving rise to the desired \underline{\emph{BGG-relation}} expressing the
finite-dimensional transfer matrix $T_{i,\st}(z)=T_{L_{\st\omega_i}}(z)$ via:
\begin{equation}\label{eq:t-via-qq honest}
  T_{i,\st}(z)=\sum_{w\in {}^{\fl}W} (-1)^{l(w)} T_{M'_{w\cdot \st\omega_i}}(z) \,,
  \qquad \forall\, \st\in \BN \,,
\end{equation}
cf.~\eqref{eq:transfer-via-BGG}. The length $N=0$ case of~\eqref{eq:t-via-qq honest} recovers back
the character formula~\eqref{eq:character identity}.

\medskip
\noindent
To make this even more feasible in the classical types, let us recall that in each case
of~\eqref{eq:KR-classification}, the resulting action of the Yangian $Y(\fg)$ on the parabolic Verma module
$M^{\fp_{\{1,\ldots,r\}\setminus \{i\}}}_{\st\omega_i}$ is given by an explicit oscillator-type Lax matrix,
as has been emphasized in Section~\ref{ssec BCD-overview}.
Utilizing further the Weyl group symmetry of the rational $R$-matrices (\ref{eq:R-matrix-A},~\ref{eq:BCD-Rmatrix})
combined with the appropriate particle-hole automorphisms of the corresponding oscillator algebra $\CA$
(in order for our $\fg$-modules to be in the category $\CO$ of~\cite{bgg}) one obtains a family of Lax matrices
parametrized by
the same set ${}^{\fl}W$. This gives rise to $Y(\fg)$-modules $\{M^+_{w\cdot \st\omega_i}\}_{w\in {}^{\fl}W}$ explicitly
realized in the Fock representation $\Fock$ of the algebras $\CA$. As $\fg$-modules, they have the same characters as
$\{M'_{w\cdot \st\omega_i}\}_{w\in {}^{\fl}W}$ and furthermore the Fock vacuum $|0\rangle \in \Fock$ is a
highest weight vector of the highest weight $w\cdot \st\omega_i$. Thus, if $M'_{w\cdot\, \st\omega_i}$ is irreducible,
then $M^+_{w\cdot \st\omega_i}\simeq M'_{w\cdot \st\omega_i}$ and the corresponding transfer matrices
$T^+_{w,\st\omega_i}(z)=T_{M^+_{w\cdot\st\omega_i}}(z)$ and $T_{M'_{w\cdot \st\omega_i}}(z)$ coincide.
Combining this with the above observation that $M'_{w\cdot \st\omega_i}$ are generically irreducible for
$\st\in \BC$ and the fact that both transfer matrices depend continuously on the parameter $\st\in \BC$,
we obtain the uniform equality of the corresponding transfer matrices:
\begin{equation}\label{eq:T-coincide}
  T^+_{w,\st\omega_i}(z) = T_{M'_{w\cdot \st\omega_i}}(z) \qquad \forall\, \st\in \BC \,,\, w\in {}^{\fl}W  \,,
\end{equation}
even though for $w\ne \mathrm{id}$ the $\fg$-modules $M^+_{w\cdot \st\omega_i}$ and $M'_{w\cdot \st\omega_i}$
may be non-isomorphic at some $\st\in \BN$ (that are exactly the values featuring in~\eqref{eq:conjectured resolution 1}).
The equality~\eqref{eq:T-coincide} allows us to recast~\eqref{eq:t-via-qq honest}~as:

\medskip

\begin{MainTh}\label{mthm 2}
For a classical rank $r$ Lie algebra $\fg$, $1\leq i\leq r$ as in~\eqref{eq:KR-classification}, $\st\in \BN$, we have:
\begin{equation}\label{eq:T-via-T+ intro}
  T_{i,\st}(z)=\sum_{w\in {}^{\fl}W} (-1)^{l(w)} T^+_{w,\st\omega_i}(z) \,,
\end{equation}
expressing the finite-dimensional transfer matrix $T_{i,\st}(z)$ as an alternating sum of the infinite-dimensional
transfer matrices  $T^+_{w,\st\omega_i}(z)$ of the $Y(\fg)$-modules $M^+_{w\cdot \st\omega_i}$ explicitly
realized in the Fock $\CA$-module $\Fock$.
\end{MainTh}

\medskip
\noindent
We refer the reader to our Theorems~\ref{thm:Main-A},~\ref{thm:Main-C},~\ref{thm:Main-D},~\ref{thm:Main-BD}
for a case-by-case treatment presented in the way that highlights the combinatorial description
of ${}^{\fl}W$~(\ref{eq:shortest left},~\ref{eq:shortest left equiv}) in each case of~\eqref{eq:KR-classification}.

\medskip
\noindent
In type $D$, this establishes the key assumption from~\cite{ffk} (the joint work of the first-named author with
G.~Ferrando and V.~Kazakov) essential for the  study of the entire $QQ$-system in~\emph{loc.cit}.

\medskip
\noindent
Let us note that the \emph{BGG-relation}~\eqref{eq:T-via-T+ intro} allows to \underline{analytically continue} the transfer
matrices $T_{i,\st}(z)$ of the finite-dimensional representations $L_{\st\omega_i}$, $\st\in\BN$, to the entire complex
plane $\st\in \BC$. With this in mind, we establish the $\st$-symmetries of the resulting family of endomorphisms
$T_{i,\st}(z)$, see Propositions~\ref{prop:t-symmetry C},~\ref{prop:t-symmetry D},~\ref{prop:t-symmetry BD},
by crucially utilizing the $QQ$-factorisation which we discuss next.

\medskip


\subsection{Factorisations}\label{ssec factorisation-overview}
$\ $

The infinite-dimensional transfer matrices $T^+_{w,\st\omega_i}(z)$ factorise into products of $Q$-operators. This factorisation can be traced back to the factorisation of oscillator-type Lax matrices used to construct the transfer matrices. In the case of $U_q(\widehat{\ssl}_2)$, such  factorisation formula was initially proposed
in \cite{Bazhanov:1998dq}, see also \cite{kt}. Here, the Lax matrix entering the definition of the infinite-dimensional transfer matrix factorises
into the product of two degenerate Lax matrices that are used to construct Baxter $Q$-operators.
The degenerate Lax matrices employed in the factorisation are solutions
of~\eqref{eq:RTT} with degenerate coefficients of the leading term, which are no longer related to quantum groups. They have a long history, going back to~\cite{izergin,kuz}. The relation to
$Q$-operators is discussed for $U_q(\widehat{\ssl}_2)$ and $Y(\ssl_2)$ in~\cite{antonov,Bazhanov:1998dq,rossi, Korff,baz},
and for higher rank
cases in ~\cite{Bazhanov:2001xm,Bazhanov:2008yc,bgknr,Bazhanov:2010jq,Tsuboi:2019vvv}, while cases
beyond $A$-type were first found in~\cite{f,cgy,Frassek:2021ogy}.

\medskip
\noindent
By now, the role of the degenerate Lax matrices, arising from the shifted Yangians and viewed as certain
\emph{normalized limits} of the non-degenerate ones (as some of the representation labels tend to infinity),
is well understood. However the actual factorisation that relates Lax matrices of different kinds remains yet
to be understood. For a discussion of $A$-type we refer the reader to \cite{fp} (which also indicates an intriguing
relation to the cluster structures). For completeness of our exposition, we also refer the reader to the slightly
different approach \cite{Derkachov:2006fw,DerkachovFac,Derkachov:2010qe} going back to \cite{gp}.

\medskip
\noindent
Let us summarize the third main result of this paper that factorises the infinite-dimensional transfer matrices
$T^+_{w,\st\omega_i}(z)$ from the previous subsection into the product of two commuting $Q$-operators
(in type $D$, this was first observed in~\cite{f,ffk}, while an interpretation of this factorisation
in terms of the $4d$ Chern-Simons theory for general types has appeared very recently in~\cite[\S16]{cgy}):

\medskip

\begin{MainTh}\label{mthm 3}
For a classical rank $r$ Lie algebra $\fg$, an index $1\leq i\leq r$ as in~\eqref{eq:KR-classification},
a scalar $\st\in \BC$, and the element $w\in {}^{\fl}W$ as in~(\ref{eq:shortest left},~\ref{eq:shortest left equiv})
for the standard parabolic $\fp_{\{1,\ldots,r\}\setminus \{i\}}$, we have:
\begin{equation}\label{eq:T-via-QQ intro}
  T^+_{w,\st\omega_i}(z)=\ch^+_{w,\st\omega_i}\cdot\, Q^+_{w,\st\omega_i}(z) Q^-_{w,\st\omega_i}(z) \,,
\end{equation}
with the scalar factor $\ch^+_{w,\st\omega_i}$ arising as a trace of the $z$-independent twist
and the $Q$-operators $Q^\pm_{w,\st\omega_i}(z)$ arising in a similar fashion to $T^+_{w,\st\omega_i}(z)$
but rather from the degenerate Lax matrices.
\end{MainTh}

\medskip
\noindent
We refer the reader to our
formulas~(\ref{eq:T-via-QQ A-type},~\ref{eq:T-via-QQ C-type},~\ref{eq:T-via-QQ D-type},~\ref{eq:T-via-QQ BD-type})
for a case-by-case treatment presented in the way that highlights the aforementioned combinatorial description
of the set ${}^{\fl}W$ in each case.

\medskip
\noindent
Combining our Main Theorems~\ref{mthm 2} and~\ref{mthm 3}, we thus obtain expressions for the finite-dimensional
transfer matrices $T_{i,\st}(z)$ (with the index $i$ from~\eqref{eq:KR-classification} and $\st\in \BN$) in terms
of the above $Q$-operators, see Propositions~\ref{prop:Main-A-recasted},~\ref{prop:main-C factorized},
~\ref{prop:main-D factorized},~\ref{prop:main-BD factorized}.

\medskip


\subsection{Transfer matrices from the universal R-matrix}\label{ssec R-matrix}
$\ $

We conclude our Introduction with a more general, but less explicit, construction of
the transfer matrices of rational spin chains (trigonometric version of which is much better understood by now).
Let $\fg$ be a semisimple Lie algebra with a non-degenerate invariant form $(\cdot,\cdot)$ and $Y_\hbar(\fg)$ denote
the Drinfeld Yangian, which is a Hopf $\BC[\hbar]$-algebra deforming the current algebra
$Y_{\hbar=0}(\fg)\simeq U(\fg[u])$. As the specializations $Y_{\hbar=a}(\fg)\simeq Y_{\hbar=b}(\fg)$
are canonically isomorphic for $a,b\in \BC^\times$, we shall omit $\hbar$-dependence by rather considering
$Y(\fg)=Y_{\hbar=1}(\fg)$. The latter is a Hopf algebra with a coproduct
\begin{equation}\label{eq:Yangian coproduct}
  \Delta\colon Y(\fg)\to Y(\fg)\otimes Y(\fg) \,,
\end{equation}
and admits a one-parameter group of Hopf algebra automorphisms $\{\tau_a\}_{a\in \BC}$, quantizing
the shift automorphisms $\{\bar{\tau}_a\}_{a\in \BC}$ of $U(\fg[u])$ given by $Xu^k\mapsto X(u+a)^k$
for $X\in \fg$ and $k\in \BN$, which may be further viewed (upon replacing $a\in \BC$ with a formal variable $z$)
as an algebra homomorphism
\begin{equation}\label{eq:Yangian shift}
  \tau_z\colon Y(\fg)\to Y(\fg)[z] \,.
\end{equation}
Let $\Omega_\fg\in \fg\otimes \fg$ be the Casimir tensor corresponding to $(\cdot,\cdot)$,
and $\Delta^{\opp}$ be the opposite coproduct.

\medskip

\begin{Thm-non}[{\cite[Theorem 3]{d}}]
There is a unique formal series
\begin{equation}\label{eq:universal-R}
  \CR(z)=1+\sum_{k=1}^\infty \CR_k z^{-k} \in Y(\fg)\otimes Y(\fg) [[z^{-1}]]
\end{equation}
satisfying the following relations:
\begin{equation}
\begin{split}
  \mathrm{intertwining\ identity}\colon \quad
  & (\mathrm{id}\otimes \tau_z) \Delta^{\opp}(y) =
    \CR(z)^{-1} \cdot (\mathrm{id}\otimes \tau_z)\Delta(y) \cdot \CR(z)
    \quad \forall\, y\in Y(\fg) \,, \\
  \mathrm{cabling\ identity}\colon \quad
  &    (\mathrm{id}\otimes \Delta)(\CR(z))=\CR_{12}(z) \CR_{13}(z) \,.
\end{split}
\end{equation}
It also satisfies the quantum Yang-Baxter equation~\eqref{eq:YBe}
and is called \textbf{the universal $R$-matrix}.
Moreover, $\CR(z)$ also satisfies the following identities:
\begin{equation}\label{eq:R-properties}
  \CR(z)=1+ \Omega_\fg\cdot z^{-1}+O(z^{-2}) \,, \quad
  \CR(z)^{-1}=\CR_{21}(-z) \,, \quad
  (\tau_a\otimes \tau_b) \CR(z)=\CR(z+b-a) \,.
\end{equation}
\end{Thm-non}

\medskip
\noindent
For any two representations $\rho^V\colon Y(\fg)\to \End(V)$ and $\rho^W\colon Y(\fg)\to \End(W)$,
consider the evaluation of $\CR(z)$:
\begin{equation}\label{specialized-R}
  R^{VW}(z) = (\rho^V\otimes \rho^W)(\CR(z))\in \End(V)\otimes \End(W)[[z^{-1}]] \,.
\end{equation}
For $W=V$, $R^{VV}(z)$ is thus a solution of the quantum Yang-Baxter equation~\eqref{eq:YBe}.
For irreducible finite-dimensional $V$ and $W$, $R^{VW}(z)$ is actually
a rational function in $z$, up to an overall (possibly divergent) power series $f(z)$,
see~\cite[Theorem 4]{d} (cf.~\cite[Theorem 3.10]{grw} for more~details).
This way one recovers the rational $R$-matrices~\eqref{eq:R-matrix-A} and~\eqref{eq:BCD-Rmatrix}.
Indeed, for $\fg=\ssl_n$ and $V=\BC^n$, we have $R^{VV}(z)=f(z)R(z)$ with $R(z)$ as in~\eqref{eq:R-matrix-A},
according to~\cite[Example 1]{d}. Likewise, for $\fg=\sso_{\NK}, \ssp_{\NK}$ and $V=\BC^{\NK}$,
we have $R^{VV}(z)=f(z)R(z)$ with $R(z)$ as in~\eqref{eq:BCD-Rmatrix}, due to~\cite[Proposition 3.13]{grw}.

\medskip
\noindent
For any $Y(\fg)$-module $W$ and a group-like element $x$ in an appropriate completion of $Y(\fg)$, consider
\begin{equation}\label{eq:transfer-def}
  \CT_{W,x}(z) = (\mathrm{id}\otimes \tr_W)\left((1\otimes x)\CR(z)\right)
\end{equation}
whenever the latter is well-defined. The above properties of the universal $R$-matrix $\CR(z)$ imply:
\begin{equation}\label{eq:transfer-product}
  \CT_{W_1\oplus W_2,x}(z)=\CT_{W_1,x}(z) + \CT_{W_2,x}(z) \,, \qquad
  \CT_{W_1\otimes W_2,x}(z)=\CT_{W_1,x}(z) \cdot \CT_{W_2,x}(z) \,.
\end{equation}
For a $Y(\fg)$-module $W$ and $a\in \BC$, we set $W(a)=\tau_a^*(W)$.
If further $\tau_z(x)=x$, then~\eqref{eq:R-properties} implies:
\begin{equation}\label{eq:transfer-shift}
  \CT_{W(a),x}(z) = \CT_{W,x}(z+a) \qquad \forall\, a\in \BC \,.
\end{equation}
Combining~(\ref{eq:transfer-product},~\ref{eq:transfer-shift}), we get the commutativity of the resulting
\textbf{universal transfer matrices}:
\begin{equation}\label{eq:transfer-comm}
  \CT_{W_1,x}(z+a) \cdot \CT_{W_2,x}(z+b) = \CT_{W_2,x}(z+b) \cdot \CT_{W_1,x}(z+a) \qquad \forall\, a,b\in \BC\,.
\end{equation}
Thus, for every finite-dimensional representation $\rho^V\colon Y(\fg)\to \End(V)$ we obtain a
commuting family of endomorphisms of $V$, defined by extracting the coefficients of the power series
\begin{equation}\label{eq:transfer-specialize}
  T_W(z)=\rho^V(\mathcal{T}_{W,x}(z)) \in \End(V)[[z^{-1}]] \,,
\end{equation}
as we vary the auxiliary representation $W$ (we suppress $x$ in $T_W(z)$ for simplicity of notation).

\medskip
\noindent
The explicit constructions of the present paper should arise as particular examples of this general setup
with $V=(\BC^n)^{\otimes N}$ for $\fg=\gl_n$ (resp.\ $V=(\BC^{\NK})^{\otimes N}$ for $\fg=\sso_{\NK},\ssp_{\NK}$),
the $Y(\fg)$-modules $W$ being isomorphic to the modules $\{M'_{w\,\cdot \st\omega_i}\}$ as $\fg$-modules
with $i$ from~\eqref{eq:KR-classification}, and finally $x=\exp(h)$ for a general Cartan element $h\in \fh\subset \fg$
(equivalently, $x=\prod_{i=1}^{\mathrm{rk}(\fg)} \tau_i^{\epsilon^*_{i}}$ with $\epsilon^*_i$ being a basis of $\fh$).

\medskip


\subsection{Outline of the paper}
\

The structure of the present paper is as follows. In Section~\ref{sec: truncated BGG resolutions}, we construct (by using a geometric approach) the novel BGG-type resolutions~\eqref{eq:conjectured resolution 1} and~\eqref{eq:conjectured resolution 2} on which the functional relations presented in this article are based on. The reader interested only in the functional relations can start in Section~\ref{sec: A-type Verma BGG} where we recall the well studied case of $A$-type and the standard BGG resolution. In Section~\ref{sec A-rectangular}, we apply the new BGG-type resolutions to type $A$ recovering functional relations that follow from the standard BGG resolution. Section~\ref{sec C-spinor}, Section~\ref{sec:spinD}, and Section~\ref{sec:BDff} are dedicated to the functional relations obtained from the BGG-type resolutions for type $BCD$. The factorisation of the corresponding infinite-dimensional transfer matrices is then discussed in Sections~\ref{sec linear factorizations} and~\ref{sec:facquad}. Finally, we mention some generalizations (to be presented elsewhere) of our work in Section~\ref{sec:further}.

\medskip


\subsection{Acknowledgments}
\

R.F.\ and A.T.\ are indebted to Vasily Pestun for the inspiring discussions and the collaboration on~\cite{fpt},
thus bringing them together close to the subject of the present note. R.F.\ is grateful to Gwena\"{e}l Ferrando
and Volodya Kazakov for fruitful discussions. I.K.\ and A.T.\ are extremely grateful to Boris Feigin and
Michael Finkelberg  for the enlightening discussions of the BGG-type resolutions. A.T.\ is deeply grateful to
Kevin Costello for a correspondence on the inspiring physics paper~\cite{cgy}; to David Hernandez for a correspondence
on $q$-characters; to Sachin Gautam for a discussion of the Yangian's universal $R$-matrices; to IHES (Bures-sur-Yvette)
for the hospitality and great working conditions in July 2021 when the first stages of the present work took place.
We are grateful to Zengo Tsuboi and the anonymous referee for useful suggestions and comments.

\medskip
\noindent
R.F.\ received funding of the German research foundation (Deutsche Forschungsgemeinschaft DFG)
Research Fellowships Programme $416527151$ and support of the GNFM - INdAM.
A.T.\ would like to gratefully acknowledge the support from NSF Grant
DMS-$2037602$. 
This project has received funding from the European Research Council (ERC) under the European Union's Horizon $2020$ research 
and innovation programme (QUASIFT grant agreement $677368$).

$\ $


\section{Truncated BGG resolutions as Cousin complexes}\label{sec: truncated BGG resolutions}

In this section, we construct both resolutions~\eqref{eq:conjectured resolution 1} and~\eqref{eq:conjectured resolution 2}
by interpreting their ``restricted dual'' as Cousin complexes of relative local cohomology groups, in the spirit
of~\cite{k,b,mr,ku}.

\medskip


\subsection{Cohomology with relative support and Cousin complexes}\label{ssec relative cohomology}
$\ $

Let $X$ be a topological space and $Z\subset X$ be a closed subset. Consider the functor $\Gamma_Z$ sending a sheaf
$\CF$ of abelian groups on $X$ to the module $\ker(\Gamma(X,\CF)\to \Gamma(X-Z,\CF))$. Let $\CF\mapsto H^i_Z(X,\CF)$
denote the $i$-th right derived functor of $\Gamma_Z(X,-)$, the \textbf{$i$-th cohomology group of $\CF$ with support
on a closed subset $Z$}. This construction admits the following relative version. Suppose that $A$ and $B$ are two
closed subsets of $X$ such that $B \subset A$. Consider the functor $\Gamma_{A/B}$ which sends a sheaf $\CF$
of abelian groups on $X$ to the module $\coker(\Gamma_{B}(X,\CF)\to \Gamma_{A}(X,\CF))$.
Let $\CF\mapsto H^i_{A/B}(X,\CF)$ denote the $i$-th right derived functor\footnote{The issues with $\coker$ not being left exact
are carefully resolved by Kempf in the beginning of~\cite[\S7]{k}.} of $\Gamma_{A/B}(X,-)$, the
\textbf{$i$-th cohomology group of $\CF$ with relative support $(A, B)$}. We note that
$H^i_A(X,\CF)=H^i_{A/\emptyset}(X,\CF)$ and $H^i_X(X,\CF)=H^i(X,\CF)$.

\medskip
\noindent
\begin{Lem}[{\cite{k}}]\label{lem: basic properties}
The functor $H^i_{A/B}(X,-)$ satisfies the following properties:

\medskip
\noindent
(a) There is a long exact sequence
\begin{equation}\label{eq:LES}
  \cdots \to  H^i_{B}(X,\CF) \to H^i_{A}(X,\CF) \to H^i_{A/B}(X,\CF) \to H^{i+1}_{B}(X,\CF) \to \cdots
\end{equation}

\medskip
\noindent
(b) Every inclusion of closed subsets $C \subseteq B$ induces a natural morphism
\begin{equation}\label{eq:naturality}
  H^i_{A/C}(X,\CF) \to H^i_{A/B}(X,\CF)
\end{equation}
functorial with respect to $A$ and $B$, and a morphism of the corresponding exact sequences~\eqref{eq:LES}.

\medskip
\noindent
(c) There is an ``excision'' isomorphism:
\begin{equation}\label{eq:excision}
  H^i_{A/B}(X,\CF) \, \iso \, H^i_{A\setminus B}(X \setminus B,\CF)
\end{equation}
functorial with respect to $A$ and $B$ (here, $A\setminus B$ denotes the complement of $B$ in $A$).
\end{Lem}

\medskip
\noindent
We will also need the following simple corollary of Lemma~\ref{lem: basic properties} and the above definitions:

\medskip
\noindent
\begin{Lem}\label{lem: cl-excision}
For any two disjoint closed subsets $Z_1$ and $Z_2$ of $X$ there exist isomorphisms:
\begin{equation}\label{eq:cl-excision}
  H^i_{Z_1}(X,\CF) \, \iso \, H^i_{Z_1}(X\setminus Z_2,\CF|_{X\setminus Z_2})
\end{equation}
and
\begin{equation}\label{eq:disjoint-excision}
  H^i_{Z_1\sqcup Z_2}(X,\CF) \, \iso \, H^i_{Z_1}(X,\CF) \oplus H^i_{Z_2}(X,\CF)
\end{equation}
functorial with respect to $Z_1$ and $Z_2$.
\end{Lem}

\medskip
\noindent
Let now $X$ be a smooth algebraic variety. Let $\CO_X$ be the structure sheaf and $\CD_X$ be the sheaf
of algebraic differential operators of finite order on $X$. We also define $\CD=\Gamma(X,\CD_X)$,
the \textbf{algebra of global differential operators on $X$}. The crucial observation is that
the above constructions remain true in the $D$-module theoretic setting:

\medskip
\noindent
\begin{Lem}[{\cite[\S2]{b}}]\label{lem: D-setting}
For a coherent sheaf $\CF$ of left $\CD_X$-modules, all cohomology groups $H^i_{A/B}(X,\CF)$ carry the natural $\CD$-action.
Moreover, all maps in Lemmas~\ref{lem: basic properties},~\ref{lem: cl-excision} are $\CD$-equivariant.
\end{Lem}

\medskip
\noindent
Suppose now that $X$ is a $G$-variety. Then, there is an evident Lie algebra homomorphism $\fg \to \CD$ (the target endowed with the commutator bracket).
It obviously induces a
$\fg$-action on $H^\bullet_{A/B}(X,\CF)$. However, whenever considering $H^\bullet_{A/B}(X,\CF)$ as a $\fg$-module,
we will be interested not literally in this $\fg$-action, but in the one
\textbf{twisted by the Chevalley involution $\phi$ of $\fg$} determined by:
\begin{equation}\label{eq:twist automorphism}
  \phi\colon \quad h\mapsto -h \,, \qquad e_\alpha \mapsto e_{-\alpha} \,, \qquad \forall\ h\in \fh \,,\ \alpha\in \Delta \,.
\end{equation}

\medskip
\noindent
Let us now recall the key tool of Cousin complexes (our exposition closely follows that of~\cite[\S3]{mr}).
Suppose that $X$ is a topological space equipped with a (\underline{not necessarily exhaustive}) filtration
\begin{equation}\label{eq:Z-filtration}
  Z_n \subseteq Z_{n-1} \subseteq \dots \subseteq Z_0 \subseteq X
\end{equation}
of closed subsets, and let $\CE$ be a sheaf of abelian groups on $X$. Picking a flabby resolution $\CE^\bullet$ of $\CE$
and considering the mapping cones $C_j=C(\Gamma_{Z_{j+1}}(X,\CE^\bullet)\to \Gamma_{Z_{j}}(X,\CE^\bullet))$, whose cohomology
are naturally isomorphic to the relative cohomology $H^\bullet_{Z_{j}/Z_{j+1}}(X,\CE)$, one can construct a double complex
$C_{\bullet,\bullet}$ with exact rows and whose $j$-th column is $C_j[j]$, the $j$-th cone $C_j$ shifted by degree $j$.
Then, on the one hand, the exactness of rows implies that the cohomology of the total complex $\mathrm{Tot}(C_{\bullet,\bullet})$
is isomorphic to the cohomology of $\Gamma_{Z_0}(X,\CE^\bullet)$, i.e.\ to $H^\bullet_{Z_0}(X,\CE)$. On the other hand,
the vertical cohomology of this double complex $C_{\bullet,\bullet}$ is $H^{k+j}_{Z_j/Z_{j+1}}(X,\CE)$, as noted above,
and the horizontal differential in the $k$-th row gives rise to the so-called \textbf{$k$-th Cousin complex}:
\begin{equation}\label{eq:Cousin}
  \mathcal{C}_k\colon \qquad
  H_{Z_{0}/Z_{1}}^{k}(X,\CE) \to H_{Z_{1}/Z_{2}}^{k+1}(X,\CE) \to H_{Z_{2}/Z_{3}}^{k+2}(X,\CE) \to \cdots
\end{equation}
Therefore, by applying the vertical spectral sequence of the double complex $C_{\bullet,\bullet}$,
we obtain:\footnote{While the Cousin complexes were introduced by Grothendieck and were first applied in the above
context~in~\cite{k},
we choose to follow the exposition of~\cite{mr} for its simplicity.}

\medskip
\noindent

\begin{Thm}[{\cite{mr}}]\label{thm:Cousin use}
If all except the $k$-th Cousin complexes are zero, then we have:
\begin{equation}\label{eq:cohomology equality}
  H^\bullet(\mathcal{C}_k) = H^\bullet_{Z_0}(X,\CE) \,.
\end{equation}
\end{Thm}

\medskip


\subsection{The geometry of partial flag varieties}
$\ $

In what follows, we shall freely use the notation of Subsection~\ref{ssec truncated-BGG}. In particular, let $G$
be a connected algebraic group with $\Lie(G)=\fg$ and $H\subset B\subset P$ be the Cartan torus, the Borel, and parabolic subgroups of $G$
with the corresponding Lie algebras $\Lie(H)=\fh$, $\Lie(B)=\fb$, $\Lie(P)=\fp$. Consider the corresponding partial flag variety
\begin{equation}\label{eq:partial flags}
  X = G/P \,.
\end{equation}
Let $B_-\subset G$ be the opposite Borel subgroup of $B$ containing $H$, and $U_{-}\subset B_{-}$ be its unipotent radical, so that
$\Lie(U_-)=\bigoplus_{\alpha\in \Delta^+} \BC e_{-\alpha}$ and $\Lie(B_-)=\fh \oplus \Lie(U_{-})$.
Then, $B_-$ naturally acts on $X$, giving rise to the stratification of $X$ by $B_-$-orbits (see e.g.~\cite[\S7.21]{ku}):
\begin{equation}\label{eq:parabolic Bruhat}
  X \ = \bigsqcup\limits_{w \in {}^{\fl}W} X_w \,, \qquad
  X_w=B_{-}wP/P=U_{-}wP/P \,, \qquad \codim_X(X_w)=l(w) \,.
\end{equation}
Here, the indexing set ${}^{\fl}W$ consists of the shortest length representatives of the left cosets $W/W_{\fl}$,
precisely as in~\eqref{eq:shortest left equiv}, and $X_w$ is an affine space of dimension equal to $l(w)$,
the length of $w\in {}^{\fl}W$.

\medskip
\noindent
Following the setup of Subsection~\ref{ssec truncated-BGG}, let $\lambda\in P^+_{\fg/\fl}$ be a dominant integral
weight of $\fg$ \underline{vanishing} on the coroot lattice of the Levi subalgebra $\fl\subset \fp$.
Let $\tilde{L}_{-\lambda}$ be the one-dimensional $P$-representation corresponding to the weight $-\lambda$, and
$\tilde{\CL}_{\lambda}$ be the corresponding $G$-equivariant line bundle on $X$:
\begin{equation}\label{eq:line bundle on X}
  \tilde{\CL}_{\lambda} = G\times_P \tilde{L}_{-\lambda} \,.
\end{equation}
For any subset $\mathsf{Y}$ of a topological space $\mathsf{X}$, we use $\ol{\mathsf{Y}}$ and $\partial(\mathsf{Y})$
to denote its closure and boundary. Since $X_w$ is locally closed (as an orbit of an algebraic group), we have:
\begin{equation}\label{eq:X-boundary}
  X_w = \ol{X}_w\setminus \partial(X_w) \,.
\end{equation}

\medskip
\noindent
All $\fg$-modules we consider in this paper do belong to the category $\CO$ of~\cite{bgg}. In particular,
every such module $V$ has the weight space decomposition with all components being finite-dimensional:
\begin{equation}\label{eq:wt decomp}
  V=\bigoplus\limits_{\nu\in \fh^*} V[\nu] \,, \qquad
  V[\nu]=\Big\{v\in V \,|\, h(v)=\nu(h)v\ \forall h\in \fh \Big\} \,.
\end{equation}
In this setup, one may define \textbf{the restricted dual module} $V^\vee\subseteq V^*$: as a vector space
\begin{equation}\label{eq:restricted-dual}
  V^\vee=\bigoplus\limits_{\nu\in \fh^*} V[\nu]^* \,,
\end{equation}
while the $\fg$-action is the restriction of the natural one on $V^*$ \textbf{twisted} by the Chevalley involution
$\phi$ of~\eqref{eq:twist automorphism}. This defines an involutive antiautoequivalence $\Phi$ of the category $\CO$:
\begin{equation}\label{eq:dual-equivalence}
  \Phi\colon V\mapsto V^\vee \,.
\end{equation}
For the finite-dimensional $\fg$-modules, we have:
\begin{equation}\label{eq:self-dual}
  L^\vee_\lambda\simeq L_\lambda \,, \qquad \forall\, \lambda\in P^+_{\fg} \,.
\end{equation}

\medskip
\noindent
\begin{Thm}\label{thm:Second Resolution}
There exists a finite length exact sequence of $\fg$-modules of the form:
\begin{equation}\label{eq:geometric resolution 2}
  0\to M_{w w_{\fl,0}\, \cdot \lambda}\to \cdots \to \bigoplus_{v\in W_{\fl}}^{l(v)=2} M_{wv\, \cdot \lambda} \to
  \bigoplus_{v\in W_{\fl}}^{l(v)=1} M_{wv\, \cdot \lambda} \to  M_{w\, \cdot \lambda} \to
  H^{l(w)}_{\ol{X}_w/\partial(X_w)}(X,\tilde{\CL}_{\lambda})^\vee \to 0 \,.
\end{equation}
\end{Thm}

\begin{proof}
Using $\Phi$ of~\eqref{eq:dual-equivalence} it suffices to prove that there exists an exact sequence of $\fg$-modules of the form:
\begin{equation}\label{eq:modified geometric resolution}
  0 \to H^{l(w)}_{\ol{X}_w/\partial(X_w)}(X,\tilde{\CL}_{\lambda}) \to
  M_{w\, \cdot \lambda}^\vee \to \bigoplus_{v\in W_{\fl}}^{l(v)=1} M_{wv\, \cdot \lambda}^\vee \to
  \bigoplus_{v\in W_{\fl}}^{l(v)=2} M_{wv\, \cdot \lambda}^\vee \to \cdots \to M_{w w_{\fl,0}\, \cdot \lambda}^\vee \to 0 \,.
\end{equation}
To this end, consider the complete flag variety
\begin{equation}\label{eq:complete flags}
  Y = G/B \,,
\end{equation}
and let $\pi\colon Y\to X$ denote the natural projection:
\begin{equation}\label{eq:full-to-partial}
  \pi\colon G/B \to G/P \,.
\end{equation}
Note that $Y$ admits a natural Bruhat decomposition by $B_{-}$-orbits, cf.~\eqref{eq:parabolic Bruhat}:
\begin{equation}\label{eq:Bruhat}
  Y = \bigsqcup\limits_{u \in W} Y_u \,,\ \qquad Y_u=B_{-}uB/B=U_{-}uB/B \,, \qquad \codim_Y(Y_u)=l(u) \,.
\end{equation}
For any $w\in {}^{\fl}W$, define $Q_w\subseteq Y$ via:
\begin{equation}\label{eq:Q-variety}
  Q_w=\pi^{-1}(B_{-}wP) \,,
\end{equation}
which is naturally stratified by $B_{-}$-orbits:
\begin{equation}\label{eq:reduction}
  Q_w=\pi^{-1}(B_{-}wP)=\bigsqcup\limits_{v \in W_{\fl}} Y_{wv} \,.
\end{equation}
Let us also note the following useful equality:
\begin{equation}\label{eq:length-addition}
  l(wv)=l(w)+l(v) \qquad \mathrm{for\ any} \qquad w\in {}^{\fl}W \,,\, v\in W_{\fl} \,.
\end{equation}

\medskip
\noindent
Let $L_{-\lambda}$ be the one-dimensional $B$-representation corresponding to the weight $-\lambda$, and
${\CL}_{\lambda}$ be the corresponding $G$-equivariant line bundle on $Y$:
\begin{equation}\label{eq:line bundle on Y}
  \CL_{\lambda}=G\times_B L_{-\lambda} \,.
\end{equation}

\medskip
\noindent
Similarly to~\cite{mr}, for any $w\in {}^{\fl}W$ consider
\begin{equation}
   U_w=Y \setminus \partial(Q_w) = Y\setminus \pi^{-1}\left(\partial(X_w)\right)
\end{equation}
(the second equality is due to $\pi$ being proper), so that $Q_w$ is closed in $U_w$.
Note that $U_w$ is naturally stratified by $B_{-}$-orbits,
which gives rise to the following filtration $Z_\bullet$ of $U_w$ by closed subsets:
\begin{equation}\label{eq:Z-filtration 1}
  Z_i\, = Q_w\cap \tilde{Z}_i \qquad  \mathrm{with} \qquad
  \tilde{Z}_i\ = \bigsqcup \limits_{\substack{u\in W\\ l(u)\geq l(w)+i}} Y_u  \,.
\end{equation}
We note that $Z_0=Q_w$, and according to~(\ref{eq:reduction},~\ref{eq:length-addition}) we have:
\begin{equation}\label{eq:union}
  \tilde{Z}_i \setminus \tilde{Z}_{i+1} \, = \bigsqcup\limits_{u \in W:\, l(u)=l(w)+i} Y_{u} \,,\qquad
  Z_i \setminus Z_{i+1} \, = \bigsqcup\limits_{v \in W_{\fl}:\, l(v)=i} Y_{wv} \,.
\end{equation}

\medskip
\noindent
We shall now apply the results of Subsection~\ref{ssec relative cohomology} to $U_w$ equipped with the
filtration~\eqref{eq:Z-filtration 1} and the sheaf $\CF=\CL_\lambda|_{U_w}$. We claim that Theorem~\ref{thm:Cousin use}
applies in this setting, and furthermore the corresponding Cousin complex, which calculates $H_{Z_0}^\bullet(U_w,\CF)$,
provides the desired exact sequence~\eqref{eq:modified geometric resolution}.

\medskip
\noindent
Indeed, according to~\cite[Proposition 9.3.7]{ku}, for any $u\in W$ we have:
\begin{equation}
  H^{k}_{\ol{Y}_u/\partial(Y_u)}(G/B,\CL_{\lambda}) =
  \begin{cases}
     M_{u\, \cdot \lambda}^\vee & \text{if } k=\codim_Y(Y_u)=l(u) \\
     0 & \text{otherwise }
  \end{cases} \,.
\end{equation}
Combining this with~\eqref{eq:union} and Lemma~\ref{lem: cl-excision}, we obtain:
\begin{equation}
  H^{i+l(w)}_{Z_i/Z_{i+1}}(U_w,\CF) =
  H^{i+l(w)}_{Z_i\setminus Z_{i+1}}(U_w \setminus Z_{i+1}, \CL_{\lambda})\, \simeq
  \bigoplus_{v\in W_{\fl}}^{l(v)=i} M_{wv\, \cdot \lambda}^\vee
\end{equation}
and
\begin{equation}
  H^{k}_{Z_i/Z_{i+1}}(U_w,\CF)=0 \qquad \mathrm{for}\qquad k\ne i+l(w) \,.
\end{equation}

\medskip

\begin{Rem}\label{rem:applicability}
We note that Lemma~\ref{lem: cl-excision} does apply, due to:
\begin{enumerate}

\item [(1)]
  $Z_i \setminus Z_{i+1} = \bigsqcup\limits_{v\in W_{\mathfrak l}:\, l(v)=i} Y_{wv}$;
	
\item [(2)]
  for every cell $Y_u \subset Z_i\setminus Z_{i+1}$, we have $\partial(Y_u) \subseteq \partial(Q_w) \cup Z_{i+1}$,
  so that
  \begin{equation*}
    U_w\setminus Z_{i+1} = Y \setminus (\partial(Q_w)\cup Z_{i+1}) =
    V_u \setminus ((V_u \cap \partial(Q_w)) \cup (V_u \cap Z_{i+1})) \,,
  \end{equation*}
  where $V_u=Y \setminus \partial(Y_u)$,
  and $V_u\cap \partial(Q_w), V_u\cap Z_{i+1}$ are closed subsets of $V_u$ disjoint from $Y_u$.

\end{enumerate}
\end{Rem}

\medskip
\noindent
Thus, all Cousin complexes $\mathcal{C}_k$ of~\eqref{eq:Cousin} vanish for $k\ne l(w)$, while the terms
of $\mathcal{C}_{l(w)}$ precisely coincide with the terms of the exact sequence~\eqref{eq:modified geometric resolution}.
Applying Theorem~\ref{thm:Cousin use}, we therefore get $H^i(\mathcal{C}_{l(w)})=H^{i+l(w)}_{Z_0}(U_w,\CF)$.

\medskip
\noindent
Hence, it remains to prove that:
\begin{enumerate}

\item [(I)]
  $H^0(\mathcal{C}_{l(w)})=H^{l(w)}_{Z_0}(U_w,\CF)$ coincides with the $\fg$-module
\begin{equation}\label{eq:N-module}
  N_w(\lambda)= H^{l(w)}_{\ol{X}_w/\partial(X_w)}(X,\tilde{\CL}_{\lambda})
\end{equation}

\item [(II)]
 $H^i(\mathcal{C}_{l(w)})=H^{i+l(w)}_{Z_0}(U_w,\CF)$ vanishes for $i\ne 0$.
\end{enumerate}
Both results follow immediately from a $B$-version of Theorem~\ref{thm:purity 2} below,  the excision isomorphism
\begin{equation*}
  H^{k}_{\ol{X}_w/\partial(X_w)}(X,\tilde{\CL}_{\lambda}) \simeq
  H^{k}_{X_w}((G/P)\setminus \partial(X_w), \tilde{\CL}_{\lambda}) \,,
\end{equation*}
cf.~\eqref{eq:X-boundary}, and the following two lemmas:

\medskip

\begin{Lem}[{\cite[p.~286]{gs}}]\label{lem:grif-schm}
$R^0\pi_*(\CL_{\lambda})=\tilde{\CL}_{\lambda}$ and $R^{>0}\pi_*(\CL_{\lambda})=0$.
\end{Lem}

\medskip

\begin{Lem}[{\cite[Expos\'{e}~5, Lemme~3.2]{g}}]\label{eq:groth}
Let $f\colon X \to X'$ be an arbitrary morphism, $S \subset X'$ be a closed subset,
and $\CF$ be a sheaf of abelian groups on $X$. Then, there is a spectral~sequence:
\begin{equation*}
  H_{S}^{i}\left(X', R^{j}f_{*} \CF\right) \Rightarrow H_{f^{-1}(S)}^{i+j}(X,\CF) \,.
\end{equation*}
\end{Lem}

\medskip
\noindent
Therefore, the $l(w)$-th Cousin complex $\mathcal{C}_{l(w)}$ realizes the exact
sequence~\eqref{eq:modified geometric resolution} of $\fg$-modules, which produces the exact
sequence~\eqref{eq:geometric resolution 2} upon a further application of the antiautoequivalence $\Phi$
of~\eqref{eq:dual-equivalence}.
\end{proof}

\medskip
\noindent
Let us now recall the highest weight $\fg$-modules
  $M'_{w\, \cdot \lambda}=M_{w\, \cdot \lambda}/M^{\mathrm{sing}}_{w\, \cdot \lambda}$
for $w\in {}^{\fl}W$, defined by the formulas~(\ref{eq:hard modules},~\ref{eq:critical vectors}) in the Introduction.
They admit the following geometric interpretation:

\medskip
\noindent

\begin{Lem}\label{lem:geometric interpretation}
For any $w\in {}^{\fl}W$, we have the isomorphism of $\fg$-modules:
\begin{equation}\label{eq:algebro-geometric identifiction}
  M'_{w\, \cdot \lambda} \simeq N_w(\lambda)^\vee \,.
\end{equation}
\end{Lem}

\medskip

\begin{proof}
This immediately follows from the following fragment of~\eqref{eq:geometric resolution 2} using the
notation~\eqref{eq:N-module}:
\begin{equation*}
  \bigoplus_{v \in W_{\fl}}^{l(v)=1} M_{wv\, \cdot \lambda} \to M_{w\, \cdot \lambda} \to N_w(\lambda)^\vee \to 0 \,.
\end{equation*}
\end{proof}

\medskip
\noindent
Combining Theorem~\ref{thm:Second Resolution} and Lemma~\ref{lem:geometric interpretation}, we immediately obtain:

\medskip

\begin{Cor}\label{cor:proof of second resolutions}
All $\fg$-modules $\{M'_{w\, \cdot \lambda}\}_{w\in {}^{\fl}W}$ admit
resolutions~\eqref{eq:conjectured resolution 2} by Verma modules.
\end{Cor}

\medskip
\noindent
As a direct corollary of~\eqref{eq:conjectured resolution 2}, we obtain the character formula for the $\fg$-modules
$M'_{w\, \cdot \lambda}$, cf.~\eqref{eq:character identity}:

\medskip

\begin{Lem}\label{lem:M-char}
For $w\in {}^{\fl}W$, we have:
\begin{equation}\label{eq:character formula}
  \ch_{M'_{w\, \cdot \lambda}} =
  \frac{e^{w(\lambda+\rho)-\rho}}
       {\prod_{\alpha \in \Delta^{+} \setminus w(\Delta_{\fl}^{+})}\left(1-e^{-\alpha}\right)} \,.
\end{equation}
\end{Lem}

\medskip

\begin{proof}
The existence of the resolution~\eqref{eq:conjectured resolution 2} implies the following equality of characters:
\begin{multline}\label{eq:char 1}
  \ch_{M'_{w\, \cdot \lambda}}=\sum_{v\in W_{\fl}} (-1)^{l(v)} \ch_{M_{wv\, \cdot \lambda}}=
  \frac{\sum_{v\in W_{\fl}} (-1)^{l(v)} e^{wv(\lambda+\rho)-\rho}}
       {\prod_{\alpha\in \Delta^+} (1-e^{-\alpha})} = \\
  e^{w(\lambda+\rho)-\rho} \cdot
  \frac{\sum_{v\in W_{\fl}} (-1)^{l(v)} e^{w(v(\rho)-\rho)}}
       {\prod_{\alpha\in \Delta^+} (1-e^{-\alpha})} =
  e^{w(\lambda+\rho)-\rho} \cdot
  \frac{\sum_{v\in W_{\fl}} (-1)^{l(v)} e^{w(v(\rho_{\fl})-\rho_{\fl})}}
       {\prod_{\alpha\in \Delta^+} (1-e^{-\alpha})} \,,
\end{multline}
where $\rho$ is defined in~\eqref{eq:rho-weight}, $\rho_{\fl}$ is defined in~\eqref{eq:rho-Levi},
and the last two equalities follow from:
\begin{equation}
  v(\lambda)=\lambda \,,\qquad  v(\rho)-\rho=v(\rho_{\fl})-\rho_{\fl} \,, \qquad
  \forall\ v\in W_{\fl} \,,\ \lambda\in P^+_{\fg/\fl} \,.
\end{equation}
Applying the Weyl denominator formula~\eqref{eq:Weyl denominator} for $\fl$:
\begin{equation}\label{eq:char 2}
  \sum_{v\in W_{\fl}} (-1)^{l(v)} e^{v(\rho_\fl)-\rho_{\fl}}\ =\prod_{\alpha\in \Delta^+_{\fl}} (1-e^{-\alpha})
\end{equation}
and noting $w(\Delta^+_{\fl})\subseteq \Delta^+$, due to~\eqref{eq:shortest left},
we get precisely the character formula~\eqref{eq:character formula}.
\end{proof}

\medskip

\begin{Rem}
The corresponding formula for $\ch_{N_w(\lambda)}$ goes back to~\cite[Lemma 12.8]{k}.
\end{Rem}

\medskip
\noindent
The next result is the key point of the further discussion:

\medskip

\begin{Thm}\label{thm:purity 2}
For any $w\in {}^{\fl}W$, we have:
\begin{equation}
  H^i_{\ol{X}_w/\partial(X_w)}(X,\tilde{\CL}_{\lambda})=0  \quad \mathrm{for\ any} \quad i\ne l(w) \,.
\end{equation}
\end{Thm}

\medskip

\begin{proof}
It is an immediate consequence of the local purity theorems (see~\cite[Proposition 4.1]{al}).
\end{proof}

\medskip


\subsection{Derivation of the truncated BGG resolutions}
$\ $

Let $\lambda\in P^+_{\fg/\fl}$ be a dominant integral weight of $\fg$ vanishing on the coroot lattice of $\fl$,
see~\eqref{eq:Levi-trivial}. Now we are ready to derive both
resolutions~(\ref{eq:conjectured resolution 1},~\ref{eq:conjectured resolution 2}) of the Introduction,
cf.~Main~Theorem~\ref{mthm 1}.

\medskip

\begin{Thm}\label{Thm:KEY}
For $\lambda\in P^+_{\fg/\fl}$, the irreducible finite-dimensional $\fg$-module $L_\lambda$ has a finite length
resolution~\eqref{eq:conjectured resolution 1}, with each term admitting a finite length
resolution~\eqref{eq:conjectured resolution 2} by Verma modules.
\end{Thm}

\medskip

\begin{proof}
We will prove the dualized version of the desired statement, just as in our proof of Theorem~\ref{thm:Second Resolution}.
Let us consider the sheaf $\CF=\tilde{\CL}_{\lambda}$ and the following filtration of $X$ by closed subsets:
\begin{equation}\label{eq:Z-filtration 2}
  Z_i \, = \bigsqcup\limits_{w\in {}^{\fl}W:\, l(w) \geq i} X_w \,.
\end{equation}
We claim that Theorem~\ref{thm:Cousin use} applies in this setup and gives rise to the following exact sequence of
$\fg$-modules:
\begin{equation}\label{eq:geometric resolution 1}
  0\to L_\lambda^\vee \to (M'_\lambda)^\vee \to \bigoplus_{w\in {}^{\fl}W}^{l(w)=1} (M'_{w\, \cdot \lambda})^\vee \to
  \bigoplus_{w\in {}^{\fl}W}^{l(w)=2} (M'_{w\, \cdot \lambda})^\vee \to  \cdots \to (M'_{{}^{\fl,0}w\, \cdot \lambda})^\vee
  \to 0 \,.
\end{equation}

\medskip
\noindent
Indeed, all Cousin complexes $\mathcal{C}_k$ of~\eqref{eq:Cousin} vanish for $k\ne 0$,
due to Theorem~\ref{thm:purity 2} and Lemma~\ref{lem: cl-excision}:
\begin{equation}\label{eq:Cousin comp 1}
  H^{k+j}_{Z_j/Z_{j+1}}(X,\CF) = H^{k+j}_{Z_j\setminus Z_{j+1}}(X \setminus Z_{j+1},\CF) \ =
  \bigoplus_{w\in {}^{\fl}W}^{l(w)=j} H^{k+j}_{\ol{X}_w/\partial(X_w)}(X,\tilde{\CL}_{\lambda})=0 \,.
\end{equation}
Therefore, Theorem~\ref{thm:Cousin use} applies, and we get:
\begin{equation}\label{eq:Cousin cor 1}
  H^i(\mathcal{C}_0)=H^i_{X}(X,\CF)=H^i(X,\CF) \,.
\end{equation}
According to the parabolic version of the Borel-Weil-Bott theorem~\cite[Theorem 6.4]{kos},
combined with our conventions of all geometric $\fg$-actions being twisted by the Chevalley involution $\phi$
of~\eqref{eq:twist automorphism} that also enters our definition~(\ref{eq:restricted-dual},~\ref{eq:dual-equivalence})
of the restricted dual $\fg$-module, we have (cf.~\eqref{eq:self-dual}):
\begin{equation}\label{eq:Borel-Weil-Bott}
  H^{i}(X,\CF) = H^{i}(X,\tilde{\CL}_{\lambda}) \simeq
  \begin{cases}
    L_{\lambda}^\vee \simeq L_\lambda & \text{for } i=0 \\
    0  & \text{for } i\ne 0
  \end{cases} \,.
\end{equation}
On the other hand, the $j$-th term of $\mathcal{C}_0$ is computed using Lemmas~\ref{lem: cl-excision}
and~\ref{lem:geometric interpretation} similarly to~\eqref{eq:Cousin comp 1}:
\begin{equation}\label{eq:Cousin comp 2}
  H^{j}_{Z_j/Z_{j+1}}(X,\CF)\ =
  \bigoplus_{w\in{}^{\fl}W}^{l(w)=j} H^{l(w)}_{\ol{X}_w/\partial(X_w)}(X,\tilde{\CL}_{\lambda}) \ =
  \bigoplus_{w\in {}^{\fl}W}^{l(w)=j} (M'_{w\, \cdot \lambda})^\vee \,.
\end{equation}
Combining~(\ref{eq:Cousin cor 1},~\ref{eq:Borel-Weil-Bott},~\ref{eq:Cousin comp 2}), we see that the Cousin
complex $\mathcal{C}_0$ realizes the exact sequence~\eqref{eq:geometric resolution 1}.

\medskip
\noindent
Applying the antiautoequivalence $\Phi$ of~\eqref{eq:dual-equivalence} to~\eqref{eq:geometric resolution 1} produces
the resolution~\eqref{eq:conjectured resolution 1}, while the resolutions~\eqref{eq:conjectured resolution 2}
were constructed in Corollary~\ref{cor:proof of second resolutions}. This completes our proof of the theorem.
\end{proof}

\medskip
\noindent
\begin{Rem}\label{rem:MR gaps}
Our argument above has been strongly influenced by~\cite{mr}, where the Lepowsky parabolic BGG
resolution~\eqref{eq:Lepowsky resolution 1} was interpreted via the Cousin complex on the complete flag variety $Y=G/B$
stratified by $P$-orbits. Nevertheless, there are two subtle points in~\cite{mr}:
\begin{enumerate}

\item[(1)]
For $u\in W$, let $Z_u$ denote the $B$-orbit $BuB/B \subseteq Y$. It is stated in~\cite[after (3.2)]{mr} that:
\begin{equation}\label{eq:B-case}
  H^{k}_{\ol{Z}_u/\partial(Z_u)}(G/B,\CL_{\lambda}) =
  \begin{cases}
	M_{u \, \cdot \lambda}^\vee & \text{if } k=\codim_Y(Z_u) \\
	0 & \text{otherwise }
  \end{cases} \,.
\end{equation}
This (as well as $\lambda \neq 0$ case at~\cite[p.~55]{b}) is wrong, as we rather have:

\medskip

\begin{Lem}
\begin{equation}\label{eq:B-case updated}
  H^{k}_{u}(\lambda) := H^{k}_{\ol{Z}_u/\partial(Z_u)}(G/B,\CL_{\lambda}) =
  \begin{cases}
    M_{uw_0 \cdot \lambda}^\vee & \text{if } k=\codim_Y(Z_u) \\
    0 & \text{otherwise }
  \end{cases} \,.
\end{equation}
\end{Lem}

\medskip

\begin{proof}
When $u = w_0$, this statement is well-known (cf.~\cite[Claim 2.4.2]{Bez}).
One can show that the general case holds along the lines of \cite[Proposition 7]{b}.
Indeed, let $w_{0}u^{-1}=s_{\alpha_{1}} \cdot \ldots \cdot s_{\alpha_{N-\ell(u)}}$ be a reduced decomposition, where
$N=l(w_0)=|\Delta^+|$. Then, similarly to the argument in~\cite[Lemma~4]{b}, it is easy to see that we still have
a sequence of the following $\fg$-module epimorphisms:
\begin{equation*}
	H^0_{w_0}(\lambda) = H^0_{s_{\alpha_1}s_{\alpha_{2}} \cdots s_{\alpha_{N-l(u)}}u}(\lambda)\twoheadrightarrow
	H^1_{s_{\alpha_{2}} \cdots s_{\alpha_{N-l(u)}}u}(\lambda) \twoheadrightarrow \cdots \twoheadrightarrow
	H^{c(u)-1}_{s_{\alpha_{N-l(u)}}u}(\lambda) \twoheadrightarrow H^{c(u)}_u(\lambda) \,,
\end{equation*}
where $c(u)=\codim_Y(Z_u)=N-l(u)$. Thus, the argument of~\cite{b} still applies and we get:
\begin{equation*}
  H^{c(u)}_{u}(\lambda) = M_{(s_{\alpha_{N-l(u)}} \cdots\, s_{\alpha_1}) \, \cdot \lambda}^\vee =
  M_{uw_0 \, \cdot \lambda}^\vee \,.
\end{equation*}
On the other hand, the vanishing result in~\eqref{eq:B-case updated} is just the $B$-case of Theorem~\ref{thm:purity 2}.
\end{proof}

\item[(2)]
The use of $H^\bullet_\bullet(-, \CL_\lambda\otimes \CK)$ in~\cite{mr} is wrong.

\end{enumerate}
However, both of the above inaccuracies can be easily fixed by replacing $\CL_\lambda\otimes \CK$ with $\CL_\lambda$
and considering the stratification of $Y$ by $P_{-}$-orbits, where $P_-\subset G$ is the opposite parabolic subgroup.
\end{Rem}

\medskip

\begin{Rem}
(a) It is instructive to point out that the results of~\cite{mr} provide the answer to~\cite[Open Problem 9.3.19]{ku}.
The only difference is that~\textit{loc.cit}.\ treats the case of an arbitrary Kac-Moody algebra. Nonetheless, the results
of~\cite{mr}, as well as ours, admit natural generalizations to such infinite-dimensional setup through the usual
stratification by Schubert varieties.

\medskip
\noindent
(b) We also note that the other possible way to generalize our results is by considering an arbitrary dominant integral weight
$\lambda\in P^+_{\fg}$, not necessarily vanishing on the coroot lattice of $\fl$. In this case, one obtains (exactly as above)
the resolutions of the form~(\ref{eq:conjectured resolution 1},~\ref{eq:conjectured resolution 2}) with $\tilde{\CL}_\lambda$
being replaced by $R^0\pi_*(\CL_\lambda)$ (note that $R^{>0}\pi_*(\CL_{\lambda})=0$, according to~\cite{gs}). However, the
corresponding infinite-dimensional $\fg$-modules (realized as $H^{l(w)}_{\ol{X}_w/\partial(X_w)}(X,R^0\pi_*(\CL_\lambda))^\vee$)
are not defined for $\lambda\notin P^+_{\fg}$ in this case, in contrast to such a key feature of our modules
$M'_{w\, \cdot \lambda}$ of~\eqref{eq:hard modules} as discussed in Subsection~\ref{ssec analytic continuation}.
\end{Rem}

\medskip


\section{Standard BGG}\label{sec: A-type Verma BGG}

In this section, we recall the standard relation between the transfer matrices of $A$-type spin chains
corresponding to finite-dimensional and infinite-dimensional (dual Verma) $\gl_n$-modules provided
via oscillator Lax matrices, as summarized in the Introduction. This exposition is mostly to motivate
the key constructions and results of the upcoming sections. We also provide an overview of
the factorisation and the determinant formulas in this setup, as mentioned in the~Introduction.

\medskip


\subsection{Oscillator realization in type A (Verma)}
$\ $

For any $n\in \BZ_{\geq 2}$, let $\CA$ denote the oscillator algebra generated by $\frac{n(n-1)}{2}$ pairs
of oscillators $\{(\oa_{j,i},\oad_{i,j})\}_{1\leq i<j\leq n}$ subject to the standard defining relations:
\begin{equation}\label{eq:oscillator relations}
   [\oa_{j,i},\oad_{k,\ell}]=\delta_{i}^{k}\delta_{j}^{\ell} \,, \qquad
   [\oa_{j,i},\oa_{\ell,k}]=0 \,, \qquad
   [\oad_{i,j},\oad_{k,\ell}]=0 \,,
\end{equation}
so that
\begin{equation}
  \CA = \BC \Big\langle \oa_{j,i} \, , \, \oad_{i,j} \Big\rangle_{1\leq i<j\leq n} \, \Big/ \,
  \eqref{eq:oscillator relations} \,.
\end{equation}
Fix $\lambda=(\lambda_1,\ldots,\lambda_n)\in \BC^n$. Following~\cite[\S2]{Derkachov:2006fw} (going back to~\cite{gn}),
let us consider the $\gl_n$-type $\CA[x]$-valued Lax matrix (i.e.\ a solution of the RTT relation~\eqref{eq:RTT} with
the $R$-matrix of~\eqref{eq:R-matrix-A}):
\begin{equation}\label{eq:classical-Lax}
  \CL_\lambda(x)=U^{-1}(x+D_\lambda)U
\end{equation}
defined through the matrices:
\begin{equation}\label{eq:UD}
  U=
  \left(\begin{array}{cccc}
     1 & -\oad_{1,2} & \cdots & -\oad_{1,n} \\
     \vdots & \ddots & \ddots & \vdots \\
     0 & \cdots & 1 & -\oad_{n-1,n} \\
     0 & \cdots & 0 & 1 \\
  \end{array} \right) \,, \qquad
  D_\lambda=
  \left(\begin{array}{cccc}
     \lambda_1 & 0 & \cdots & 0 \\
     \tilde{\oa}_{2,1} & \lambda_2-1 & \ddots & \vdots \\
     \vdots & \ddots & \ddots & 0 \\
     \tilde{\oa}_{n,1} & \cdots & \tilde{\oa}_{n,n-1} & \lambda_n-n+1 \\
  \end{array}\right) \,,
\end{equation}
with
\begin{equation}
  \tilde{\oa}_{j,i}=-\oa_{j,i}\,+\sum_{k=j+1}^{n}\oad_{j,k}\oa_{k,i} \,.
\end{equation}
Writing~\eqref{eq:classical-Lax} in the form
\begin{equation}
  \CL_\lambda(x)=x\ID_n + \sum_{i,j=1}^n e_{ij}\CE_{ji} \,,
\end{equation}
we note that the RTT relation~\eqref{eq:RTT} implies that $\{\CE_{ij}\}_{i,j=1}^n$
satisfy the $\gl_n$ commutation relations:
\begin{equation}\label{eq:glncom}
  [\CE_{ij},\CE_{k\ell}]=\delta_{j}^{k}\CE_{i\ell}-\delta_{\ell}^{i}\CE_{kj}\,.
\end{equation}

\medskip
\noindent
Let us consider the standard Fock module $\Fock$ of $\CA$, generated by the Fock vacuum $|0\rangle \in \Fock$ satisfying
\begin{equation}
  \oa_{j,i} |0\rangle = 0 \,, \qquad 1\leq i< j\leq n \,.
\end{equation}
Thus, $\Fock$ has a basis obtained by the action of the pairwise commuting \emph{creation} operators on $|0\rangle$:
\begin{equation}\label{eq:vec}
  |\vec{m}\, \rangle \, = \prod_{1\leq i<j\leq n} \oad_{i,j}^{m_{i,j}} \, |0\rangle \,, \qquad
  \forall\, \vec{m}=(m_{i,j})_{1\leq i<j\leq n}\in \BN^{\frac{n(n-1)}{2}} \,.
\end{equation}
We shall use $\langle \vec m|$ to denote the dual basis of $\Fock^*$, so that
\begin{equation}\label{eq:vec-covec}
  \langle \vec k|X|\vec m\rangle = \langle \vec k| \Big( X|\vec{m}\rangle \Big)
\end{equation}
denotes the $(|\vec k\rangle, |\vec m\rangle)$-matrix coefficient of any linear operator $X$
acting on the Fock space $\Fock$.

\medskip
\noindent
By straightforward computation, we find:
\begin{equation}\label{eq:A-diag}
  \CE_{ii}=\lambda_i + \sum_{k<i}\oad_{k,i}\oa_{i,k} - \sum_{k>i}\oad_{i,k}\oa_{k,i}  \,, \qquad 1\leq i\leq n \,,
\end{equation}
\begin{equation}\label{eq:A-overdiag}
  \CE_{ij}=-\oa_{j,i} + \sum_{k<i} \oad_{k,i}\oa_{j,k} \,, \qquad  1\leq i<j\leq n \,.
\end{equation}
Hence, the Fock vacuum $|0\rangle$ is a highest weight state of the resulting $\gl_n$-action:
\begin{equation}
  \CE_{ij}|0\rangle =0\,, \qquad 1\leq i<j\leq n\,,
\end{equation}
with the highest weight $\lambda$, that is:
\begin{equation}
  \CE_{ii} |0\rangle = \lambda_i |0\rangle \,, \qquad 1\leq i\leq n \,.
\end{equation}
We can now identify the resulting $\gl_n$-modules $\Fock$ with those featuring in the Introduction:

\medskip

\begin{Lem}\label{lem:A-Fock as Verma}
There is a $\gl_n$-module isomorphism:
\begin{equation}\label{eq:A-Fock-as-BasicVerma}
  \Fock \simeq M_\lambda^\vee \,,
\end{equation}
identifying $\Fock$ with the restricted dual~\eqref{eq:restricted-dual} of the highest weight Verma module $M_\lambda$.
\end{Lem}

\medskip

\begin{proof}
Since the restricted dual $\Fock^\vee$ has the same highest weight and the character as $M_\lambda$, it suffices
to prove that it is also a highest weight $\gl_n$-module, i.e.\ generated by its highest weight vector $\langle 0|$.
To this end, we note that the formula~\eqref{eq:A-overdiag} implies that $\langle \vec{m}|$ is in fact
a non-zero multiple of
\begin{equation*}
  (\CE^*_{12})^{m_{1,2}} \cdots (\CE^*_{1n})^{m_{1,n}}
  (\CE^*_{23})^{m_{2,3}} \cdots (\CE^*_{2n})^{m_{2,n}} \cdots (\CE^*_{n-1,n})^{m_{n-1,n}} \langle 0|
\end{equation*}
for any $\vec{m}=(m_{i,j})_{1\leq i<j\leq n}\in \BN^{\frac{n(n-1)}{2}}$, where $\CE^*_{ij}\in \End(\Fock^*)$
is the dual of the $\CE_{ij}$-action on $\Fock$.
\end{proof}

\medskip
\noindent
Combining this with the determinant formula of~\cite{j} and the isomorphism~\eqref{eq:self-dual}, we obtain:

\medskip

\begin{Cor}\label{cor:t-special}
(a) The Fock space $\Fock$ is irreducible as a $\gl_n$-module if and only if
\begin{equation}\label{eq:gl-generic}
  \lambda_i-\lambda_j\notin i-j+\BZ_{>0} \,, \qquad \forall\, 1\leq i<j\leq n \,.
\end{equation}

\medskip
\noindent
(b) The Fock vacuum $|0\rangle$ generates an irreducible finite-dimensional $\gl_n$-module $L_{\lambda}$
if and only if
\begin{equation}\label{eq:gl-dominant}
  \lambda\in P^+ = \Big\{ \mu\in \BC^n \, |\, \mu_i-\mu_{i+1}\in \BN \ \ \forall\, 1\leq i<n \Big\} \,.
\end{equation}
\end{Cor}

\medskip


\subsection{Transfer matrices}
$\ $

Recall the notion of transfer matrices $\{T_W(x)\}_{W\in \mathrm{Rep}\, Y(\gl_n)}$,
as discussed in Subsection~\ref{ssec R-matrix}.
In particular, we shall consider the following explicit infinite-dimensional transfer matrices:
\begin{equation}\label{eq:Verma-via-Q}
  T^+_{\lambda}(x)=
  \tr\, \prod_{i=1}^{n} \tau_i^{\CE_{ii}}\, \underbrace{\CL_{\lambda}(x) \otimes \dots \otimes \CL_{\lambda}(x)}_{N} \,.
\end{equation}
Here, we use the $N$-fold tensor product and the trace is taken over the entire Fock space~(\ref{eq:vec},~\ref{eq:vec-covec}):
\begin{equation}\label{eq:Fock-trace}
  \tr(X) = \sum_{\vec{m}} \langle \vec m|X|\vec m\rangle \,.
\end{equation}

\medskip

\begin{Rem}
The twist parameters $\tau_i \in \BC$ lift the degeneracies in the spectrum, i.e.\ break
the $\ssl_n$ invariance of the transfer matrix, and regularize the infinite-dimensional trace.
\end{Rem}

\medskip
\noindent
For a dominant integral $\lambda\in P^+$, see~\eqref{eq:gl-dominant}, we also consider the finite-dimensional
transfer matrices $T_{\lambda}(x)$ corresponding to the modules $L_{\lambda}$ in the auxiliary space:
those are defined similarly to~\eqref{eq:Verma-via-Q}, but with the trace taken over the finite-dimensional
submodule $L_{\lambda}$ of $\Fock$, see Corollary~\ref{cor:t-special}(b).

\medskip
\noindent
Recall the dot action~\eqref{eq:dot action} of $S_n$ on $\BC^n$, cf.~\eqref{eq:transfer-via-BGG}:
\begin{equation}\label{eq:rho-again}
  \sigma\, \cdot \lambda = \sigma(\lambda+\rho)-\rho \,, \qquad
  \rho=\left(\sfrac{n-1}{2},\sfrac{n-3}{2},\ldots,\sfrac{1-n}{2}\right)\in \BC^n \,.
\end{equation}
Then, according to~\cite{Bazhanov:2010jq,Derkachov:2010qe}, while for low ranks it goes back
to~\cite{Bazhanov:1998dq, Bazhanov:2001xm}, we have:

\medskip

\begin{Thm}\label{thm:baby-bgg}
For $\lambda\in P^+$, we have:
\begin{equation}\label{eq:A-transfer-relation}
  T_\lambda(x)=\sum_{\sigma\in S_n} (-1)^{l(\sigma)}T^+_{\sigma \, \cdot \lambda}(x) \,.
\end{equation}
\end{Thm}

\medskip
\noindent
As recalled in Subsection~\ref{ssec BGG resol}, the formula~\eqref{eq:A-transfer-relation} is a consequence of
the BGG resolution of the finite-dimensional $\gl_n$-module $L_\lambda$ by means of the infinite-dimensional
dual Verma $\gl_n$-modules:
\begin{equation}\label{eq:A-resolution-basic}
  0\to L_{\lambda}\to M^\vee_{\lambda}\to
  \bigoplus_{\sigma\in S_n}^{l(\sigma)=1} M^\vee_{\sigma\, \cdot \lambda}\to
  \bigoplus_{\sigma\in S_n}^{l(\sigma)=2} M^\vee_{\sigma\, \cdot \lambda} \to \cdots \to
  M^\vee_{(\lambda_n+1-n,\lambda_{n-1}+3-n,\ldots,\lambda_1+n-1)}\to 0 \,,
\end{equation}
cf.~(\ref{eq:restricted-dual},~\ref{eq:geometric resolution 1}).
We note that the character limit of~\eqref{eq:A-transfer-relation}, corresponding to its special case with $N=0$
(zero length of the spin chain), recovers the classical Weyl character formula:
\begin{equation}\label{eq:char-A-basic}
  \ch_{L_\lambda}=
  \sum_{\sigma\in S_n} (-1)^{l(\sigma)}
  \frac{e^{\sigma(\lambda+\rho)-\rho}}{\prod_{\alpha\in \Delta^+} (1-e^{-\alpha})}=
  \sum_{\sigma\in S_n} (-1)^{l(\sigma)} \ch_{M^\vee_{\sigma\, \cdot \lambda}} \,,
\end{equation}
where $\Delta^+$ denotes the set of positive roots of $\fg=\ssl_n$. In particular,
identifying simple roots with $\alpha_i=\epsilon_i-\epsilon_{i+1}$ in the standard way, we get
$\Delta^+=\{\epsilon_i-\epsilon_j\}_{1\leq i<j\leq n}$ and~\eqref{eq:rho-again} agrees with~\eqref{eq:rho-weight}.

\medskip

\begin{Rem}\label{rem:math-to-physics}
For the physics' reader, let us explain how~\eqref{eq:char-A-basic} fits into the rest of our notation.
Consider a general Cartan element $h=\sum_{i=1}^n x_i \epsilon^*_i\in \fh\subset \gl_n$
(here, one can think of $x_i\in \BC$ or as a formal parameter), and let $\tau_i=e^{x_i}$.
Then, $e^h$ and $\prod_{i=1}^n \tau_i^{\CE_{ii}}$ agree and the formula~\eqref{eq:char-A-basic} reads:
\begin{equation}
  \ch_{L_\lambda}(e^h) = \tr_{L_\lambda} \prod_{1\leq i\leq n}\tau_i^{\CE_{ii}} =
  \sum_{\sigma\in S_n} (-1)^{l(\sigma)}
    \frac{e^{(\sigma(\lambda+\rho)-\rho,h)}}{\prod_{i<j} (1-\frac{\tau_j}{\tau_i})} =
  \sum_{\sigma\in S_n} (-1)^{l(\sigma)} \ch_{M^\vee_{\sigma\,\cdot \lambda}}(e^h)
\end{equation}
with
\begin{equation}
  \ch_{M^\vee_\lambda}(e^h) = \tr_{M^\vee_\lambda} \prod_{1\leq i\leq n}\tau_i^{\CE_{ii}} =
  \prod_{1\leq i\leq n} \tau_i^{\lambda_i}\prod_{1\leq i <j\leq n}\frac{1}{1-\frac{\tau_j}{\tau_i}} \,,
\end{equation}
where the left-hand sides denote the traces of $e^h$ (invertible diagonal element of $GL_n$) on $L_\lambda,M^\vee_\lambda$.
\end{Rem}

\medskip


\subsection{Factorisation}
$\ $

The factorisation formula for the transfer matrices $T^+_\lambda(x)$ of the restricted dual of Verma modules,
cf.~\eqref{eq:restricted-dual} and Lemma~\ref{lem:A-Fock as Verma},  was proven in~\cite{Bazhanov:2010jq}.
It was further combined with the BGG-relation~\eqref{eq:A-transfer-relation} to derive the determinant expression for
the finite-dimensional transfer matrices $T_\lambda(x)$. We review these constructions in the present subsection.

\medskip
\noindent
Following~\cite[(1.16)]{Bazhanov:2010jq}, let us consider the following $\gl_n$-type $\CA[x]$-valued Lax matrices:
\begin{equation}\label{eq:partonic-A}
  L_i(x)=
  \left(\begin{array}{ccccccc}
    1 &&& \oad_{i,1} &&& \\
    & \ddots && \vdots &&& \\
    && 1 & \oad_{i,i-1} &&& \\
    \oa_{1,i} & \cdots & \oa_{i-1,i} &
    x+\sum_{j=1}^{i-1} \oa_{j,i}\oad_{i,j}-\sum_{j=i+1}^n \oad_{i,j}\oa_{j,i} & \oad_{i,i+1} & \cdots & \oad_{i,n} \\
    &&& -\oa_{i+1,i} & 1 && \\
    &&& \vdots && \ddots & \\
    &&& -\oa_{n,i} &&& 1 \\
  \end{array}\right)
\end{equation}
with $1\leq i\leq n$ and the oscillators $\{\oa_{j,i},\oad_{i,j}\}_{j\ne i}$ subject to
the standard commutation relations~\eqref{eq:oscillator relations}.

\medskip
\noindent
The following  factorisation formula has been shown in~\cite[Appendix B]{Bazhanov:2010jq}:
\begin{equation}\label{fullfac}
  L_1(x+\ell_1) L_2(x+\ell_2) \cdots L_n(x+\ell_n) = \bfS \CL_{\lambda}(x) \Gop \bfS^{-1} \,,
\end{equation}
where $\CL_{\lambda}(x)$ is the Lax matrix of~\eqref{eq:classical-Lax} and the shifted weights
$\{\ell_i\}_{i=1}^n$ are defined via:
\begin{equation}
  \ell_i=\lambda_i-i+1 \,, \qquad 1\leq i\leq n \,.
\end{equation}
Here, the similarity transformation $\bfS=\bfS_n \cdots \bfS_1$ is defined via:
\begin{equation}\label{eq:elementary similarity}
  \bfS_i=\exp \left[\sum_{j=1}^{i-1} \left(\oad_{ji}-\sum_{k=j+1}^{i-1}\oad_{ki}\oad_{kj}\right)\oa_{ji}\right] \,,
\end{equation}
while the matrix $\Gop$ reads:
\begin{equation}\label{eq:gop}
  \Gop=
  \left(\begin{array}{cccc}
     1 & -\oad_{2,1} & \cdots & -\oad_{n,1} \\
     \vdots & \ddots & \ddots & \vdots \\
     0 & \cdots & 1 & -\oad_{n,n-1} \\
     0 & \cdots & 0 & 1 \\
  \end{array} \right)^{-1} \,.
\end{equation}

\medskip
\noindent
As explained below, the factorisation formula~\eqref{fullfac} can be lifted to the level of transfer matrices.
To this end, let us first define single-index $Q$-operators $\{Q_i(x)\}_{i=1}^n\subset \End(\BC^n)^{\otimes N}$ via:
\begin{equation}\label{eq:singleindexQ}
  Q_i(x)= \wtr_{D_i}\, \Big( \underbrace{L_i(x)\otimes \dots \otimes L_i(x)}_{N} \Big) \,,
\end{equation}
where we use the \emph{normalized trace} $\wtr_{D_i}$ defined through:
\begin{equation}\label{eq:normalized trace}
  \wtr_{D_i} (X) = \frac{\tr (D_i X)}{\tr (D_i)} \,,
\end{equation}
cf.~\eqref{eq:Fock-trace}.
The twist $D_i$ in~\eqref{eq:singleindexQ} acts only on the Fock space and is defined via:
\begin{equation}\label{eq:twist-bflms}
  D_i\, =
  \prod_{1\leq j<i} \left(\frac{\tau_i}{\tau_j}\right)^{\oa_{ji}\oad_{ij}}
  \prod_{i<j\leq n} \left(\frac{\tau_j}{\tau_i}\right)^{\oad_{ij}\oa_{ji}} \,.
\end{equation}
We note that the action of this twist on the Fock module is uniquely determined (up to a scalar function)
by the following condition:
\begin{equation}\label{eq:DiD agrees}
  DL_i(x)D^{-1} = D_i^{-1} L_i(x)D_i \,,
\end{equation}
with
\begin{equation}\label{eq:D-diag}
  D=\diag(\tau_1,\ldots,\tau_n) \,.
\end{equation}
The relation~\eqref{eq:DiD agrees} ensures the commutativity of $Q_i(x)$ defined via~\eqref{eq:singleindexQ}
and the transfer matrix $T_{(1,0,\ldots,0)}(y)$ of the defining fundamental representation
in the auxiliary space, see Remark~\ref{rem:commutativity}(b).

\medskip
\noindent
Next,  we consider the $N$-fold tensor product on the matrix space of the factorisation formula~\eqref{fullfac}.
Taking the normalized traces of each of the resulting monodromies $L_i(x+\ell_i)\otimes \cdots \otimes L_i(x+\ell_i)$
on the left-hand side of the corresponding relation yields a product of the $Q$-operators. On the right-hand side,
we recover the transfer matrix $T^+_\lambda(x)$ multiplied by the inverse of the character
\begin{equation}
  \ch_\lambda^+ = \tr \prod_{1\leq i\leq n}\tau_i^{\CE_{ii}} =
  \prod_{1\leq i\leq n} \tau_i^{\lambda_i}\prod_{1\leq i <j\leq n}\frac{\tau_i}{\tau_i-\tau_j} =
  \prod_{1\leq i\leq n} \tau_i^{\ell_i}\prod_{1\leq i <j\leq n}\frac{1}{\tau_j^{-1}-\tau_i^{-1}} \,.
\end{equation}

\medskip

\begin{Rem}\label{rem:details for A-factor}
The above computation requires a few steps. First, let us note the following relation among
the twists of the $Q$-operators and the transfer matrices:
\begin{equation}\label{eq:TQtwist}
  \prod_{i=1}^n D_i \ =
  \prod_{1\leq i \leq n} \tau_i^{\CE_{ii}-\lambda_i}
  \prod_{1\leq j <i\leq n} \left(\frac{\tau_j}{\tau_i}\right)^{-\oa_{ji}\oad_{ij}} \,.
\end{equation}
Second, we note that apart from its diagonal the matrix $\Gop$ in \eqref{eq:gop} only contains
creation operators which do not contribute to the trace. Finally, we need the commutativity~\eqref{eq:comSD}
established below.
\end{Rem}

\medskip
\noindent
Thus, we finally obtain the factorisation formula for $T^+_\lambda(x)$, cf~\eqref{eq:Verma-through-Q}:
\begin{equation}\label{eq:details-fact}
  T^+_\lambda(x)=\ch_\lambda^+ \cdot\,  Q_1(x+\ell_1)Q_2(x+\ell_2) \cdots Q_n(x+\ell_n)\,.
\end{equation}

\medskip
\noindent
Combining~\eqref{eq:details-fact} with the BGG-relation \eqref{eq:A-transfer-relation}
and evoking the Vandermonde determinant:
\begin{equation}
  \prod_{1\leq i<j\leq n} \Big(\tau_j^{-1}-\tau_i^{-1} \Big) =
  \det \left\Vert \tau_i ^{-j+1} \right\Vert_{1\leq i,j\leq n} \,,
\end{equation}
we get the determinant formula for $T_\lambda(x)$ in type $A$, cf.~\eqref{eq:det-intro}:

\medskip

\begin{Thm}\label{thm:baby-det}
For $\lambda\in P^+$, we have:
\begin{equation}\label{eq:Tdet}
  T_\lambda(x) =
  \frac{\det \left\Vert \tau_i^{\ell_j} Q_i(x+\ell_j) \right\Vert_{1\leq i,j\leq n}}
       {\det \left\Vert \tau_i^{-j+1} \right\Vert_{1\leq i,j\leq n}} \,.
\end{equation}
\end{Thm}

\medskip

\begin{Rem}\label{rem:our-vs-bflms first}
The \emph{partonic} Lax matrices $\mathsf{L}_i(x)$ of~\cite[(1.16)]{Bazhanov:2010jq} are related to
ours~\eqref{eq:partonic-A} via:
\begin{equation}\label{eq:L-vs-BFLMS}
  \mathsf{L}_i(x)=L_i\left(x-\sfrac{n-1}{2}\right)
\end{equation}
upon the following identification of the oscillators $\{\ob^\dagger_{j,i},\ob_{i,j}\}_{1\leq i\ne j\leq n}$
of~\emph{loc.cit}.\ with ours:
\begin{equation}\label{eq:v-vs-a osc}
  \ob^\dagger_{i,j}=
  \begin{cases}
    \oad_{i,j} &\quad\text{for}\quad i<j \\
    \oa_{j,i} &\quad\text{for}\quad i>j \\
  \end{cases} \,, \qquad
  \ob_{i,j}=
  \begin{cases}
    -\oad_{j,i} &\quad\text{for}\quad i<j\\
    \oa_{i,j} &\quad\text{for}\quad i>j\\
  \end{cases} \,.
\end{equation}
To continue this comparison, we note that the $Q$-operators $\{\mathsf{Q}_i(x)\}_{i=1}^n$ of~\cite{Bazhanov:2010jq}
slightly differ from ours. Explicitly, to simplify the functional relations an extra factor $\tau_i^{x}$
has been introduced in~\cite[(4.23)]{Bazhanov:2010jq}, so that the $Q$-operators of~\cite{Bazhanov:2010jq}
are related to ours~\eqref{eq:singleindexQ} via:
\begin{equation}\label{eq:Q-comparison}
  \mathsf{Q}_i(x) = \tau_i^x \cdot Q_i\left(x-\sfrac{n-1}{2}\right) \,.
\end{equation}
Here, the exponential twist parameters of \cite{Bazhanov:2010jq}  are related to our conventions simply through
\begin{equation}\label{eq:bflms-vs-ous twists}
  \tau_a=e^{{\rm i}\Phi_a} \,, \qquad 1\leq a\leq n \,.
\end{equation}
Likewise, the $T$-operators $\mathsf{T}^+_\lambda(x)$ and $\mathsf{T}_\lambda(x)$
of~\cite[(4.15, 4.16)]{Bazhanov:2010jq} are related to ours via:
\begin{equation}\label{eq:T-comparison}
  \mathsf{T}^+_\lambda(x) = \prod_{i=1}^n \tau_i^{x} \cdot T^+_{\lambda}(x) \,, \qquad
  \mathsf{T}_\lambda(x) = \prod_{i=1}^n \tau_i^{x} \cdot T_{\lambda}(x) \,,
\end{equation}
thus differing by an overall factor $\prod_{i=1}^n \tau_i^x$. However, we note that
the constraint $\sum_{a=1}^n \Phi_a=0$ was imposed in~\cite[(4.4)]{Bazhanov:2010jq},
thus translating into $\prod_{i=1}^n \tau_i=1$, see~\eqref{eq:bflms-vs-ous twists}.
Finally, our determinant formula~\eqref{eq:Tdet} is equivalent to the determinant formula
of~\cite[(5.10)]{Bazhanov:2010jq}, see~\eqref{eq:det-intro}:
\begin{equation}\label{eq:Tdet-comparison}
  \Delta_{\{1,\ldots,n\}}\cdot \mathsf{T}_\lambda(x) =
  \det \left\Vert \mathsf{Q}_i(x+\lambda_j') \right\Vert_{1\leq i,j\leq n} \,,
\end{equation}
where
\begin{equation}
  \Delta_{\{1,\ldots,n\}} \, = \det \left\Vert \tau_i ^{-j+1} \right\Vert_{1\leq i,j\leq n} =
  \det \left\Vert \tau_i ^{-j+1} \right\Vert_{1\leq i,j\leq n} \cdot\, \prod_{i=1}^n \tau_i^{\frac{n-1}{2}}\, =
  \prod_{1\leq i<j\leq n}\frac{\tau_i-\tau_j}{\sqrt{\tau_i\tau_j}}
\end{equation}
and
\begin{equation}
  \lambda_i' = \ell_i + \frac{n-1}{2} = \lambda_i + \frac{n+1-2i}{2} \,, \qquad 1\leq i\leq n \,,
\end{equation}
cf.~\cite[(3.19, 5.3)]{Bazhanov:2010jq}. Indeed, identifying our twist parameters $\tau_a$ with
the $\Phi_a$ of~\cite{Bazhanov:2010jq} via~\eqref{eq:bflms-vs-ous twists}, we get precisely
the formulas of~\cite{Bazhanov:2010jq}, due to:
\begin{equation}
  \sqrt{\frac{\tau_a}{\tau_b}} - \sqrt{\frac{\tau_b}{\tau_a}} =
  2{\rm i}\sin\left( \frac{\Phi_a-\Phi_b}{2} \right) \,.
\end{equation}
\end{Rem}

\medskip

\begin{Rem}\label{rem:SDcomFull}
An essential step used in the derivation of~\eqref{eq:details-fact} is the following commutativity:
\begin{equation}\label{eq:comSD}
  \Big[ \bfS,\prod_{i=1}^n D_i \Big] = 0 \,.
\end{equation}
As the proof of this result was missing in~\cite{Bazhanov:2010jq}, we provide the corresponding argument below.
To this end, it is more convenient to switch to the oscillators $(\ob_{i,j},\obdag_{j,i})_{i\ne j}$
of~\cite{Bazhanov:2010jq}, related to our $(\oa_{i,j},\oad_{j,i})_{i\ne j}$ via~\eqref{eq:v-vs-a osc}.
With this choice of conventions, the twist $D_i$ of~\eqref{eq:twist-bflms} reads:
\begin{equation}
  D_i\, = \prod_{j\ne i} \left(\frac{\tau_j}{\tau_i}\right)^{\obdag_{ij}\ob_{ji}} \,,
\end{equation}
so that the product of twists in \eqref{eq:comSD} simplifies to:
\begin{equation}
  \prod_{i=1}^n D_i = \prod_{i=1}^n \tau_i^{\bN_i} \,,
\end{equation}
where $\bN_i$ are defined via:
\begin{equation}\label{eq:numbopi}
  \bN_i=\sum_{j\neq i} \left(\obdag_{ji}\ob_{ij}-\obdag_{ij}\ob_{ji}\right) \,.
\end{equation}
Likewise, the similarity transformation $\bfS_i$ of~\eqref{eq:elementary similarity} reads:
\begin{equation}\label{eq:defSb}
  \bfS_i =
  \exp \left[\sum_{1\leq j<i} \obdag_{ji}\obdag_{ij} \ + \sum_{1\leq j<k<i} \obdag_{ki}\ob_{jk}\obdag_{ij} \right] \,,
\end{equation}
which can be further simplified to:
\begin{equation}
  \bfS_i =
  \prod_{1\leq j<i} \exp \left[ \obdag_{ij}\obdag_{ji} \right]
  \prod_{1\leq j<k<i} \exp \left[ \obdag_{ij}\ob_{jk}\obdag_{ki} \right] \,,
\end{equation}
since the oscillators $\obdag_{\bullet,\bullet}$ in~\eqref{eq:defSb} always have one of $\bullet$ equal to $i$,
while the $\ob_{\bullet,\bullet}$'s never have. Therefore, to prove~\eqref{eq:comSD} it suffices to verify that:
\begin{equation}\label{eq:N-commutativity}
  [\bN_i,\obdag_{k \ell}\obdag_{\ell k}]=0 \,, \qquad [\bN_i,\obdag_{k \ell}\ob_{\ell m}\obdag_{mk}]=0
\end{equation}
for any $1\leq i\leq n$ and $1\leq \ell<m<k\leq n$. These equalities follow immediately from the fact that $[\bN_i,-]$
acts on a given state by counting the number of creation and annihilation operators via:
\begin{equation}\label{eq:N-action}
  \sum_{j\ne i} \Big[(\#\obdag_{ji} - \#\ob_{ij}) - ( \#\obdag_{ij} - \#\ob_{ji}) \Big] \,.
\end{equation}
\end{Rem}

\medskip

\begin{Rem}\label{rem:commutativity}
Let us conclude with the commutativity of
$\{T^+_\lambda(x)\}_{\lambda\in \BC^n}$, $\{T_\mu(x)\}_{\mu\in P^+}$, $\{Q_i(x)\}_{i=1}^n$.

\medskip
\noindent
(a) The commutativity
\begin{equation}\label{eq:TT-commutativity inf}
  [T^+_\lambda(x),T^+_\mu(y)]=0 \,, \qquad \lambda,\mu\in \BC^n \,,
\end{equation}
is a direct consequence of their realization through the universal $R$-matrix as outlined in
Section~\ref{ssec R-matrix}. To this end, we also would like to point out that an explicit form of the $R$-matrix
intertwining the monodromies of these infinite-dimensional transfer matrices was obtained in \cite[(2.28)]{DerkachovFac}.
By a direct application of Theorem~\ref{thm:baby-bgg} (or, alternatively, the construction of Section~\ref{ssec R-matrix}),
we also get:
\begin{equation}\label{eq:TT-commutativity fin}
  [T^+_\lambda(x),T_\mu(y)]=0 \,, \qquad [T_\nu(x),T_\mu(y)]=0 \,,
  \qquad \lambda\in \BC^n \,,\, \mu,\nu \in P^+ \,.
\end{equation}

\medskip
\noindent
(b) The commutativity of the above transfer matrices with all single-index $Q$-operators:
\begin{equation}\label{eq:QT-general commutativity}
  [T^+_\lambda(x),Q_i(y)]=0 \,, \qquad [T_\mu(x),Q_i(y)]=0 \,, \qquad
  1\leq i\leq n \,,\, \lambda\in \BC^n \,,\, \mu\in P^+ \,,
\end{equation}
follows from the $R$-matrix of~\cite[(2.15)]{Frassekhigher} intertwining the monodromies of all
non-degenerate Lax matrices and the degenerate one of~\eqref{eq:partonic-A}. Nonetheless, let us present
a self-contained proof~of
\begin{equation}\label{eq:QT-commutativity}
  [Q_i(x),T_{(1,0,\ldots,0)}(y)]=0 \,, \qquad 1\leq i\leq n \,,
\end{equation}
in order to emphasize the role of the relation~\eqref{eq:DiD agrees}, as promised after~\eqref{eq:D-diag}.
To this end, combining the RTT relation~\eqref{eq:RTT} for $L_i(x)$ of~\eqref{eq:partonic-A}
with the identity $R(x)R(-x)=(1-x^2)\id_n$, we obtain:
\begin{equation}
  L_i(x-y)\otimes L_i(x)R(y)=R(y)L_i(x)\otimes L_i(x-y) \,.
\end{equation}
Building further the monodromy matrices
\begin{equation}
  M_i(x)=\underbrace{L_{i}(x)\otimes \cdots \otimes L_{i}(x)}_{N} \,, \qquad
  M_a(y)=R_{a1}(y) \cdots R_{aN}(y) \,,
\end{equation}
we end up with the following equation:
\begin{equation}\label{eq:LM-YB}
  L_i(x-y)M_i(x)M_a(y)=M_a(y)M_i(x)L_i(x-y) \,,
\end{equation}
where $L_i(x-y)$ acts nontrivially on the auxiliary spaces of these two monodromies:
the monodromy $M_i$ built from the oscillators and the monodromy $M_a$ of the fundamental
transfer matrix with an $n$-dimensional auxiliary space denoted by $a$.
Multiplying~\eqref{eq:LM-YB} by the twists $D_i$ and $D$ on the left
and taking the trace, we end up precisely with~\eqref{eq:QT-commutativity}
(we should note that to move both twists past $L_i(x)$ in the left-hand side of~\eqref{eq:LM-YB},
we use the commutativity $[D_iD,L_i(x)]=0$ of~\eqref{eq:DiD agrees}).

\medskip
\noindent
(c) Finally, to establish the commutativity among the $Q$-operators $Q_i(x)$, it is convenient to realize
them as the \emph{normalized limits} of the transfer matrices $T^+_\lambda(x)$ in which some of the representation
labels tend to infinity. On the level of the corresponding Lax matrices, the relevant limit is:
\begin{multline*}
  \lim_{\st\to\infty}
  \left\{ \diag(\underbrace{\st^{-1},\ldots,\st^{-1}}_{n-1},-1)
         \CL_{(\footnotesize{\underbrace{\st,\ldots,\st}_{n-1}},0)}(x)
         \diag(\underbrace{1,\ldots,1}_{n-1},-1) \right\} = \\
  \left(\begin{BMAT}[5pt]{c:c}{c:c}
    U^{-1}_{n-1} & 0\\
    \am_{n-1} & -x+n-1
  \end{BMAT} \right)
  \left(\begin{BMAT}[5pt]{c:c}{c:c}
    U_{n-1} & -\ap_{n-1} \\
    0 & 1
  \end{BMAT} \right)
  \left(\begin{BMAT}[5pt]{c:c}{c:c}
    \ID_{n-1} & 0 \\
    0 & -1
  \end{BMAT} \right) = \\
  \left(\begin{BMAT}[5pt]{c:c}{c:c}
    \id_{n-1} & U^{-1}_{n-1}\ap_{n-1} \\
    \am_{n-1}U_{n-1} & x-n+1+\am_{n-1}\ap_{n-1}
  \end{BMAT} \right) =
  \bfS L'_n(x-n+1) \bfS^{-1} \,,
\end{multline*}
where $\am_{n-1}=(\oa_{n,1},\ldots,\oa_{n,n-1})$, $\ap_{n-1}=(\oad_{1,n},\ldots,\oad_{n-1,n})^T$,
$U_i$ denotes the upper-left $i\times i$ block of the matrix $U=U_n$ from~\eqref{eq:UD},
and $L'_n(x)$ is the $\gl_n$-type Lax matrix obtained from $L_n(x)$ of~\eqref{eq:partonic-A}
by relabelling $\oa_{k,n}\mapsto \oa_{n,k}$ and $\oad_{n,k}\mapsto \oad_{k,n}$ for $1\leq k<n$.
Here, the similarity transformation $\bfS$ is defined via:
\begin{equation}\label{eq:tot-similarity}
  \bfS=\bfS_{n-1}\cdots\bfS_{1}\,, \qquad
  \bfS_k=\exp\left[\oad_{kn}\sum_{i=1}^{k-1} \oa_{ni}\oad_{ik}\right] \,.
\end{equation}
This implies $[Q_n(x),Q_n(y)]=0$ and combining this with the action of the Weyl group, we obtain:
\begin{equation}\label{eq:qq-com 1}
  [Q_i(x),Q_i(y)]=0 \,, \qquad 1\leq i\leq n \,.
\end{equation}
Finally, the proof of
\begin{equation}\label{eq:qq-com 2}
  [Q_i(x),Q_j(y)]=0 \,, \qquad 1\leq i\ne j\leq n \,,
\end{equation}
follows immediately from the factorisation of the $X$-operators from~\cite[\S4-5]{Bazhanov:2010jq},
see~\eqref{eq:X-factorisation}.
\end{Rem}

$\ $


\section{A-type: rectangular}\label{sec A-rectangular}

Let us now consider $A$-type Lax matrices corresponding to \emph{rectangular} representations, i.e.\ those whose
highest weight is a multiple of a fundamental weight. These  play a special role because their transfer matrices
are related by the Hirota equation~\cite{kns}. But also, as we will see below, they will be relevant to
our approach to the study of transfer matrices in other classical types.

\medskip


\subsection{Oscillator realization in type A (parabolic Verma)}
$\ $

For any $n\in \BZ_{\geq 2}$ and $1\leq a \leq n-1$, let $\CA$ denote the oscillator algebra generated
by $a(n-a)$ pairs of oscillators $\{(\oa_{j,i},\oad_{i,j})\}_{1\leq i \leq a}^{a<j\leq n}$ subject to
the defining relations~\eqref{eq:oscillator relations}:
\begin{equation}\label{eq:osc-A-rect}
  \CA = \BC \Big\langle \oa_{j,i} \, , \, \oad_{i,j} \Big\rangle_{1\leq i \leq a}^{a<j\leq n} \, \Big/ \,
  \eqref{eq:oscillator relations} \,.
\end{equation}
Following~\cite{Bazhanov:2010jq}, let us consider the $\gl_n$-type $\CA[x]$-valued Lax matrix
(depending on $\st\in \BC$):
\begin{equation}\label{eq:LaxRec}
  \CL_a(x)=
  \left(\begin{BMAT}[5pt]{c:c}{c:c}
    (x+\st)\ID_a-\ap\am & -\ap(\st+a-\am\ap) \\
    -\am & (x-a)\ID_{n-a}+\am\ap
  \end{BMAT} \right)
\end{equation}
with the blocks $\ap\in \mathrm{Mat}_{a\times(n-a)}(\CA)$ and $\am\in \mathrm{Mat}_{(n-a)\times a}(\CA)$
encoding all the generators via:
\begin{equation}\label{eq:ApAmA}
  \ap=
  \left(\begin{array}{cccc}
    \oad_{1,a+1} & \cdots & \oad_{1,n} \\
    \vdots & \iddots & \vdots \\
    \oad_{a,a+1} & \cdots & \oad_{a,n} \\
  \end{array}\right) \,, \qquad
  \am=
  \left(\begin{array}{cccc}
    \oa_{a+1,1} & \cdots & \oa_{a+1,a} \\
    \vdots & \iddots & \vdots \\
    \oa_{n,1} & \cdots & \oa_{n,a} \\
  \end{array}\right) \,.
\end{equation}
Writing~\eqref{eq:LaxRec} in the form
\begin{equation}
  \CL_a(x)=x\ID_n + \sum_{i,j=1}^{n} e_{ij}\CE_{ji} \,,
\end{equation}
we note that $\{\CE_{ij}\}_{i,j=1}^n$ satisfy the $\gl_n$ commutation relations~\eqref{eq:glncom},
as a consequence of RTT~\eqref{eq:RTT}.

\medskip
\noindent
As before, let $\Fock$ denote the Fock module of $\CA$,
generated by the Fock vacuum $|0\rangle \in \Fock$ satisfying:
\begin{equation}
  \oa_{j,i} |0\rangle = 0 \,, \qquad 1\leq i\leq a < j\leq n \,.
\end{equation}
Then, the Fock vacuum $|0\rangle$ is a highest weight state of the resulting $\gl_n$-action:
\begin{equation}
  \CE_{ij} |0\rangle = 0 \,, \qquad 1\leq i<j\leq n \,,
\end{equation}
with the highest weight
  $\lambda = \st \omega_a = (\underbrace{\st,\ldots,\st}_{a},\underbrace{0,\ldots,0}_{n-a})$,
that is:
\begin{equation}
  \CE_{kk} |0\rangle = \st \delta_{k\leq a} |0\rangle \,, \qquad 1\leq k\leq n \,.
\end{equation}
We can now identify the resulting $\gl_n$-modules $\Fock$ with those featuring in the Introduction:

\medskip

\begin{Lem}\label{lem:A-Fock as par-Verma}
There is a $\gl_n$-module isomorphism:
\begin{equation}\label{eq:A-Fock-as-verma}
  \Fock \simeq \left(M^{\fp_{\{1,\ldots,n-1\}\setminus\{a\}}}_{\st\omega_a}\right)^\vee,
\end{equation}
identifying $\Fock$ with the restricted dual~\eqref{eq:restricted-dual}
of the parabolic Verma module~\eqref{eq:parabolic def}.
\end{Lem}

\medskip

\begin{proof}
Since the restricted dual $\Fock^\vee$ has the same highest weight and the character as
$M^{\fp_{\{1,\ldots,n-1\}\setminus\{a\}}}_{\st\omega_a}$, it suffices to prove that it is also
a highest weight $\gl_n$-module, i.e.\ generated by its highest weight vector $\langle 0|$. To this end,
we note that $\langle \vec{m}|$ is in fact a non-zero multiple of the image of $\langle 0|$ under the
order-independent product $\prod_{1\leq i\leq a}^{a<j\leq n} (\CE^*_{ij})^{m_{i,j}}$ for any
$\vec{m}=(m_{i,j})_{1\leq i\leq a<j\leq n}\in \BN^{a(n-a)}$, cf.\ the proof of Lemma~\ref{lem:A-Fock as Verma}.
\end{proof}

\medskip
\noindent
Combining this with the determinant formula of~\cite{j} and the isomorphism~\eqref{eq:self-dual}, we obtain:

\medskip

\begin{Cor}\label{cor:t-special A}
(a) For $\st\notin \BZ_{\geq 2-n}$, the $\gl_n$-module $\Fock$ is irreducible (thus, is generated by $|0\rangle$).

\medskip
\noindent
(b) For $\st\in \BN$, the Fock vacuum $|0\rangle$ generates an irreducible finite-dimensional
$\gl_n$-module $L_{\st\omega_a}$.
\end{Cor}

\medskip


\subsection{More oscillator realizations in type A via underlying symmetries}
\label{ssec more-A}
$\ $

Since the $R$-matrix~\eqref{eq:R-matrix-A} is invariant, cf.~\eqref{eq:R-inv}, with respect to the natural action
of the symmetric group $S_n$ (via the standard embedding $S_n\hookrightarrow GL_n$), we can generate more solutions
to the RTT relation~\eqref{eq:RTT} from the Lax matrix~\eqref{eq:LaxRec} above by simultaneously permuting its rows
and columns. We shall further apply the (unique) particle-hole automorphism of $\CA$ to insure that the Fock vacuum
$|0\rangle \in \Fock$ remains to be a $\gl_n$ highest weight state. To this end, consider the indexing set
\begin{equation}\label{eq:CS-set}
  \CS_a=\Big\{I\subseteq \{1,\ldots,n\} \, | \, \#I=a \Big\} \,.
\end{equation}
Then, we construct the following explicit $\gl_n$-type $\CA[x]$-valued Lax matrices:
\begin{equation}\label{eq:symmetry-conjugate-A}
  \CL_I(x)=\left. B_I \CL_a(x) B_I^{-1} \right|_{p.h.} =  x\ID_n + \sum_{i,j=1}^{n} e_{ij}\CE_{ji}^I \,,
  \qquad \forall\, I\in \CS_a \,,
\end{equation}
with the similarity matrix $B_I$ and the particle-hole transformation (denoted \emph{p.h.}) described below.

\medskip
\noindent
Let $J=\bar{I}$ denote the complement of the subset $I$:
\begin{equation}\label{eq:set-complement}
  J=\bar{I}=\{1,\ldots,n\} \setminus I \,,
\end{equation}
and let us order the elements of $I$ and $J$ in the increasing order:
\begin{equation}\label{eq:ordering sets}
\begin{split}
  & I=\{i_1,i_2,\ldots,i_a\} \,, \qquad 1\leq i_1<i_2<\dots<i_a\leq n \,, \\
  & J=\{j_1,j_2,\ldots,j_{n-a}\} \,, \qquad 1\leq j_1<j_2<\dots<j_{n-a}\leq n \,.
\end{split}
\end{equation}
We consider the following permutation $\sigma_I$ of the set $\{1,\ldots,n\}$:
\begin{equation}\label{eq:special permutation}
  \sigma_I(c)=
  \begin{cases}
    i_c & \text{for } 1\leq c\leq a \\
    j_{c-a} & \text{for } a<c\leq n
  \end{cases} \,.
\end{equation}

\medskip

\begin{Rem}\label{rem:sigma properies}
Consider the natural transitive action of the symmetric group $S_n$ on the set $\CS_a$~\eqref{eq:CS-set}.
Then, the stabilizer of $\{1,\ldots,a\}\in \CS_a$ is the subgroup $S_{a}\times S_{n-a}\subset S_n$ of those
permutations that map $\{1,\ldots,a\}\mapsto \{1,\ldots,a\}$ and $\{a+1,\ldots,n\}\mapsto \{a+1,\ldots,n\}$.
This gives rise to a set bijection
\begin{equation}\label{eq:set-bij}
  \pi\colon S_n/(S_a\times S_{n-a}) \ \iso \ \CS_a
\end{equation}
satisfying the property $\sigma_I\in \pi^{-1}(I)$ with $\sigma_I$ defined in~\eqref{eq:special permutation}.
Furthermore, $\sigma_I$ can be characterized as the shortest representative of the corresponding left coset
$\pi^{-1}(I)$, and its length is given by:
\begin{equation}\label{eq:sigma-len}
  l(\sigma_I) = \# \left\{ (k,\ell) \in I\times \bar{I} \, | \, k>\ell \right\} \,.
\end{equation}
\end{Rem}

\medskip
\noindent
Then, we define $B_I$ in~\eqref{eq:symmetry-conjugate-A} as the permutation matrix corresponding to $\sigma_I\in S_n$:
\begin{equation}\label{eq:permutation-matrix}
  B_I=\sum_{i=1}^{n} e_{\sigma_I(i),i} \,.
\end{equation}
We note that its inverse coincides with its transpose: $B_I^{-1} = B_I^T = \sum_{i=1}^{n} e_{i,\sigma_I(i)}$.

\medskip
\noindent
To determine the particle-hole transformation used in~\eqref{eq:symmetry-conjugate-A} so as to preserve
the $\gl_n$ highest weight state condition, let us check where the $\oa_{\cdot,\cdot}$-generators from
the lower-left block of~\eqref{eq:LaxRec} are moved to under the conjugation by the matrix $B_I$
of~\eqref{eq:permutation-matrix}. Explicitly, the oscillator $-\oa_{j,i}\ (i\leq a<j)$ is located above
the main diagonal of $B_I \CL_a(x) B_I^{-1}$ if and only if $\sigma_I(j)<\sigma_I(i)$, due to the equality
\begin{equation}\label{eq:conj matr els}
  \Big(B_I \CL_a(x) B_I^{-1}\Big)_{\sigma_I(j),\sigma_I(i)} = \, ({\CL_a(x)})_{ji} \,, \qquad 1\leq i,j\leq n \,.
\end{equation}
Therefore, we consider the following particle-hole automorphism of $\CA$ in~\eqref{eq:symmetry-conjugate-A}:
\begin{equation}\label{eqn:A-ph}
  \oad_{i,j}\mapsto -\oa_{j,i}\,, \quad \oa_{j,i}\mapsto \oad_{i,j} \quad
  \mathrm{for}\quad 1\leq i\leq a<j\leq n\quad \mathrm{such\ that} \quad \sigma_I(j)<\sigma_I(i)\,.
\end{equation}

\medskip
\noindent
The resulting matrix elements $\{\CE_{ij}^I\}_{i,j=1}^n\subset \CA$ of~\eqref{eq:symmetry-conjugate-A} still satisfy
the $\gl_n$ commutation relations~\eqref{eq:glncom}, as follows from the RTT relation~\eqref{eq:RTT}.
This makes the Fock module $\Fock$ into a $\gl_n$-module, denoted by $M^+_{I,t}$. The Fock vacuum $
|0\rangle\in M^+_{I,t}$ is easily seen to be a $\gl_n$ highest weight~state:
\begin{equation}
  \CE_{ij}^I |0\rangle = 0 \,, \qquad   1\leq i<j\leq n \,.
\end{equation}
To compute its highest weight, we note that
\begin{equation}\label{eq:diagonal-A}
\begin{split}
  & \left(B_I \CL_a(x )B_I^{-1} \right)_{\sigma_I(i),\sigma_I(i)} =
    x + \st \, - \sum_{a<j\leq n} \oad_{ij}\oa_{ji}
    \,, \qquad 1\leq i\leq a \,, \\
  & \left(B_I \CL_a(x) B_I^{-1} \right)_{\sigma_I(j),\sigma_I(j)} =
    x - a \, + \sum_{1\leq i\leq a} \oa_{ji}\oad_{ij} \, = x \, + \sum_{1\leq i\leq a} \oad_{ij}\oa_{ji}
    \,, \qquad a < j\leq n\,,
\end{split}
\end{equation}
which after implementing the particle-hole transformation~\eqref{eqn:A-ph} gives:
\begin{equation}\label{eq:Cartan I}
\begin{split}
  & \CE_{\sigma_I(i),\sigma_I(i)}^I |0\rangle =
    \Big(\st+\#\{a<j\leq n \,|\, \sigma_I(j)<\sigma_I(i)\}\Big) |0\rangle \,, \qquad 1\leq i\leq a \,, \\
  & \CE_{\sigma_I(j),\sigma_I(j)}^I |0\rangle =
    \Big(-\#\{1\leq i\leq a \,|\, \sigma_I(j)<\sigma_I(i)\}\Big) |0\rangle \,, \qquad a < j\leq n \,.
\end{split}
\end{equation}
Evoking~\eqref{eq:special permutation} and the particular ordering~\eqref{eq:ordering sets}, we find:
\begin{equation}\label{eq:simple equalities}
\begin{split}
  & \#\Big\{a<j\leq n \,|\, \sigma_I(j)<\sigma_I(i)\Big\}=\sigma_I(i)-i \,, \qquad  1\leq i\leq a \,, \\
  & -\#\Big\{1\leq i\leq a \,|\, \sigma_I(j)<\sigma_I(i)\Big\}=\sigma_I(j)-j \,, \qquad a< j \leq n \,,
\end{split}
\end{equation}
so that:
\begin{equation}\label{eq:htwtA1}
  \CE_{kk}^I |0\rangle = \left(\delta_{\sigma_I^{-1}(k)\leq a}\st + k - \sigma_I^{-1}(k)\right) |0\rangle
  \,, \qquad 1\leq k\leq n \,.
\end{equation}
Hence, the highest weight of the Fock vacuum $|0\rangle\in M^+_{I,t}$ is precisely
\begin{equation*}
  \sigma_I\cdot t\omega_a \,,
\end{equation*}
see~\eqref{eq:dot action}, or equivalently:
\begin{equation}\label{eq:htwtA2}
  \CE_{kk}^I |0\rangle =
  \begin{cases}
    \left(\st+\#\left\{j\notin I \,|\, j<k\right\}\right) |0\rangle & \text{for } k\in I \\
    \left(-\#\left\{i\in I \,|\, i>k\right\}\right) |0\rangle  & \text{for } k\notin I
  \end{cases}\,.
\end{equation}

\medskip
\noindent
Thus, the highest weight of $|0\rangle\in M^+_{I,\st}$ is the same as the highest weight of our key modules
$M'_{\sigma_I\, \cdot \, \st\omega_a}$, introduced in~\eqref{eq:hard modules}, with respect to the standard parabolic
subalgebra $\fp_S\subset \gl_n$ corresponding to $S=\{1,\ldots,n-1\}\setminus\{a\}$. Furthermore, the module
$M'_{\sigma_I\, \cdot \, \st\omega_a}$ has the same character as $M^+_{I,\st}$ (according to Lemma~\ref{lem:M-char})
and is irreducible for $\st\notin \BZ$ (as follows from~\cite{j}). Therefore, we obtain:

\medskip

\begin{Prop}\label{lem:A-Fock as M-mod}
For any $I\in \CS_a$ and $\st\notin \BZ$, we have $\gl_n$-module isomorphisms:
\begin{equation}\label{eq:A-Fock-as-M}
  M^+_{I,\st} \simeq M'_{\sigma_I\, \cdot \, \st\omega_a} \simeq \left(M'_{\sigma_I\, \cdot \, \st\omega_a}\right)^\vee \,.
\end{equation}
\end{Prop}

\medskip

\begin{Rem}
Let us point out the key difference between Proposition~\ref{lem:A-Fock as M-mod} and Lemma~\ref{lem:A-Fock as par-Verma}:

\medskip
\noindent
(a) For $I=\{1,\ldots,a\}\in \CS_a$, we actually have $M^+_{I,\st} \simeq (M'_{\sigma_I\, \cdot \, \st\omega_a})^\vee$
for any $\st\in \BC$, due to Lemma~\ref{lem:A-Fock as par-Verma}.

\medskip
\noindent
(b) Likewise, for $I=\{n-a+1,\ldots,n\}\in \CS_a$, we have $M^+_{I,\st} \simeq M'_{\sigma_I\, \cdot \, \st\omega_a}$
for any $\st\in \BC$.

\medskip
\noindent
(c) For other $I\in \CS_a$, $M^+_{I,\st}$ is \underline{not isomorphic} to either of
$M'_{\sigma_I\, \cdot \, \st\omega_a}$ or $(M'_{\sigma_I\, \cdot \, \st\omega_a})^\vee$ at certain $\st\in \BZ$
(but is expected to be isomorphic to one of the twisted Verma modules in the sense of~\cite{al}).
\end{Rem}

\medskip

\begin{Rem}\label{rem:A-details}
(a) The Weyl group $W_{\fl}$ of the Levi subalgebra $\fl$ of $\fp_S$ is precisely $S_a\times S_{n-a}\subset S_n$.

\medskip
\noindent
(b) We indeed have $\sigma_I\in {}^{\fl}W$ in the notation~(\ref{eq:shortest left},~\ref{eq:shortest left equiv}),
due to Remark~\ref{rem:sigma properies}.

\medskip
\noindent
(c) For any other permutation $\sigma'\in \sigma_I (S_a\times S_{n-a})$, conjugating $\CL_a(x)$ with
$B'_I=\sum_{i=1}^n e_{\sigma'(i),i}$ and applying the corresponding particle-hole transformation will
produce an isomorphic $\gl_n$-module.
\end{Rem}

\medskip
\noindent
Evoking the above bijection $\CS_a\ni I \leftrightarrow \sigma_I\in {}^{\fl}W$, see~\eqref{eq:special permutation}
and Remarks~\ref{rem:sigma properies},~\ref{rem:A-details}(b), we define:
\begin{equation}\label{eq:AM-modified}
  M^\vee_{I,\st}=\left(M'_{\sigma_I\, \cdot \, \st\omega_a}\right)^\vee \,, \qquad \forall\, \st\in \BC \,.
\end{equation}
Then, Proposition~\ref{lem:A-Fock as M-mod} can be recast as the isomorphism of the following $\gl_n$-modules:
\begin{equation}\label{A-generic-coincidence}
  M^+_{I,\st}\simeq M^\vee_{I,\st} \,, \qquad \forall \, \st\in \BC \setminus \BZ \,.
\end{equation}
\noindent

\medskip
\noindent
For $I\in \CS_a$, we also define its \emph{length} $l(I)$ as the length of the corresponding $\sigma_I\in S_n$, see~\eqref{eq:sigma-len}:
\begin{equation}\label{eq:sigma-length}
  l(I) = l(\sigma_I) = \# \left\{ (k,\ell) \in I\times \bar{I} \, | \, k>\ell \right\} \,.
\end{equation}

\medskip


\subsection{Type A transfer matrices}\label{ssec A-transfer rect}
$\ $

Recall the notion of transfer matrices $\{T_W(x)\}_{W\in \mathrm{Rep}\, Y(\gl_n)}$,
as discussed in Subsection~\ref{ssec R-matrix}.
In particular, we shall consider the following explicit infinite-dimensional transfer matrices:
\begin{equation}\label{eq:Tpgln}
  T_{I,\st}^+(x)=\tr \prod_{i=1}^n \tau_i^{\CE^I_{ii}} \underbrace{\CL_I(x) \otimes \dots \otimes \CL_I(x)}_{N} \,,
  \qquad \forall\, I\in \CS_a \, , \, \st\in \BC \,,
\end{equation}
corresponding to $M^+_{I,\st}$. For $\st\in \BN$, we also consider the finite-dimensional transfer matrices
$T_{a,\st}(x)$ corresponding to the modules $L_{\st\omega_a}$ in the auxiliary space:
those are defined similarly to~\eqref{eq:Tpgln}, but with the trace taken over the finite-dimensional submodule
$L_{\st\omega_a}$ of $M^+_{\{1,\ldots,a\},\st}$, see Corollary~\ref{cor:t-special A}(b).

\medskip
\noindent
Using the notation~(\ref{eq:AM-modified},~\ref{eq:sigma-length}), let us recast
the resolution~\eqref{eq:geometric resolution 1}, dual to~\eqref{eq:conjectured resolution 1}, as follows:
\begin{equation}\label{eq:A-rectangular-resolution}
  0\to  L_{\st \omega_a} \to M^\vee_{\{1,\ldots,a\},\st} \to \bigoplus_{I\in \CS_a}^{l(I)=1} M^\vee_{I,\st}
  \to \bigoplus_{I\in \CS_a}^{l(I)=2} M^\vee_{I,\st} \to \dots \to M^\vee_{\{n-a+1,\ldots,n\},\st} \to 0
\end{equation}
for any $\st\in \BN$. Combining this with~\eqref{A-generic-coincidence} and the fact that
\underline{the transfer matrices~\eqref{eq:Tpgln} depend} \underline{continuously on $\st\in \BC$}
(as so do the Lax matrices $\CL_I(x)$), we obtain the key result of this section:

\medskip

\begin{Thm}\label{thm:Main-A}
For $1\leq a<n$ and $\st\in \BN$, we have:
\begin{equation}\label{eq:transfer-A}
  T_{a,\st}(x)\, = \sum_{I\in \CS_a} (-1)^{l(I)}\, T_{I,\st}^+(x) \,.
\end{equation}
\end{Thm}

\medskip
\noindent
The character limit of~\eqref{eq:transfer-A} expresses the character
of the $\gl_n$-modules $\{L_{\st \omega_a}\}_{1\leq a<n}^{\st\in \BN}$ defined as
\begin{equation}\label{eq:A-char-def}
  \ch_{a,\st}=\ch_{a,\st}(\tau_1, \ldots, \tau_n) := \tr_{L_{\st\omega_a}} \prod_{i=1}^n \tau_i^{\CE_{ii}} \,,
\end{equation}
that is the length $N=0$ case of $T_{a,\st}(x)$, via:
\begin{equation}\label{eq:char-A}
  \ch_{a,\st} \, =
  \sum_{I\in \CS_a} (-1)^{l(I)}
  \frac{\prod_{k\in I} \tau_k^{\st+\#\{\ell\notin I|\ell<k\}} \prod_{\ell\notin I}\tau_\ell^{-\#\{k\in I|k>\ell\}}}
       {\prod_{k\in I,\ell\notin I}^{k>\ell}\left(1-\frac{\tau_k}{\tau_\ell}\right)
        \prod_{k\in I,\ell\notin I}^{k<\ell}\left(1-\frac{\tau_\ell}{\tau_k}\right)}
\end{equation}
with the $I$'s summand in the right-hand side of~\eqref{eq:char-A} equal to the character of $M^+_{I,\st}$, up to a sign.

\medskip
\noindent
Let us note right away that formulas~\eqref{eq:transfer-A} and~\eqref{eq:char-A} allow to \underline{analytically continue}
the transfer matrices $T_{a,\st}(x)$ and their particular length $N=0$ case $\ch_{a,\st}$ of~\eqref{eq:A-char-def}
from the discrete set $\st\in \BN$ to the entire complex plane $\st\in \BC$.

\medskip

\begin{Rem}
For the physics' reader who skipped Section~\ref{sec: truncated BGG resolutions}, let us present a concise proof
of~\eqref{eq:char-A}. We shall identify the set $\Delta^+$ of positive roots of $\fg=\gl_n$ with
$\Delta^+=\{\epsilon_i-\epsilon_j|1\leq i<j\leq n\}$, so that $\rho=(\frac{n-1}{2},\frac{n-3}{2},\ldots,\frac{1-n}{2})$
in the basis $\{\epsilon_i\}_{i=1}^n$ and the Weyl group $W$ gets identified with $W\simeq S_n$ (acting by permutations
on the basis $\{\epsilon_i\}_{i=1}^n$). According to the Weyl character formula, we have:
\begin{equation}\label{eq:Weyl A-rect}
  \ch_{L_{\st\omega_a}}=
  \sum_{\sigma\in S_n}
  (-1)^{l(\sigma)}\frac{e^{\sigma(\st\omega_a+\rho)-\rho}}{\prod_{1\leq i<j\leq n}(1-e^{\epsilon_j-\epsilon_i})} \,.
\end{equation}
Following Remark~\ref{rem:A-details}, let us consider the parabolic subalgebra $\fp_{\{1,\ldots,n-1\}\setminus \{a\}}\subset\fg$
with the Levi subalgebra $\fl\simeq \gl_a\oplus \gl_{n-a}$ and the Weyl group $W_{\fl}\simeq S_a\times S_{n-a}\subset S_n$.
We can rewrite~\eqref{eq:Weyl A-rect} as:
\begin{equation}\label{eq:Weyl A-rect doubled}
  \ch_{L_{\st\omega_a}} \, =
  \sum_{\bar{\sigma}\in W/W_{\fl}} \sum_{\tau\in W_{\fl}}
  (-1)^{l(\sigma\tau)}\frac{e^{\sigma\tau(\st\omega_a+\rho)-\rho}}{\prod_{i<j}(1-e^{\epsilon_j-\epsilon_i})} \,,
\end{equation}
where $\sigma\in W$ is a representative of $\bar{\sigma}\in W/W_{\fl}$ (the inner sum is independent of the choice of $\sigma$).
The key step is to simplify the inner sum of~\eqref{eq:Weyl A-rect doubled} using the Weyl denominator formula for $\fl$:
\begin{equation}\label{eq:Weyl denominator A-rect}
  \sum_{\tau \in W_{\fl}} (-1)^{l(\tau)} e^{\tau(\rho_\fl)-\rho_\fl} \, =
  \prod_{\alpha\in \Delta^+_\fl} (1-e^{-\alpha}) \,,
\end{equation}
where
  $\Delta^+_{\fl}=\{\epsilon_i-\epsilon_j\}_{1\leq i<j\leq a} \cup \{\epsilon_i-\epsilon_j\}_{a<i<j\leq n}\subset \Delta^+$
denotes the set of positive roots of $\fl$ and
  $\rho_{\fl}=\frac{1}{2}\sum_{\alpha\in \Delta^+_{\fl}} \alpha =
   (\frac{a-1}{2},\ldots,\frac{1-a}{2},\frac{n-a-1}{2},\ldots,\frac{1+a-n}{2})$.
As $\tau(\rho)-\rho=\tau(\rho_{\fl})-\rho_{\fl}$ for $\tau\in W_{\fl}$, we get:
\begin{equation}
  \sum_{\tau\in W_{\fl}}
  (-1)^{l(\tau)}\frac{e^{\tau(\st\omega_a+\rho)-\rho}}{\prod_{1\leq i<j\leq n}(1-e^{\epsilon_j-\epsilon_i})} =
  \frac{e^{\st\omega_a}}{\prod_{1\leq i\leq a}^{a<j\leq n}(1-e^{\epsilon_j-\epsilon_i})} \,.
\end{equation}
Therefore, the inner sum of~\eqref{eq:Weyl A-rect doubled} corresponding to the trivial left coset $W_\fl\in W/W_\fl$
gives rise to the $I=\{1,\ldots,a\}$'s term of~\eqref{eq:char-A}. Likewise, we claim that any $I$'s term of~\eqref{eq:char-A}
precisely arises from the inner sum of~\eqref{eq:Weyl A-rect doubled} corresponding to the left coset $\sigma_I W_{\fl}$ with
$\sigma_I\in S_n$ of~(\ref{eq:ordering sets},~\ref{eq:special permutation}), which amounts to the proof
of~\eqref{eq:char derivation A-rect} below. To this end, let us apply $\sigma_I$ to both sides
of~\eqref{eq:Weyl denominator A-rect}:
\begin{equation}
  \sum_{\tau\in W_{\fl}} (-1)^{l(\sigma_I\tau)}e^{\sigma_I\tau(\rho)-\rho} \, =\,
  (-1)^{l(\sigma_I)}e^{\sigma_I(\rho)-\rho}
  \prod_{i,j\in I}^{i<j} (1-e^{\epsilon_j-\epsilon_i})
  \prod_{i,j\notin I}^{i<j} (1-e^{\epsilon_j-\epsilon_i}) \,.
\end{equation}
Combining this with the straightforward formula
\begin{equation*}
  \sigma_I(\rho)-\rho=\sum_{k\in I} \#\{\ell\notin I \, |\, \ell<k\}\epsilon_k-
  \sum_{\ell\notin I} \#\{k\in I \, |\, k>\ell\}\epsilon_\ell \,,
\end{equation*}
we obtain the desired equality:
\begin{equation}\label{eq:char derivation A-rect}
  \sum_{\tau\in W_{\fl}}
  (-1)^{l(\sigma_I\tau)}\frac{e^{\sigma_I\tau(\st\omega_a+\rho)-\rho}}{\prod_{i<j}(1-e^{\epsilon_j-\epsilon_i})} = (-1)^{l(I)}
  \frac{\prod_{k\in I} e^{(\st+\#\{\ell\notin I|\ell<k\})\epsilon_k}
        \prod_{\ell\notin I} e^{-\#\{k\in I|k>\ell\}\epsilon_\ell}}
       {\prod_{k\in I,\ell\notin I}^{k>\ell} (1-e^{\epsilon_k-\epsilon_\ell})
        \prod_{k\in I,\ell\notin I}^{k<\ell} (1-e^{\epsilon_\ell-\epsilon_k})} \,.
\end{equation}
This completes our direct proof of~\eqref{eq:char-A}, due to the bijection $\pi$ of~\eqref{eq:set-bij}, see
Remark~\ref{rem:sigma properies}
(cf.~Remark~\ref{rem:math-to-physics} for more details as for perceiving $\ch_{a,\st}$ of~\eqref{eq:A-char-def}
as a specialization of $\ch_{L_{\st\omega_a}}$).
\end{Rem}

$\ $


\section{Resolutions for transfer matrices of C-type}\label{sec C-spinor}

In this section, we generalize the key constructions and results of Section~\ref{sec A-rectangular} to $C$-type.


\subsection{Oscillator realization in type C (parabolic Verma)}\label{ssec nondeg-C}
$\ $

Let $\CA$ denote the oscillator algebra generated by $\frac{r(r+1)}{2}$ pairs of oscillators
$\{(\oa_{j',i},\oad_{i,j'})\}_{1\leq i \leq j \leq r}$ with $r\in \BZ_{\geq 1}$,
cf.\ notation~\eqref{eq:prime-index}, subject to the standard defining relations:
\begin{equation}\label{eq:oscillator relations 2}
  [\oa_{j',i},\oad_{k,\ell'}]=\delta_{i}^{k}\delta_{j}^{\ell} \,, \qquad
  [\oa_{j',i},\oa_{\ell',k}]=0 \,, \qquad
  [\oad_{i,j'},\oad_{k,\ell'}]=0 \,,
\end{equation}
so that
\begin{equation}
  \CA = \BC \Big\langle \oa_{j',i} \, , \, \oad_{i,j'} \Big\rangle_{1\leq i \leq j \leq r} \, \Big/ \,
  \eqref{eq:oscillator relations 2} \,.
\end{equation}
Following~\cite[p.~593]{Frassek:2021ogy}, see also~\cite[\S6.2]{Karakhanyan:2020jtq},
let us consider the $C_r$-type $\CA[x]$-valued Lax matrix:
\begin{equation}\label{eq:oscLaxC}
  \CL(x)=
  \left(\begin{BMAT}[5pt]{c:c}{c:c}
    (x+\st)\ID_r-\ap\am & -\ap(2\st+r+1-\am\ap) \\
    -\am & (x-\st-r-1)\ID_r+\am\ap
  \end{BMAT}\right) \,,
\end{equation}
depending on $\st\in \BC$, with the blocks $\ap,\am\in \mathrm{Mat}_{r\times r}(\CA)$ encoding all the generators via:
\begin{equation}\label{eq:ApAmC}
  \ap=
  \left(\begin{array}{cccc}
    \oad_{1,r'} & \cdots & \oad_{1,2'} & \oad_{1} \\
    \vdots & \iddots & \oad_{2} & \oad_{1,2'} \\
    \oad_{r-1,r'} & \oad_{r-1} & \iddots & \vdots \\
    \oad_{r} & \oad_{r-1,r'} & \cdots & \oad_{1,r'}
  \end{array}\right) \,, \qquad
  \am=
  \left(\begin{array}{cccc}
    \oa_{r',1} & \cdots & \oa_{r',r-1} & \oa_{r} \\
    \vdots & \iddots & \oa_{r-1} & \oa_{r',r-1} \\
    \oa_{2',1} & \oa_{2} & \iddots & \vdots \\
    \oa_{1} & \oa_{2',1} & \cdots & \oa_{r',1}
  \end{array}\right) \,,
\end{equation}
where the anti-diagonal terms $\{\oad_i,\oa_i\}_{i=1}^r$ are defined via:
\begin{equation}
  \oad_i=2\oad_{i,i'} \,, \qquad \oa_i=\oa_{i',i} \,.
\end{equation}

\medskip

\begin{Rem}
The Lax matrix~(\ref{eq:oscLaxC}) is the specialization of~\cite[(3.51)]{Frassek:2021ogy}
at $x_1=\st,\ x_2=-\st-r-1$.
\end{Rem}

\medskip
\noindent
Writing~\eqref{eq:oscLaxC} in the form
\begin{equation}\label{eq:F-generators C-type}
  \CL(x)=x\ID_{2r} + \sum_{i,j=1}^{2r} e_{ij}\CF_{ji} \,,
\end{equation}
we note that the RTT relation~\eqref{eq:RTT} implies that $\{\CF_{ij}\}_{i,j=1}^{2r}$
satisfy the $\ssp_{2r}$ commutation relations:
\begin{equation}\label{eq:comC}
  [\CF_{ij},\CF_{k\ell}]=\delta_{k}^{j}\CF_{i\ell}-\delta_{i}^{\ell}\CF_{kj}-
  \epsilon_{i}\epsilon_j(\delta_{k}^{i'}\CF_{j'\ell}-\delta_{\ell}^{j'}\CF_{ki'}) \,, \qquad
  \CF_{ij}=-\epsilon_i\epsilon_j\CF_{j'i'}\,,
\end{equation}
with $\{\epsilon_i\}_{i=1}^{2r}$ defined as in the Introduction:
\begin{equation}
  \epsilon_1=\cdots=\epsilon_r=1 \qquad \mathrm{and} \qquad \epsilon_{r+1}=\cdots=\epsilon_{2r}=-1 \,.
\end{equation}

\medskip
\noindent
As before, let $\Fock$ denote the Fock module of $\CA$, generated by the Fock vacuum $|0\rangle\in \Fock$ satisfying:
\begin{equation}
  \oa_{j',i} |0\rangle = 0 \,, \qquad 1\leq i\leq j \leq r \,.
\end{equation}
Then, the Fock vacuum $|0\rangle$ is a highest weight state of the resulting $\ssp_{2r}$-action:
\begin{equation}
  \CF_{ij} |0\rangle=0 \,, \qquad 1\leq i< j\leq 2r \,,
\end{equation}
with the highest weight $\lambda = \st \omega_r = (\underbrace{\st,\ldots,\st}_{r})$, that is:
\begin{equation}\label{eq:hwt1}
  \CF_{ii} |0\rangle = \st |0\rangle \,, \qquad 1\leq i\leq r \,.
\end{equation}
The latter is a consequence of the following explicit formulas for any $1\leq i\leq r$:
\begin{equation}\label{eq:C-diag}
\begin{split}
  & \CF_{ii}=\st-2\oad_{ii'}\oa_{i'i}\, - \sum_{k=i+1}^r\oad_{ik'}\oa_{k'i}-\sum_{k=1}^{i-1}\oad_{ki'}\oa_{i'k} \,, \\
  & \CF_{i'i'}=-\st-r-1+2\oa_{i'i}\oad_{ii'}+\sum_{k=1}^{i-1}\oa_{i'k}\oad_{ki'}+\sum_{k=i+1}^{r}\oa_{k'i}\oad_{ik'}
    = -\CF_{ii} \,.
\end{split}
\end{equation}
Similarly to Lemma~\ref{lem:A-Fock as par-Verma}, we can identify the resulting $\ssp_{2r}$-modules $\Fock$
as follows:

\medskip

\begin{Lem}\label{lem:C-Fock as par-Verma}
There is an $\ssp_{2r}$-module isomorphism:
\begin{equation}\label{eq:C-Fock-as-verma}
  \Fock \simeq \left(M^{\fp_{\{1,\ldots,r-1\}}}_{\st\omega_r}\right)^\vee \,,
\end{equation}
identifying $\Fock$ with the restricted dual~\eqref{eq:restricted-dual}
of the parabolic Verma module~\eqref{eq:parabolic def}.
\end{Lem}

\medskip
\noindent
Combining this with the determinant formula of~\cite{j} and the isomorphism~\eqref{eq:self-dual}, we obtain:

\medskip

\begin{Cor}\label{cor:t-special C}
(a) For $\st\notin \frac{1}{2}\BZ_{\geq 2-2r}$, the $\ssp_{2r}$-module $\Fock$ is irreducible
(thus, is generated by $|0\rangle$).

\medskip
\noindent
(b) For $\st\in \BN$, the Fock vacuum $|0\rangle$ generates an irreducible finite-dimensional
$\ssp_{2r}$-module $L_{\st\omega_r}$.
\end{Cor}

\medskip


\subsection{More oscillator realizations in type C via underlying symmetries}
\label{ssec C all osc}
$\ $

Consider the following endomorphisms of $\BC^{2r}$:
\begin{equation}
  B_{i}=e_{ii'}-e_{i'i}+\sum_{1\leq j\leq r}^{j\ne i} \left(e_{jj}+e_{j'j'}\right)
  \,, \qquad 1\leq i\leq r \,,
\end{equation}
along with their order-independent products:
\begin{equation}\label{eq:C-type weyl elt}
  B_{\vec\mu}\, = \prod_{1\leq j\leq r}^{\mu_j=-1} B_j \,,
  \qquad \vec\mu=(\mu_1,\ldots,\mu_r) \in \{\pm 1\}^r \,.
\end{equation}

\medskip

\begin{Rem}\label{rem:B-action C-type}
For $1\leq i\leq r$, we have:
  $B_{\vec \mu}(e_i)=
     \begin{cases}
       e_i & \text{if } \mu_i=1 \\
       -e_{i'} & \text{if } \mu_i=-1
     \end{cases}\, , \,
   B_{\vec \mu}(e_{i'})=
     \begin{cases}
       e_{i'} & \text{if } \mu_i=1 \\
       e_{i} & \text{if } \mu_i=-1
     \end{cases}$.
\end{Rem}

\medskip
\noindent
Since the $R$-matrix~\eqref{eq:BCD-Rmatrix} is invariant under such transformations, cf.~\eqref{eq:R-inv}:
\begin{equation}
  [R(x),B_{\vec\mu}\otimes B_{\vec\mu}]=0 \,, \qquad \forall \vec\mu\in \{\pm 1\}^r \,,
\end{equation}
we can generate more solutions to the RTT relation~\eqref{eq:RTT} from the Lax matrix~\eqref{eq:oscLaxC} via:
\begin{equation}\label{eq:Laxhat}
  \hat{\CL}_{\vec\mu}(x) = B_{\vec\mu}\CL(x)B_{\vec\mu}^{-1}=
  x\ID_{2r}+\sum_{i,j=1}^{2r} e_{ij}\hat{\CF}_{ji}^{\vec\mu} \,, \qquad \forall \vec\mu\in \{\pm 1\}^r \,.
\end{equation}

\medskip
\noindent
We shall further apply the following particle-hole automorphism of $\CA$ (denoted \emph{p.h.}):
\begin{equation}\label{eq:pt-transform}
  \oad_{i,j'}\mapsto -\oa_{j',i} \,, \quad \oa_{j',i}\mapsto\oad_{i,j'} \,\quad
  \text{for}\quad 1\leq i\leq j\leq r \quad \text{such\ that} \quad \mu_i=-1 \,.
\end{equation}
Thus, we obtain the following explicit $C_r$-type $\CA[x]$-valued Lax matrices:
\begin{equation}\label{eq:Laxa}
   {\CL}_{\vec\mu}(x)=\left.\hat{\CL}_{\vec\mu}(x)\right|_{p.h.} =
   \left. B_{\vec\mu}\CL(x)B_{\vec\mu}^{-1} \right|_{p.h.} = \
   x\ID_{2r} + \sum_{i,j=1}^{2r} e_{ij}{\CF}_{ji}^{\vec\mu} \,,
   \qquad \forall \vec\mu\in \{\pm 1\}^r \,.
\end{equation}
The resulting matrix elements $\{\CF_{ij}^{\vec\mu}\}_{i,j=1}^{2r}$ of $\CA$ satisfy the $\ssp_{2r}$ commutation
relations~\eqref{eq:comC}, due to the RTT relation~\eqref{eq:RTT}. This makes the Fock module $\Fock$ into
an $\ssp_{2r}$-module, denoted by $M^+_{\vec\mu,\st}$. We furthermore note that the particular choice
of the particle-hole transformation~\eqref{eq:pt-transform} is uniquely made to insure that the Fock vacuum
$|0\rangle\in M^+_{\vec\mu,\st}$ remains to be an $\ssp_{2r}$ highest weight state:
\begin{equation}
  \CF^{\vec\mu}_{ij} |0\rangle=0 \,, \qquad 1\leq i<j\leq 2r \,.
\end{equation}
To compute its highest weight, we note that:
\begin{equation}\label{eq:diagFC}
  \diag\left(\hat{\CF}^{\vec\mu}\right)=\diag\left(B_{\vec\mu} \CF B^{-1}_{\vec\mu}\right)=
  \left(\mu_1\CF_{11},\ldots,\mu_r\CF_{rr},-\mu_r\CF_{rr},\ldots,-\mu_1\CF_{11}\right) \,,
\end{equation}
due to~\eqref{eq:C-diag}, which after implementing the particle-hole transformation~\eqref{eq:pt-transform} gives:
\begin{equation}\label{eq:replC}
  \CF_{ii}^{\vec\mu} |0\rangle =
  \mu_i\left(\st+(r-i+1)\delta_{\mu_i}^{-}+\sum_{k=1}^i \delta_{\mu_k}^{-}\right) |0\rangle
  \,, \qquad 1\leq i\leq r \,.
\end{equation}
Thus, the Fock vacuum $|0\rangle\in M^+_{\vec\mu,\st}$ is an $\ssp_{2r}$ highest weight state whose weight
is given by~\eqref{eq:replC}.

\medskip
\noindent
We shall now compare the above modules $M^+_{\vec\mu,\st}$'s with those from the Introduction.
To this end, let us consider the parabolic $\fp_S\subset \ssp_{2r}$ corresponding to $S=\{1,\ldots,r-1\}$,
see Section~\ref{ssec truncated-BGG}.
The Weyl group of $\ssp_{2r}$ can be identified with $W\simeq (\BZ/2\BZ)^r\rtimes S_r\simeq \{\pm 1\}^r\rtimes S_r$,
so that elements of $W$ are indexed by pairs $(\vec\mu,\sigma)$ with $\vec{\mu}\in \{\pm 1\}^r$ and $\sigma\in S_r$.
The Weyl group of the Levi subalgebra $\fl\simeq \gl_r$ is $W_{\fl}\simeq S_r$ consisting of the elements
$\left((+1,\ldots,+1),\sigma\right)_{\sigma\in S_r}\subset W$. Thus, we have a set bijection
\begin{equation}\label{eq:set-bij C}
  \pi\colon W/W_{\fl} \ \iso \ \{\pm 1\} ^r \,,
\end{equation}
cf.~\eqref{eq:set-bij}. Given any $\vec\mu\in \{\pm 1\}^r$, we define the permutation $\sigma_{\vec\mu}\in S_r$ via:
\begin{equation}\label{eq:mu-to-sigma}
  \sigma_{\vec\mu}^{-1}(i)=
  \begin{cases}
    \#\{1\leq k\leq i \, | \, \mu_k=1\} & \text{if }\ \mu_i=1 \\
    r+1-\#\{1\leq k\leq i \, | \, \mu_k=-1\} & \text{if }\ \mu_i=-1
  \end{cases} \,,
\end{equation}
and further consider $w_{\vec\mu}\in W\simeq \{\pm 1\}^r\rtimes S_r$ defined via:
\begin{equation}\label{eq:mu-to-w}
  w_{\vec\mu}=(\vec\mu,\sigma_{\vec\mu}) \,.
\end{equation}
It is clear that $w_{\vec\mu}\in \pi^{-1}(\vec\mu)$, see~\eqref{eq:set-bij C}.
Furthermore, it can be characterized as follows:

\medskip

\begin{Lem}\label{lem:shortest decsription C}
$w_{\vec\mu}$ is the shortest representative of the left coset $\pi^{-1}(\vec\mu)$, for any $\vec\mu \in \{\pm 1\}^r$.
\end{Lem}

\medskip

\begin{proof}
This follows from the standard combinatorial description of the length function on the Weyl group
of any Lie algebra $\fg$:
\begin{equation}\label{eq:length-interpretation}
  l(w)=\#\Big\{\alpha\in \Delta^+ \, | \, w(\alpha)\in -\Delta^+ \Big\}
\end{equation}
for any $w\in W$, where $\Delta^+$ denotes the set of positive roots of $\fg$ (cf.\ Remark~\ref{rem:proof-C-char}).
\end{proof}

\medskip

\begin{Cor}\label{C-left-coset}
${}^{\fl}W = \{w_{\vec\mu}\}_{\vec\mu\in \{\pm 1\}^r}$.
\end{Cor}

\medskip
\noindent
Combining now the formula~\eqref{eq:replC} with the definition of $w_{\vec\mu}\in W$,
see~(\ref{eq:mu-to-sigma},~\ref{eq:mu-to-w}), we conclude that the highest weight of the Fock vacuum
$|0\rangle\in M^+_{\vec\mu,\st}$ coincides with the highest weight of our key modules
$M'_{w_{\vec\mu}\, \cdot \, \st\omega_r}$ introduced in~\eqref{eq:hard modules},
see Corollary~\ref{C-left-coset}. Furthermore, $M'_{w_{\vec\mu}\, \cdot \, \st\omega_r}$ has
the same character as $M^+_{\vec\mu,\st}$ (according to Lemma~\ref{lem:M-char}) and is irreducible
for $\st\notin \frac{1}{2}\BZ$ (as follows from~\cite{j}).
Therefore, similarly to Proposition~\ref{lem:A-Fock as M-mod}, we obtain:

\medskip

\begin{Prop}\label{lem:C-Fock as M-mod}
For any $\vec\mu\in \{\pm 1\}^r$ and $\st\notin \sfrac{1}{2}\BZ$, we have $\ssp_{2r}$-module isomorphisms:
\begin{equation}\label{eq:C-Fock-as-M}
  M^+_{\vec\mu,\st} \simeq M'_{w_{\vec\mu}\, \cdot \, \st\omega_r} \simeq
  \left(M'_{w_{\vec\mu}\, \cdot \, \st\omega_r}\right)^\vee \,.
\end{equation}
\end{Prop}

\medskip

\begin{Rem}\label{rem:C-details}
Let us point out the key difference between Proposition~\ref{lem:C-Fock as M-mod} and Lemma~\ref{lem:C-Fock as par-Verma}:

\medskip
\noindent
(a) For $\vec\mu=(+1,\ldots,+1)$, we actually have $M^+_{\vec\mu,\st} \simeq (M'_{w_{\vec\mu}\, \cdot \, \st\omega_r})^\vee$
for any $\st\in \BC$, due to Lemma~\ref{lem:C-Fock as par-Verma}.

\medskip
\noindent
(b) Likewise, for $\vec\mu=(-1,\ldots,-1)$, we have $M^+_{\vec\mu,\st} \simeq M'_{w_{\vec\mu}\, \cdot \, \st\omega_r}$
for any $\st\in \BC$.

\medskip
\noindent
(c) For other $\vec\mu \in \{\pm 1\}^r$, $M^+_{\vec\mu,\st}$ is \underline{not isomorphic} to either of
$M'_{w_{\vec\mu}\, \cdot \, \st\omega_r}$ or $(M'_{w_{\vec\mu}\, \cdot \, \st\omega_r})^\vee$ at certain
$\st\in \BZ$ (but is expected to be isomorphic to one of the twisted Verma modules in the sense of~\cite{al}).
\end{Rem}

\medskip
\noindent
Evoking the above bijection $\{\pm 1\}^r\ni \vec\mu \leftrightarrow w_{\vec\mu}\in {}^{\fl}W$
of Corollary~\ref{C-left-coset}, let us define:
\begin{equation}\label{eq:CM-modified}
  M^\vee_{\vec\mu,\st}=\left(M'_{w_{\vec\mu}\, \cdot \, \st\omega_r}\right)^\vee \,, \qquad \forall\, \st\in \BC \,.
\end{equation}
Then, Proposition~\ref{lem:C-Fock as M-mod} can be recast as the isomorphism of the following $\ssp_{2r}$-modules:
\begin{equation}\label{C-generic-coincidence}
  M^+_{\vec\mu,\st}\simeq M^\vee_{\vec\mu,\st} \,, \qquad \forall \, \st\in \BC \setminus \sfrac{1}{2}\BZ \,.
\end{equation}
\noindent

\medskip
\noindent
For $\vec\mu\in \{\pm 1\}^r$, we also define its \emph{length} $\mathsf{l}(\vec\mu)$ as the length of the corresponding
element $w_{\vec\mu}\in W$. Using formula~\eqref{eq:length-interpretation} and the explicit description of the
set $\Delta^+$ of positive roots of $\ssp_{2r}$, we find:
\begin{equation}\label{eq:sigma-length C}
  \mathsf{l}(\vec\mu) = l(w_{\vec\mu}) = \sum_{i=1}^r (r-i+1)\delta_{\mu_i}^{-} \,.
\end{equation}
Our choice of notation is due to $\mathsf{l}(\vec\mu)\ne l(\vec\mu)$, the latter being used for the length of
$(\vec\mu,\mathrm{id})\in W$.

\medskip


\subsection{Type C transfer matrices}
$\ $

Recall the notion of transfer matrices $\{T_W(x)\}_{W\in \mathrm{Rep}\, Y(\ssp_{2r})}$,
as discussed in Subsection~\ref{ssec R-matrix}.
In particular, we shall consider the following explicit infinite-dimensional transfer matrices:
\begin{equation}\label{eq:C-inf-transfer}
  T_{\vec\mu,\st}^+(x)=\tr \prod_{i=1}^r \tau_i^{\CF^{\vec\mu}_{ii}}
  \underbrace{\CL_{\vec\mu}(x) \otimes \dots \otimes \CL_{\vec\mu}(x)}_{N} \,,
\end{equation}
corresponding to $M^+_{\vec\mu,\st}$. For $\st\in \BN$, the finite-dimensional transfer matrices $T_{r,\st}(x)$
corresponding to the modules $L_{\st\omega_r}$ in the auxiliary space are defined similarly to~\eqref{eq:C-inf-transfer},
but with the trace taken over the
finite-dimensional submodule $L_{\st\omega_r}$ of $M^+_{(+1,\ldots,+1),\st}$, see Corollary~\ref{cor:t-special C}(b).

\medskip
\noindent
Using the notation~(\ref{eq:CM-modified},~\ref{eq:sigma-length C}), let us recast
the resolution~\eqref{eq:geometric resolution 1}, dual to~\eqref{eq:conjectured resolution 1}, as follows:
\begin{equation}\label{eq:C-resolution}
  0\to L_{\st \omega_r}\to  M^\vee_{(+1,\ldots,+1),\st}\to
  \bigoplus_{\vec\mu\in \{\pm 1\}^r}^{\mathsf{l}(\vec\mu)=1} M^\vee_{\vec\mu,\st} \to
  \bigoplus_{\vec\mu\in \{\pm 1\}^r}^{\mathsf{l}(\vec\mu)=2} M^\vee_{\vec\mu,\st} \to
  \cdots \to  M^\vee_{(-1,\ldots,-1),\st}\to 0
\end{equation}
for any $\st\in \BN$. Combining this with~\eqref{C-generic-coincidence} and the fact that
\underline{the transfer matrices~\eqref{eq:C-inf-transfer} depend} \underline{continuously on $\st\in \BC$}
(as so do the Lax matrices $\CL_{\vec\mu}(x)$), we obtain the key result of this section:

\medskip

\begin{Thm}\label{thm:Main-C}
For $\st\in \BN$, we have:
\begin{equation}\label{eq:transfer-C}
  T_{r,\st}(x)\, = \sum_{\vec\mu\in \{\pm 1\}^r} (-1)^{\mathsf{l}(\vec{\mu})}\, T_{\vec\mu,\st}^+(x) \,.
\end{equation}
\end{Thm}

\medskip
\noindent
The character limit of~\eqref{eq:transfer-C} expresses the character
of the $\ssp_{2r}$-modules $\{L_{\st \omega_r}\}_{\st\in \BN}$ defined as
\begin{equation}\label{eq:C-char-def}
  \ch_{r,\st}=\ch_{r,\st}(\tau_1, \ldots, \tau_r) :=
  \tr_{L_{\st\omega_r}} \prod_{i=1}^r \tau_i^{\CF_{ii}} \,,
\end{equation}
that is the length $N=0$ case of $T_{r,\st}(x)$, via:
\begin{equation}\label{eq:char-C}
  \ch_{r,\st} \, =
  \sum_{\vec\mu=(\mu_1,\ldots,\mu_r)\in \{\pm 1\}^r} (-1)^{\mathsf{l}(\vec{\mu})}
  \frac{\prod_{i=1}^{r} \tau_i^{\mu_i\left(\st+(r-i+1)\delta_{\mu_i}^{-}+\sum_{k=1}^i \delta_{\mu_k}^{-}\right)}}
       {\prod_{1\leq i\leq j\leq r}\left(1-\frac{1}{\tau_i\tau_j^{\mu_i \mu_j}}\right)}
\end{equation}
with the $\vec\mu$'s summand in the right-hand side of~\eqref{eq:char-C} equal to the character of
$M^+_{\vec\mu,\st}$, up to a sign.

\medskip

\begin{Rem}\label{rem:proof-C-char}
For the physics' reader who skipped Section~\ref{sec: truncated BGG resolutions},
let us present a concise proof of~\eqref{eq:char-C}.
Let us identify the set $\Delta^+$ of positive roots of $\fg=\ssp_{2r}$ with
  $\Delta^+=\{\epsilon_i-\epsilon_j\}_{1\leq i<j\leq r} \cup \{\epsilon_i+\epsilon_j\}_{1\leq i\leq j\leq r}$,
so that $\rho=(r,r-1,\ldots,2,1)$ in the basis $\{\epsilon_i\}_{i=1}^r$ and the Weyl group $W$ gets identified with
$W\simeq (\BZ/2\BZ)^r\rtimes S_r\simeq \{\pm 1\}^r\rtimes S_r$. According to the Weyl character formula, we have:
\begin{equation}\label{eq:Weyl C}
  \ch_{L_{\st\omega_r}}\, =
  \sum_{(\vec{\mu},\sigma)\in \{\pm 1\}^r\rtimes S_r} (-1)^{l(\vec\mu,\sigma)}
  \frac{e^{(\vec\mu,\sigma)(\st\omega_r+\rho)-\rho}}
       {\prod_{1\leq i<j\leq r}(1-e^{\epsilon_j-\epsilon_i})
        \prod_{1\leq i\leq j\leq r}(1-e^{-\epsilon_j-\epsilon_i})} \,.
\end{equation}
Following~\S\ref{ssec C all osc}, let us consider the parabolic subalgebra
$\fp_{\{1,\ldots,r-1\}}\subset\fg$ whose Levi subalgebra is $\fl\simeq \gl_r$ and the Weyl group
is $W_{\fl}\simeq S_r=\{\left((+1,\ldots,+1),\sigma\right)\}_{\sigma\in S_r}\subset W$.
We can rewrite~\eqref{eq:Weyl C} as:
\begin{equation}\label{eq:Weyl C doubled}
  \ch_{L_{\st\omega_r}}\, =
  \sum_{\vec\mu\in \{\pm 1\}^r} \sum_{\sigma\in S_r}
  (-1)^{l(\vec\mu,\sigma)}\frac{e^{(\vec\mu,\sigma)(\st\omega_r+\rho)-\rho}}
   {\prod_{1\leq i<j\leq r}(1-e^{\epsilon_j-\epsilon_i})
        \prod_{1\leq i\leq j\leq r}(1-e^{-\epsilon_j-\epsilon_i})} \,.
\end{equation}
The key step is to simplify the inner sum of~\eqref{eq:Weyl C doubled} using the Weyl denominator formula for $\fl$:
\begin{equation}\label{eq:Weyl denominator C}
  \sum_{\sigma \in S_r} (-1)^{l(\sigma)} e^{\sigma(\rho_\fl)-\rho_\fl} \, =
  \prod_{\alpha\in \Delta^+_\fl} (1-e^{-\alpha}) \,,
\end{equation}
where $\Delta^+_{\fl}=\{\epsilon_i-\epsilon_j\}_{1\leq i<j\leq r} \subset \Delta^+$ consists of positive roots
of $\fl$, $\rho_{\fl}=\frac{1}{2}\sum_{\alpha\in \Delta^+_{\fl}} \alpha = (\frac{r-1}{2},\ldots,\frac{1-r}{2})$.
As $\sigma(\rho)-\rho=\sigma(\rho_{\fl})-\rho_{\fl}$ and $\sigma(\omega_r)=\omega_r$ for any $\sigma\in S_r$, we get:
\begin{equation*}
  \sum_{\sigma\in S_r} (-1)^{l(\sigma)}
  \frac{e^{\sigma(\st\omega_r+\rho)-\rho}}
       {\prod_{1\leq i<j\leq r}(1-e^{\epsilon_j-\epsilon_i})
        \prod_{1\leq i\leq j\leq r}(1-e^{-\epsilon_j-\epsilon_i})} =
  \frac{\prod_{i=1}^{r} e^{\st\epsilon_i}}{\prod_{1\leq i\leq j\leq r}(1-e^{-\epsilon_j-\epsilon_i})} \,.
\end{equation*}
Thus, the inner sum of~\eqref{eq:Weyl C doubled} indexed by $\vec\mu=(+1,\ldots,+1)$ gives rise to
the corresponding summand of~\eqref{eq:char-C}. We claim that the same holds for any $\vec\mu\in \{\pm 1\}^r$,
which amounts to the proof of~\eqref{eq:char derivation C} below.
To this end, let us apply $\vec\mu=(\vec\mu,1)\in W$ to both sides of the equality~\eqref{eq:Weyl denominator C}:
\begin{equation}
  \sum_{\sigma\in S_r} (-1)^{l(\vec\mu,\sigma)}e^{(\vec\mu,\sigma)(\rho)-\rho} \, =\,
  (-1)^{l(\vec\mu)}e^{\vec\mu(\rho)-\rho}
  \prod_{1\leq i<j\leq r} (1-e^{\mu_j \epsilon_j-\mu_i \epsilon_i}) \,.
\end{equation}
We note that
\begin{equation*}
  1-e^{\mu_j \epsilon_j-\mu_i \epsilon_i}=
  \begin{cases}
      1-e^{\epsilon_j-\epsilon_i} & \text{if }\ \mu_i=1,\ \mu_j=1 \\
      1-e^{-\epsilon_j-\epsilon_i} & \text{if }\ \mu_i=1,\ \mu_j=-1 \\
      -e^{\epsilon_i-\epsilon_j}(1-e^{\epsilon_j-\epsilon_i}) & \text{if }\ \mu_i=-1,\ \mu_j=-1 \\
      -e^{\epsilon_i+\epsilon_j}(1-e^{-\epsilon_j-\epsilon_i}) & \text{if }\ \mu_i=-1,\ \mu_j=1 \\
  \end{cases}
\end{equation*}
as well as
\begin{equation*}
  \vec\mu(\rho)-\rho=\sum_{i=1}^r (\mu_i-1)(r+1-i)\epsilon_i \,, \qquad
  (-1)^{l(\vec\mu)}=(-1)^{\sum_{i=1}^r \delta_{\mu_i}^-} \,, \qquad
  \vec\mu(\st\omega_r)=(\st\mu_1,\ldots,\st\mu_r) \,.
\end{equation*}
Thus, we obtain the desired equality:
\begin{multline}\label{eq:char derivation C}
  \sum_{\sigma\in S_r}
  (-1)^{l(\vec\mu,\sigma)}\frac{e^{(\vec\mu,\sigma)(\st\omega_r+\rho)-\rho}}
   {\prod_{1\leq i<j\leq r}(1-e^{\epsilon_j-\epsilon_i})
        \prod_{1\leq i\leq j\leq r}(1-e^{-\epsilon_j-\epsilon_i})} = \\
  (-1)^{\mathsf{l}(\vec\mu)}
  \frac{\prod_{i=1}^{r} e^{\mu_i\left(\st+(r-i+1)\delta_{\mu_i}^{-}+\sum_{k=1}^i \delta_{\mu_k}^{-}\right)\epsilon_i}}
       {\prod_{1\leq i\leq j\leq r}\left(1-e^{-\mu_i \mu_j \epsilon_j - \epsilon_i}\right)} \,.
\end{multline}
This completes our direct proof of the character formula~\eqref{eq:char-C}
(see~Remark~\ref{rem:math-to-physics} for more details in regards to perceiving
$\ch_{r,\st}$ of~\eqref{eq:C-char-def} as a specialization of $\ch_{L_{\st\omega_r}}$).
\end{Rem}

\medskip
\noindent
We note that the formula~\eqref{eq:char-C} allows to \underline{analytically continue} the character
$\ch_{r,\st}$ of~\eqref{eq:C-char-def} from the discrete set $\st\in \BN$ to the entire
complex plane $\st\in \BC$. With this convention in mind, we have:

\medskip

\begin{Lem}\label{lem:t-property C}
(a) $\ch_{r,\st}=(-1)^{\frac{r(r+1)}{2}}\ch_{r,-r-1-\st}$ for any $\st\in \BC$.

\medskip
\noindent
(b) $\ch_{r,\st}=0$ for $\st\in \Big\{-\frac{2}{2},-\frac{3}{2},\ldots,-\frac{2r-1}{2},-\frac{2r}{2}\Big\}$.
\end{Lem}

\medskip

\begin{proof}
(a) For any $\vec\mu=(\mu_1,\ldots,\mu_r)\in \{\pm 1\}^r$ and $\st\in \BC$,
define $\vec{\bar\mu}\in \{\pm 1\}^r$ and $\bar\st\in \BC$ via:
\begin{equation}\label{eq:opposite-mu}
  \vec{\bar\mu}=(\bar\mu_1,\ldots,\bar\mu_r):=-\vec\mu
  \,, \qquad \bar\st=-r-1-\st \,.
\end{equation}
Then, the obvious equality
  $\sum_{i=1}^r (r-i+1) (\delta_{\mu_i}^{-} + \delta_{\bar\mu_i}^-) = \frac{r(r+1)}{2}$
implies:
\begin{equation}\label{eq:L+L}
  \mathsf{l}(\vec\mu)+\mathsf{l}(\vec{\bar\mu})=\sfrac{r(r+1)}{2} \,, \qquad \forall\, \vec\mu\in \{\pm 1\}^r \,.
\end{equation}
Let $\ch^+_{\vec\mu,\st}$ denote the
$\vec\mu$-th summand in the right-hand side of~\eqref{eq:char-C} without a sign, see~\eqref{eq:ch-C}:
\begin{equation}\label{eq:ch-plus}
   \ch^+_{\vec\mu,\st} =
  \frac{\prod_{i=1}^{r} \tau_i^{\mu_i\left(\st+(r-i+1)\delta_{\mu_i}^{-}+\sum_{k=1}^i \delta_{\mu_k}^{-}\right)}}
       {\prod_{1\leq i\leq j\leq r}\left(1-\tau_i^{-1}\tau_j^{-\mu_i \mu_j}\right)} \,.
\end{equation}
Then, we have the symmetry of~\eqref{eq:ch-plus} with respect to~\eqref{eq:opposite-mu}:
\begin{equation}\label{eq:ch=ch}
  \ch^+_{\vec\mu,\st}=\ch^+_{\vec{\bar\mu},\bar\st}
\end{equation}
as follows from the equality
  $\mu_i\Big(\st + (r-i+1)\delta_{\mu_i}^{-} + \sum_{k=1}^i \delta_{\mu_k}^{-}\Big) =
   \bar\mu_i\Big(\bar\st + (r-i+1)\delta_{\bar\mu_i}^{-} + \sum_{k=1}^i \delta_{\bar\mu_k}^{-}\Big)$.
This implies the desired equality:
\begin{equation}\label{eq:charcater comparison}
  \ch_{r,\st}=\sum_{\vec\mu\in \{\pm 1\}^r} (-1)^{\mathsf{l}(\vec\mu)}\ch^+_{\vec\mu,\st} =
  (-1)^{\frac{r(r+1)}{2}}\sum_{\vec\mu\in \{\pm 1\}^r} (-1)^{\mathsf{l}(\vec{\bar\mu})}\ch^+_{\vec{\bar\mu},\bar\st} =
  (-1)^{\frac{r(r+1)}{2}}\ch_{r,\bar\st}  \,,
\end{equation}
by matching the $\vec\mu$-th summand in $\ch_{r,\st}$ with the $\vec{\bar\mu}$-th summand in $\ch_{r,\bar\st}$ for every
$\vec\mu\in \{\pm 1\}^r$.

\medskip
\noindent
(b) To prove (b), we shall rather use the Weyl character formula~\eqref{eq:Weyl C}, which implies that
$\ch_{r,\st}=0$ if and only if one can split elements of the Weyl group $W\simeq \{\pm 1\}^r \rtimes S_r$ into pairs
$(w,w')$ so that:
\begin{equation}\label{eq:pair compatibility}
  (-1)^{l(w)}=-(-1)^{l(w')}  \qquad \mathrm{and} \qquad w(\st\omega_r+\rho)=w'(\st\omega_r+\rho)
\end{equation}
with $\rho=(r,\ldots,1)$ and $\omega_r=(1,\ldots,1)$.
Let us indicate such splittings for the desired values of~$\st$:
\begin{enumerate}

\item[$\bullet$]
set $w'=w \Big( (\underbrace{+1,\ldots,+1}_{k-1},-1,\underbrace{+1\ldots,+1}_{r-k}),\mathrm{id} \Big)$
if $\st=-(r+1-k)$ with $1\leq k\leq r$;

\item[$\bullet$]
set $w'=w \Big( (\underbrace{+1,\ldots,+1}_{k-1},-1,\underbrace{+1,\ldots,+1}_{m-k-1},-1,\underbrace{+1\ldots,+1}_{r-m}),
(k\ m) \Big)$ with $(k\ m)\in S_r$ denoting the transposition exchanging $k$ and $m$, if $\st=-(r+1-\sfrac{k+m}{2})$ with
$1\leq k<m\leq r$.

\end{enumerate}
By Remark~\ref{rem:proof-C-char}, note that~\eqref{eq:char-C} and~\eqref{eq:Weyl C} provide the same analytic continuations
of~\eqref{eq:C-char-def}.
\end{proof}

\medskip

\begin{Rem}
Following the above proof, we get $\ch_{r,\st}=0$ only for
$\st\in \Big\{-\frac{2}{2},-\frac{3}{2},\ldots,-\frac{2r-1}{2},-\frac{2r}{2}\Big\}$.
\end{Rem}

\medskip
\noindent
In a completely similar way, the formula~\eqref{eq:transfer-C} allows to \underline{analytically continue} the transfer
matrices $T_{r,\st}(x)$ of the finite-dimensional representations $L_{\st\omega_r},\st\in\BN$, to the entire complex
plane $\st\in \BC$. With this convention in mind, we have the following generalization of Lemma~\ref{lem:t-property C}(a):

\medskip

\begin{Prop}\label{prop:t-symmetry C}
$T_{r,\st}(x)=(-1)^{\frac{r(r+1)}{2}} T_{r,-r-1-\st}(x)$ for any $\st\in \BC$.
\end{Prop}

\begin{proof}
This follows from the factorisation~\eqref{eq:T-via-QQ C-type} of each infinite-dimensional
transfer matrix $T^+_{\vec\mu,\st}(x)$ into the product of $Q$-operators that allows to
recast Theorem~\ref{thm:Main-C} in the form of Proposition~\ref{prop:main-C factorized}:
\begin{equation}\label{eq:C-recasted}
  T_{r,\st}(x)\, = \sum_{\vec\mu\in \{\pm 1\}^r} (-1)^{\mathsf{l}(\vec{\mu})}\, \ch_{\vec\mu,\st}^+ \cdot\,
  Q_{\vec\mu}(x+\st)Q_{\vec{\bar\mu}}(x+\bar{\st}) \,,
\end{equation}
with $\vec{\bar\mu},\bar{\st}$ as in~\eqref{eq:opposite-mu} and $\ch_{\vec\mu,\st}^+$ as in~\eqref{eq:ch-plus},
see~\eqref{eq:ch-C}. Similarly to our proof of Lemma~\ref{lem:t-property C}(a), we claim that any $\vec\mu$-th summand
in the right-hand side of~\eqref{eq:C-recasted} for $T_{r,\st}(x)$ coincides with the $\vec{\bar\mu}$-th summand
in the corresponding expression for $T_{r,\bar\st}(x)$, up to an overall sign $(-1)^{\frac{r(r+1)}{2}}$:
\begin{equation}\label{eq:C-comparison}
  (-1)^{\mathsf{l}(\vec{\mu})}\, \ch_{\vec\mu,\st}^+ \cdot\, Q_{\vec\mu}(x+\st)Q_{\vec{\bar\mu}}(x+\bar{\st}) =
  (-1)^{\mathsf{l}(\vec{\bar\mu})+\sfrac{r(r+1)}{2}}\,
  \ch_{\vec{\bar\mu},\bar\st}^+ \cdot\, Q_{\vec{\bar\mu}}(x+\bar\st)Q_{\vec{\mu}}(x+\st) \,.
\end{equation}
The latter is a consequence of~(\ref{eq:L+L},~\ref{eq:ch=ch}) combined with
the essential property of the $Q$-operators:
\begin{equation}\label{eq:Q-comm C}
  [Q_{\vec\mu}(x),Q_{\vec{\bar\mu}}(y)]=0 \,,
\end{equation}
that follows from the natural commutativity of
the transfer matrices, $[T^+_{\vec\mu}(x),T^+_{\vec\mu}(y)]=0$, combined with the realization of the
$Q$-operators as \emph{renormalized limits} of the transfer matrices, cf.~\eqref{eq:C renormalized limit}.
\end{proof}

\medskip
\noindent
We also expect the generalization of Lemma~\ref{lem:t-property C}(b) to hold:
$T_{r,\st}(x)=0$ for $\st\in \Big\{-\frac{2}{2},-\frac{3}{2},\ldots,-\frac{2r}{2}\Big\}$.

$\ $


\section{Resolutions for transfer matrices of D-type: spinorial representations}\label{sec:spinD}

In this section, we present a natural counterpart of the results from Section~\ref{sec C-spinor} for $D$-type.

\subsection{Oscillator realization in type D (parabolic Verma)}
$\ $

Let $\CA$ denote the oscillator algebra generated by $\frac{r(r-1)}{2}$ pairs of oscillators
$\{(\oa_{j',i},\oad_{i,j'})\}_{1\leq i<j \leq r}$ with $r\geq 2$, cf.\ notation~\eqref{eq:prime-index},
subject to the defining relations~\eqref{eq:oscillator relations 2}:
\begin{equation}
  \CA = \BC \Big\langle \oa_{j',i} \, , \, \oad_{i,j'} \Big\rangle_{1\leq i < j \leq r} \, \Big/ \,
  \eqref{eq:oscillator relations 2} \,.
\end{equation}
Following~\cite[(5.4)]{f} and similarly to~\eqref{eq:oscLaxC},
let us consider the $D_r$-type
$\CA[x]$-valued Lax matrix:
\begin{equation}\label{eq:linear-Lax-D}
  \CL(x)=
  \left(\begin{BMAT}[5pt]{c:c}{c:c}
    (x+\st)\ID_r-\ap\am & -\ap(2\st+r-1-\am\ap) \\
    -\am & (x-\st-r+1)\ID_r+\am\ap
  \end{BMAT} \right) \,,
\end{equation}
depending on $\st\in \BC$, with the blocks $\ap,\am\in \mathrm{Mat}_{r\times r}(\CA)$
encoding all the generators via:
\begin{equation}\label{eq:ApAmD}
  \ap=
    \left(\begin{array}{cccc}
      \oad_{1,r'} & \cdots & \oad_{1,2'} & 0 \\
      \vdots & \iddots & 0 & -\oad_{1,2'} \\
      \oad_{r-1,r'} & 0 & \iddots & \vdots \\
      0 & -\oad_{r-1,r'} & \cdots & -\oad_{1,r'}
    \end{array}\right) \,, \quad
  \am=
    \left(\begin{array}{cccc}
      \oa_{r',1} & \cdots & \oa_{r',r-1} & 0 \\
      \vdots & \iddots & 0 & -\oa_{r',r-1} \\
      \oa_{2',1} & 0 & \iddots & \vdots \\
      0 & -\oa_{2',1} & \cdots & -\oa_{r',1}
    \end{array}\right)\,.
\end{equation}

\medskip
\noindent
Writing~\eqref{eq:linear-Lax-D} in the form
\begin{equation}\label{eq:F-generators D-type}
  \CL(x)=x\ID_{2r}+\sum_{i,j=1}^{2r} e_{ij}\CF_{ji} \,,
\end{equation}
we note that the RTT relation~\eqref{eq:RTT} implies that $\{\CF_{ij}\}_{i,j=1}^{2r}$
satisfy the $\sso_{2r}$ commutation relations:
\begin{equation}\label{eq:comD}
  [\CF_{ij},\CF_{k\ell}]=\delta_{k}^{j}\CF_{i\ell}-\delta_{i}^{\ell}\CF_{kj}-
  \delta_{k}^{i'}\CF_{j'\ell}+\delta_{\ell}^{j'}\CF_{ki'} \,, \qquad
  \CF_{ij}=-\CF_{j'i'}\,.
\end{equation}

\medskip
\noindent
As before, let $\Fock$ denote the Fock module of $\CA$, generated by the Fock vacuum $|0\rangle\in \Fock$
satisfying:
\begin{equation}
  \oa_{j',i} |0\rangle = 0 \,, \qquad 1\leq i<j \leq r \,.
\end{equation}
Then, the Fock vacuum $|0\rangle$ is a highest weight state of the resulting $\sso_{2r}$-action:
\begin{equation}
  \CF_{ij} |0\rangle=0 \,, \qquad 1\leq i< j\leq 2r \,,
\end{equation}
with the highest weight $\lambda = 2\st \omega_r = (\underbrace{\st,\ldots,\st}_{r})$, that is:
\begin{equation}
  \CF_{ii} |0\rangle = \st |0\rangle \,, \qquad 1\leq i\leq r \,.
\end{equation}
The latter is a consequence of the following explicit formulas for any $1\leq i\leq r$:
\begin{equation}\label{eq:D-diagonal}
\begin{split}
  & \CF_{ii}=\st \, - \sum_{k=i+1}^r\oad_{ik'}\oa_{k'i}-\sum_{k=1}^{i-1}\oad_{ki'}\oa_{i'k} \,, \\
  & \CF_{i'i'}=-\st-r+1+\sum_{k=1}^{i-1}\oa_{i'k}\oad_{ki'}+\sum_{k=i+1}^{r}\oa_{k'i}\oad_{ik'}
    = -\CF_{ii} \,.
\end{split}
\end{equation}
Similarly to Lemmas~\ref{lem:A-Fock as par-Verma} and~\ref{lem:C-Fock as par-Verma},
we can identify the resulting $\sso_{2r}$-modules $\Fock$ as~follows:

\medskip

\begin{Lem}\label{lem:Dspin-Fock as par-Verma}
There is an $\sso_{2r}$-module isomorphism:
\begin{equation}\label{eq:Dspin-Fock-as-verma}
  \Fock \simeq \left(M^{\fp_{\{1,\ldots,r-1\}}}_{2\st\omega_r}\right)^\vee \,.
\end{equation}
\end{Lem}

\medskip
\noindent
Combining this with the determinant formula of~\cite{j} and the isomorphism~\eqref{eq:self-dual}, we obtain:

\medskip

\begin{Cor}\label{cor:t-special D}
(a) For $\st\notin \frac{1}{2}\BZ_{\geq 4-2r}$, the $\sso_{2r}$-module $\Fock$ is irreducible
(thus, is generated by $|0\rangle$).

\medskip
\noindent
(b) For $\st\in \frac{1}{2}\BN$, the Fock vacuum $|0\rangle$ generates an irreducible
finite-dimensional $\sso_{2r}$-module $L_{2\st\omega_r}$.
\end{Cor}

\medskip


\subsection{More oscillator realizations in type D via underlying symmetries}
$\ $

Similarly to the $C$-type case, let us consider the following endomorphisms of $\BC^{2r}$:
\begin{equation}\label{eq:BD}
  B_{i}=e_{ii'}+e_{i'i}+\sum_{1\leq j\leq r}^{j\ne i} \left(e_{jj}+e_{j'j'}\right)\,,
  \qquad 1\leq i\leq r \,,
\end{equation}
along with their order-independent products:
\begin{equation}\label{eq:fullBD}
  B_{\vec\mu}\, = \prod_{1\leq j\leq r}^{\mu_j=-1} B_j \,,
  \qquad \vec{\mu}=(\mu_1,\ldots,\mu_r) \in \{\pm 1\}^r \,.
\end{equation}

\medskip

\begin{Rem}
For $1\leq i\leq r$, we have:
  $B_{\vec \mu}(e_i)=
     \begin{cases}
       e_i & \text{if } \mu_i=1 \\
       e_{i'} & \text{if } \mu_i=-1
     \end{cases}\, , \,
   B_{\vec \mu}(e_{i'})=
     \begin{cases}
       e_{i'} & \text{if } \mu_i=1 \\
       e_{i} & \text{if } \mu_i=-1
     \end{cases}$.
\end{Rem}

\medskip
\noindent
Since the $R$-matrix~\eqref{eq:BCD-Rmatrix} is invariant under such transformations, cf.~\eqref{eq:R-inv}:
\begin{equation}
  [R(x),B_{\vec\mu}\otimes B_{\vec\mu}]=0 \,, \qquad \forall \vec\mu\in \{\pm 1\}^r \,,
\end{equation}
we can generate more solutions to the RTT relation~\eqref{eq:RTT} from the Lax matrix~\eqref{eq:linear-Lax-D} via:
\begin{equation}\label{eq:LaxDspin}
   {\CL}_{\vec\mu}(x)=\left. B_{\vec\mu}\CL(x)B^{-1}_{\vec\mu}\right|_{p.h.} =\
   x\ID_{2r} + \sum_{i,j=1}^{2r} e_{ij}{\CF}_{ji}^{\vec\mu}
   \,, \qquad \forall \vec\mu\in \{\pm 1\}^r \,.
\end{equation}
Here, we apply the following particle-hole automorphism of $\CA$ (denoted \emph{p.h.}), cf.~\eqref{eq:pt-transform}:
\begin{equation}\label{eq:ph-Dspin}
  \oad_{i,j'}\mapsto -\oa_{j',i} \,, \quad \oa_{j',i}\mapsto\oad_{i,j'} \, \quad
  \text{for}\quad 1\leq i<j\leq r \quad \text{such\ that} \quad \mu_i=-1 \,,
\end{equation}
uniquely chosen to insure that the Fock vacuum $|0\rangle$ remains to be an $\sso_{2r}$ highest weight state:
\begin{equation}
  \CF^{\vec\mu}_{ij} |0\rangle=0 \,, \qquad 1\leq i<j\leq 2r \,.
\end{equation}
The resulting matrix elements $\{\CF_{ij}^{\vec\mu}\}_{i,j=1}^{2r}$ of $\CA$ satisfy the $\sso_{2r}$ commutation
relations~\eqref{eq:comD}, due to the RTT relation~\eqref{eq:RTT}. This makes the Fock module $\Fock$ into an
$\sso_{2r}$-module, denoted by $M^+_{\vec\mu,\st}$. The corresponding highest weight of $|0\rangle$ is computed
similarly to $C$-type, see~\eqref{eq:replC}:
\begin{equation}\label{eq:replD}
  \CF_{ii}^{\vec\mu} |0\rangle =
  \mu_i \left(\st+(r-i-1)\delta_{\mu_i}^{-}+\sum_{k=1}^{i} \delta_{\mu_k}^{-}\right) |0\rangle \,,
  \qquad 1\leq i\leq r \,.
\end{equation}

\medskip
\noindent
Let us note that for the particular choice $\vec\mu=(+1,\ldots,+1,-1)$, the particle-hole
transformation~\eqref{eq:ph-Dspin} is the identity, and the resulting $\sso_{2r}$-module $M^+_{\vec\mu,\st}$
can be read off the Lax matrix
\begin{equation}\label{eq:linear-Lax-D-2nd}
  \CL_{-}(x) = \CL_{(\footnotesize{\underbrace{+1,\ldots,+1}_{r-1}},-1)}(x) \,,
\end{equation}
which is obtained from the Lax matrix $\CL_+(x)=\CL(x)$ of~\eqref{eq:linear-Lax-D} by permuting
its $r$-th and $(r+1)$-st rows and columns. We also have the following counterparts of
Lemma~\ref{lem:Dspin-Fock as par-Verma} and Corollary~\ref{cor:t-special D}:

\medskip

\begin{Lem}\label{lem:Dspin-Fock-2nd as par-Verma}
There is an $\sso_{2r}$-module isomorphism:
\begin{equation}\label{eq:Dspin-2nd-Fock-as-verma}
  M^+_{(\footnotesize{\underbrace{+1,\ldots,+1}_{r-1}},-1),\st} \simeq
  \left(M^{\fp_{\{1,\ldots,r-2,r\}}}_{2\st\omega_{r-1}}\right)^\vee \,.
\end{equation}
\end{Lem}

\medskip

\begin{Cor}\label{cor:t-special D-2nd}
(a) For $\st\notin \frac{1}{2}\BZ_{\geq 4-2r}$, the $\sso_{2r}$-module $M^+_{(+1,\ldots,+1,-1),\st}$ is irreducible.

\medskip
\noindent
(b) For $\st\in \frac{1}{2}\BN$, the Fock vacuum $|0\rangle$ generates an irreducible
finite-dimensional $\sso_{2r}$-module~$L_{2\st\omega_{r-1}}$.
\end{Cor}

\medskip
\noindent
For $\vec\mu=(\mu_1,\ldots,\mu_r)\in \{\pm 1\}^r$, we define its \emph{sign} $|\vec\mu|\in \{\pm 1\}$ via:
\begin{equation}\label{eq:mu-sign}
  |\vec\mu|=\mu_1\cdot \ldots \cdot \mu_r \,.
\end{equation}
We call $\vec\mu\in \{\pm 1\}^r$ \emph{even} (resp.\ \emph{odd}) if $|\vec\mu|=1$ (resp.\ $|\vec\mu|=-1$),
and denote the sets of such by
\begin{equation}\label{eq:even-odd}
  \{\pm 1\}^r_{+} = \Big\{\mathrm{even}\ \vec\mu\in \{\pm 1\}^r\Big\} \,,\qquad
  \{\pm 1\}^r_{-} = \Big\{\mathrm{odd}\ \vec\mu\in \{\pm 1\}^r\Big\} \,.
\end{equation}

\medskip
\noindent
For $\st\in \frac{1}{2}\BN$, let $L^\pm_{\st}$ denote the following irreducible
finite-dimensional $\sso_{2r}$-representations:
\begin{equation}\label{eq:two-spin}
  L^+_{\st} = L_{2\st\omega_{r}} \,, \qquad L^-_{\st} = L_{2\st\omega_{r-1}} \,,
\end{equation}
which can be uniformly written as $L^\pm_{\st}=L_{2\st\omega^\pm}$ with the weights $\omega^\pm$ defined via:
\begin{equation}\label{eq:+-spin}
  \omega^+ = \omega_{r} \,, \qquad \omega^- = \omega_{r-1} \,.
\end{equation}

\medskip
\noindent
Let us now generalize Lemmas~\ref{lem:Dspin-Fock as par-Verma},~\ref{lem:Dspin-Fock-2nd as par-Verma}
by comparing the above modules $\{M^+_{\vec\mu,\st}\}_{\vec\mu\in \{\pm 1\}^r}$ to those from the Introduction.
To this end, we shall consider two parabolic subalgebras $\fp_{S^\pm}\subset \sso_{2r}$ with
\begin{equation}\label{eq:pmS}
  S^+ = \{1,\ldots,r-2,r-1\} \,, \qquad S^- = \{1,\ldots,r-2,r\} \,.
\end{equation}
The Weyl group of $\sso_{2r}$ can be identified with
  $W\simeq (\BZ/2\BZ)^{r-1}\rtimes S_r\simeq \{\pm 1\}^r_{+}\rtimes S_r$,
cf.~\eqref{eq:even-odd}, so that elements of $W$ are indexed by pairs $(\vec\mu,\sigma)$ with
$\sigma\in S_r$ and $\vec{\mu}\in \{\pm 1\}^r$ that are even ($|\vec\mu|=1$).

\medskip
\noindent
The Weyl group of the Levi subalgebra $\fl^+\simeq \gl_r$ of $\fp_{S^+}$ is $W_{\fl^+}\simeq S_r$, which consists
of the elements $\left((+1,\ldots,+1),\sigma\right)_{\sigma\in S_r}\subset W$. Equivalently, $W_{\fl^+}$
is the stabilizer of $(+1,\ldots,+1)$ for the natural transitive action of $W$ on the set $\{\pm 1\}^r_{+}$.
Thus, we have a set bijection
\begin{equation}\label{eq:set-bij D+}
  \pi_+ \colon W/W_{\fl^+} \ \iso \ \{\pm 1\}^r_{+} \,,
\end{equation}
cf.~\eqref{eq:set-bij C}. For any $\vec\mu\in \{\pm 1\}^r_+$,
we define $w_{\vec\mu}\in W\simeq \{\pm 1\}^r_{+}\rtimes S_r$ via:
\begin{equation}\label{eq:mu-to-w D+}
  w_{\vec\mu}=(\vec\mu,\sigma_{\vec\mu}) \,,
\end{equation}
with $\sigma_{\vec\mu}\in S_r$ as in~\eqref{eq:mu-to-sigma}. The element $w_{\vec\mu}\in \pi^{-1}_+(\vec\mu)$
can be characterized similarly to Lemma~\ref{lem:shortest decsription C}:

\medskip

\begin{Lem}\label{lem:shortest decsription D+}
$w_{\vec\mu}$ is the shortest representative of the left coset $\pi^{-1}_+(\vec\mu)$,
for any $\vec\mu\in \{\pm 1\}^r_{+}$.
\end{Lem}

\medskip

\begin{Cor}\label{left-coset D+}
${}^{\fl^+}W = \{w_{\vec\mu}\}_{\vec\mu\in \{\pm 1\}^r_{+}}$.
\end{Cor}

\medskip
\noindent
Likewise, the Weyl group of the Levi subalgebra $\fl^-\simeq \gl_r$ of $\fp_{S^-}$ is $W_{\fl^-}\simeq S_r$, which consists of
the elements $\left((+1,\ldots,+1,-1,+1,\ldots,+1,-1),\sigma\right)_{\sigma\in S_r}\subset W$ with $-1$'s at the $r$-th and
$\sigma(r)$-th spots (and there are no $-1$'s at all if $r=\sigma(r)$). Equivalently, $W_{\fl^-}$ is the stabilizer
of $(+1,\ldots,+1,-1)$ for the natural transitive action of $W$ on the set $\{\pm 1\}^r_{-}$. Thus, we have a set bijection
\begin{equation}\label{eq:set-bij D-}
  \pi_- \colon  W/W_{\fl^-} \ \iso \ \{\pm 1\}^r_{-} \,,
\end{equation}
cf.~\eqref{eq:set-bij D+}. For any $\vec\mu\in \{\pm 1\}^r_-$,
we define $w_{\vec\mu}\in W\simeq \{\pm 1\}^r_{+}\rtimes S_r$ via:
\begin{equation}\label{eq:mu-to-w D-}
  w_{\vec\mu}=(\vec\mu',\sigma_{\vec\mu}) \,,
\end{equation}
with $\sigma_{\vec\mu}\in S_r$ as in~\eqref{eq:mu-to-sigma} and $\vec\mu'\in \{\pm 1\}^r_{+}$
obtained from $\vec\mu$ by replacing the first $\mu_i=-1$ (with minimal $i$) by $+1$.
The above element $w_{\vec\mu}\in \pi^{-1}_-(\vec\mu)$ can be characterized similarly to Lemma~\ref{lem:shortest decsription D+}:

\medskip

\begin{Lem}\label{lem:shortest decsription D-}
$w_{\vec\mu}$ is the shortest representative of the left coset $\pi^{-1}_-(\vec\mu)$,
for any $\vec\mu\in \{\pm 1\}^r_{-}$.
\end{Lem}

\medskip

\begin{Cor}\label{left-coset D-}
${}^{\fl^-}W = \{w_{\vec\mu}\}_{\vec\mu\in \{\pm 1\}^r_{-}}$.
\end{Cor}

\medskip
\noindent
Combining now the formula~\eqref{eq:replD} with the definition of $w_{\vec\mu}\in W$,
see~(\ref{eq:mu-to-w D+},~\ref{eq:mu-to-w D-}), we conclude that the highest weight of the Fock vacuum
$|0\rangle\in M^+_{\vec\mu,\st}$ coincides with the highest weight of our key modules
$M'_{w_{\vec\mu}\, \cdot \, \st\omega^{|\vec\mu|}}$ introduced in~\eqref{eq:hard modules},
see Corollaries~\ref{left-coset D+},~\ref{left-coset D-} (we set $\omega^{\pm 1}=\omega^\pm$).
Furthermore, $M'_{w_{\vec\mu}\, \cdot \, \st\omega^{|\vec\mu|}}$ has the same character as
$M^+_{\vec\mu,\st}$ (according to Lemma~\ref{lem:M-char}) and is irreducible for $\st\notin \frac{1}{2}\BZ$
(as follows from~\cite{j}). Therefore, similarly to Propositions~\ref{lem:A-Fock as M-mod}
and~\ref{lem:C-Fock as M-mod}, we obtain:

\medskip

\begin{Prop}\label{lem:Dspin-Fock as M-mod}
For any $\vec\mu\in \{\pm 1\}^r$ and $\st\notin \sfrac{1}{2}\BZ$, we have $\sso_{2r}$-module isomorphisms:
\begin{equation}\label{eq:Dspin-Fock-as-M}
  M^+_{\vec\mu,\st} \simeq M'_{w_{\vec\mu}\, \cdot \, \st\omega^{|\vec\mu|}} \simeq
  \left(M'_{w_{\vec\mu}\, \cdot \, \st\omega^{|\vec\mu|}}\right)^\vee \,.
\end{equation}
\end{Prop}

\medskip

\begin{Rem}\label{rem:D-details}
Let us point out the key difference between Proposition~\ref{lem:Dspin-Fock as M-mod} and
Lemmas~\ref{lem:Dspin-Fock as par-Verma},~\ref{lem:Dspin-Fock-2nd as par-Verma}:

\medskip
\noindent
(a) For $\vec\mu=(+1,\ldots,+1,\pm 1)$, we actually have
$M^+_{\vec\mu,\st} \simeq (M'_{w_{\vec\mu}\, \cdot \, \st\omega^\pm})^\vee$
for any $\st\in \BC$, due to Lemmas~\ref{lem:Dspin-Fock as par-Verma},~\ref{lem:Dspin-Fock-2nd as par-Verma}.

\medskip
\noindent
(b) Likewise, for $\vec\mu=(-1,\ldots,-1,\mp 1)$, we have
$M^+_{\vec\mu,\st} \simeq M'_{w_{\vec\mu}\, \cdot \, \st\omega^{|\vec\mu|}}$
for any $\st\in \BC$.

\medskip
\noindent
(c) For other $\vec\mu \in \{\pm 1\}^r$, $M^+_{\vec\mu,\st}$ is \underline{not isomorphic} to either of
$M'_{w_{\vec\mu}\, \cdot \, \st\omega^{|\vec\mu|}}$ or $(M'_{w_{\vec\mu}\, \cdot \, \st\omega^{|\vec\mu|}})^\vee$
at some $\st\in \sfrac{1}{2}\BZ$ (but we expect it to be isomorphic to a twisted Verma module in the sense of~\cite{al}).
\end{Rem}

\medskip
\noindent
Evoking the above bijections $\{\pm 1\}^r_{\pm}\ni \vec\mu \leftrightarrow w_{\vec\mu}\in {}^{\fl^\pm}W$
of Corollaries~\ref{left-coset D+} and~\ref{left-coset D-}, let us define:
\begin{equation}\label{eq:DM-modified}
  M^\vee_{\vec\mu,\st}=\left(M'_{w_{\vec\mu}\, \cdot \, \st\omega^{|\vec\mu|}}\right)^\vee
  \,, \qquad \forall\, \st\in \BC \,.
\end{equation}
Then, Proposition~\ref{lem:Dspin-Fock as M-mod} can be recast as the isomorphism of the following $\sso_{2r}$-modules:
\begin{equation}\label{D-generic-coincidence}
  M^+_{\vec\mu,\st}\simeq M^\vee_{\vec\mu,\st} \,, \qquad \forall \, \st\in \BC \setminus \sfrac{1}{2}\BZ \,.
\end{equation}
\noindent

\medskip
\noindent
For $\vec\mu\in \{\pm 1\}^r$, we also define its \emph{length} $\mathsf{l}(\vec\mu)$ as the length of the corresponding
element $w_{\vec\mu}\in W$. Using formula~\eqref{eq:length-interpretation} and the explicit description
of the set $\Delta^+$ of positive roots of $\sso_{2r}$, we find:
\begin{equation}\label{eq:sigma-length D}
  \mathsf{l}(\vec\mu) = l(w_{\vec\mu}) = \sum_{i=1}^r (r-i)\delta_{\mu_i}^{-} \,,
\end{equation}
cf.~\eqref{eq:sigma-length C}.
We note that $\mathsf{l}(\vec\mu)$ differs from the lengths of $(\vec\mu,\mathrm{id}),(\vec\mu',\mathrm{id})\in W$
denoted by $l(\vec\mu),l(\vec\mu')$.

\medskip


\subsection{Type D transfer matrices}
$\ $

Recall the notion of transfer matrices $\{T_W(x)\}_{W\in \mathrm{Rep}\, Y(\sso_{2r})}$,
as discussed in Subsection~\ref{ssec R-matrix}.
In particular, we shall consider the following explicit infinite-dimensional transfer matrices:
\begin{equation}\label{eq:Dspin-inf-transfer}
  T_{\vec\mu,\st}^+(x)=\tr \prod_{i=1}^r \tau_i^{\CF^{\vec\mu}_{ii}}
  \underbrace{\CL_{\vec\mu}(x)\otimes \dots \otimes \CL_{\vec\mu}(x)}_{N} \,,
\end{equation}
corresponding to $M^+_{\vec\mu,\st}$.
For $\st\in \frac{1}{2}\BN$, we also consider the finite-dimensional transfer matrices $T^\pm_{\st}(x)$
corresponding to the modules $L^\pm_{\st}$~\eqref{eq:two-spin} in the auxiliary space:
those are defined similarly to~\eqref{eq:Dspin-inf-transfer}, but with the trace taken over
the finite-dimensional submodules $L^\pm_{\st}$ of $M^+_{(+1,\ldots,+1,\pm 1),\st}$,
see Corollaries~\ref{cor:t-special D}(b),~\ref{cor:t-special D-2nd}(b).

\medskip
\noindent
Using the notation~(\ref{eq:DM-modified},~\ref{eq:sigma-length D}), let us recast
the resolution~\eqref{eq:geometric resolution 1}, dual to~\eqref{eq:conjectured resolution 1}, as follows:
\begin{equation}\label{eq:Dspin-resolution}
  0\to L^\pm_{\st}\to M^\vee_{(1,\ldots,1,\pm 1),\st} \, \to \,
  \bigoplus_{\vec\mu\in \{\pm 1\}^r_{\pm}}^{\mathsf{l}(\vec\mu)=1} M^\vee_{\vec\mu,\st} \, \to \,
  \bigoplus_{\vec\mu\in \{\pm 1\}^r_{\pm}}^{\mathsf{l}(\vec\mu)=2} M^\vee_{\vec\mu,\st} \, \to \,
  \cdots \, \to \, 0
\end{equation}
for any $\st\in \frac{1}{2}\BN$. Combining this with~\eqref{D-generic-coincidence} and the fact that
\underline{the transfer matrices~\eqref{eq:Dspin-inf-transfer} depend} \underline{continuously on $\st\in \BC$}
(as so do the Lax matrices $\CL_{\vec\mu}(x)$), we obtain the key result of this section:

\medskip

\begin{Thm}\label{thm:Main-D}
For $\st\in \frac{1}{2}\BN$, we have:
\begin{equation}\label{eq:transfer-Dspin}
   T_{\st}^\pm(x)\, =  \sum_{\vec\mu\in \{\pm 1\}^r_{\pm}} (-1)^{\mathsf{l}(\vec{\mu})}\, T_{\vec\mu,\st}^+(x) \,.
\end{equation}
\end{Thm}

\medskip
\noindent
The character limit of~\eqref{eq:transfer-Dspin} expresses the character
of the $\sso_{2r}$-modules $\{L^\pm_{\st}\}_{\st\in \frac{1}{2}\BN}$ defined as
\begin{equation}\label{eq:Dspin-char-def}
  \ch^\pm_{\st}=\ch^\pm_{\st}(\tau_1, \ldots, \tau_r) :=
  \tr_{L^\pm_{\st}} \prod_{i=1}^r \tau_i^{\CF^{(+1,\ldots,+1,\pm 1)}_{ii}} \,,
\end{equation}
that is the length $N=0$ case of $T^\pm_{\st}(x)$, via:
\begin{equation}\label{eq:char-Dspin}
  \ch^\pm_{\st} =
  \sum_{\vec\mu=(\mu_1,\ldots,\mu_r)\in \{\pm 1\}^r_{\pm}} (-1)^{\mathsf{l}(\vec{\mu})}
  \frac{\prod_{i=1}^{r} \tau_i^{\mu_i\left(\st+(r-i-1)\delta_{\mu_i}^{-}+\sum_{k=1}^i \delta_{\mu_k}^{-}\right)}}
       {\prod_{1\leq i<j\leq r}\left(1-\frac{1}{\tau_i\tau_j^{\mu_i \mu_j}}\right)}
\end{equation}
with the $\vec\mu$'s summand in the right-hand side of~\eqref{eq:char-Dspin} equal to the character of
$M^+_{\vec\mu,\st}$, up to a sign.

\medskip

\begin{Rem}
The character formula~\eqref{eq:char-Dspin} can be derived directly from the Weyl character and Weyl denominator
formulas, as in Remark~\ref{rem:proof-C-char}; we leave details to the interested reader.
\end{Rem}

\medskip
\noindent
We note that the formula~\eqref{eq:char-Dspin} allows to \underline{analytically continue} the characters
$\ch^\pm_{\st}$ of~\eqref{eq:Dspin-char-def} from the discrete set $\st\in \frac{1}{2}\BN$ to
the entire complex plane $\st\in \BC$. With this convention in mind, we obtain:

\medskip

\begin{Lem}\label{lem:t-property D}
(a) $\ch^\pm_{\st}=(-1)^{\frac{r(r-1)}{2}}\ch^{\pm(-1)^r}_{-r+1-\st}$ for any $\st\in \BC$.

\medskip
\noindent
(b) $\ch^\pm_{\st}=0$ for $\st\in \Big\{-\frac{1}{2},-\frac{2}{2},\ldots,-\frac{2r-4}{2},-\frac{2r-3}{2}\Big\}$.
\end{Lem}

\medskip

\begin{proof}
The proof is completely analogous to that of Lemma~\ref{lem:t-property C} with the following two changes.
In part (a), we should rather use $\bar\st=-r+1-\st$ instead of~\eqref{eq:opposite-mu}, and note that for
$\vec\mu\in \{\pm 1\}^r_\pm$ of~\eqref{eq:even-odd} we have $\vec{\bar\mu}=-\vec\mu \in \{\pm 1\}^r_{\pm(-1)^r}$.
In part (b) for $\ch^+_{\st}$, the splitting of elements of $W$ into the pairs satisfying~\eqref{eq:pair compatibility}
is performed following only the second rule in our proof of Lemma~\ref{lem:t-property C}(b).
\end{proof}

\medskip
\noindent
In a completely similar way, the formula~\eqref{eq:transfer-Dspin} allows to \underline{analytically continue}
the transfer matrices $T^\pm_{\st}(x)$ of the finite-dimensional representations $L^\pm_{\st},\st\in \frac{1}{2}\BN$,
to the entire complex plane $\st\in \BC$. With this convention in mind, we have the following generalization of
Lemma~\ref{lem:t-property D}(a):

\medskip

\begin{Prop}\label{prop:t-symmetry D}
$T^\pm_{\st}(x)=(-1)^{\frac{r(r-1)}{2}} T^{{\pm(-1)^r}}_{-r+1-\st}(x)$ for any $\st\in \BC$.
\end{Prop}

\medskip

\begin{proof}
The proof is completely analogous to that of Proposition~\ref{prop:t-symmetry C} and follows
from the proof of Lemma~\ref{lem:t-property D}(a) combined with the factorisation~\eqref{eq:T-via-QQ D-type}
of the transfer matrices $T^+_{\vec\mu,\st}(x)$ into the product of two commuting $Q$-operators,
cf.\ Proposition~\ref{prop:main-D factorized}.
\end{proof}

\medskip
\noindent
We expect the generalization of Lemma~\ref{lem:t-property D}(b) to hold:
$T^\pm_{\st}(x)=0$ for $\st\in \Big\{-\frac{1}{2},-\frac{2}{2},\ldots,-\frac{2r-3}{2}\Big\}$.
This was first observed in~\cite{ffk} for small length and rank.

$\ $


\section{Resolutions for transfer matrices of BD-types: first fundamental representations}\label{sec:BDff}

In this section, we apply similar ideas to treat the remaining case of~\eqref{eq:KR-classification}:
$i=1$ for $BD$-types.


\subsection{Oscillator realization in types BD (parabolic Verma)}
$\ $

For $\NK\geq 5$, let $\CA$ denote the oscillator algebra generated by $\NK-2$ pairs of oscillators
$\{(\oa_i,\oad_i)\}_{1<i<1'}$, cf.\ notation~\eqref{eq:prime-index} so that $1'=\NK$, subject to the standard defining relations:
\begin{equation}\label{eq:oscillator relations 3}
  [\oa_{i},\oad_{j}]=\delta_{i}^{j} \,, \qquad
  [\oa_{i},\oa_{j}]=0 \,, \qquad [\oad_{i},\oad_{j}]=0 \,,
\end{equation}
so that
\begin{equation}
  \CA = \BC \Big\langle \oa_{i} \, , \, \oad_{i} \Big\rangle_{1<i<1'} \, \Big/ \, \eqref{eq:oscillator relations 3} \,.
\end{equation}
Following~\cite[(5.36, 5.38)]{f} and~\cite[(2.243), \S4.3]{Frassek:2021ogy}, let us consider the quadratic non-degenerate
$\CA[x]$-valued Lax matrices of $\sso_{\NK}$-type (i.e.\ of types $D_r$ or $B_r$ with $r=\lfloor \NK/2\rfloor$ for $\NK$ even
or odd, respectively), depending on $x_1,x_2\in \BC$:
\begin{equation}\label{Matrix Example D4 a-osc factroized}
  \fL_{x_1,x_2}(x)=
    \left(\begin{BMAT}[5pt]{c|c|c}{c|c|c}
      1 & \wp & -\tfrac{1}{2}\wp\idb_{\NK-2}\wp^T \\
      0 & \ID_{\NK-2} & -\idb_{\NK-2}\wp^T \\
      0 & 0 & 1 \\
    \end{BMAT}\right)
    \cdot \, D_{x_1,x_2}(x) \, \cdot
    \left(\begin{BMAT}[5pt]{c|c|c}{c|c|c}
      1 & -\wp & -\tfrac{1}{2}\wp\idb_{\NK-2}\wp^T \\
      0 & \ID_{\NK-2} & \idb_{\NK-2}\wp^T \\
      0 & 0 & 1 \\
    \end{BMAT}\right)
\end{equation}
with $\idb_k$ denoting the anti-diagonal $k\times k$-matrix and the middle factor explicitly given by:
\begin{equation*}
  D_{x_1,x_2}(x)=
    \left(\begin{BMAT}[5pt]{c|c|c}{c|c|c}
      (x-x_1)(x-x_1-\tfrac{\NK}{2}+2) & 0 & 0 \\
      -\wm_{}(x-x_1) & (x-x_1)(x-x_2)\ID_{\NK-2} & 0 \\
      -\tfrac{1}{2}\wm^T\idb_{\NK-2}\wm & \wm^T\idb_{\NK-2}(x-x_2) & (x-x_2)(x-x_2-\tfrac{\NK}{2}+2) \\
    \end{BMAT}\right) \,,
\end{equation*}
while the row-vector $\wp\in \mathrm{Mat}_{1\times (\NK-2)}(\CA)$ and the column-vector
$\wm \in \mathrm{Mat}_{(\NK-2)\times 1}(\CA)$ encode all the generators via:
\begin{equation}\label{eq:D-osc}
  D_r\colon \quad
  \wp=(\oad_{2},\ldots,\oad_{r},\oad_{r'},\ldots,\oad_{2'}) \,, \quad
  \wm=(\oa_{2},\ldots,\oa_{r},\oa_{r'},\ldots,\oa_{2'})^T \,,
\end{equation}
\begin{equation}\label{eq:B-osc}
  B_r\colon \quad
  \wp=(\oad_{2},\ldots,\oad_{r},\oad_{r+1},\oad_{r'},\ldots,\oad_{2'}) \,, \quad
  \wm=(\oa_{2},\ldots,\oa_{r},\oa_{r+1},\oa_{r'},\ldots,\oa_{2'})^T  \,.
\end{equation}

\medskip
\noindent
Following~\cite[Remark~4.37]{Frassek:2021ogy}, we also consider
\begin{equation}\label{eq:L vs fL}
  L_{x_{12}}(x)=\fL_{x_1,x_2}(x+c)=x^2\ID_{\NK}+xM_{x_{12}}+G_{x_{12}}
\end{equation}
with the shift $c$ of the spectral parameter given by:
\begin{equation}\label{eq:a-mass}
  c=\frac{x_1+x_2-1}{2} \,,
\end{equation}
and $x_{12}=x_1-x_2$ (note that the right-hand side of~\eqref{eq:L vs fL}
depends only on the difference of $x_1,x_2$).

\medskip
\noindent
It is straightforward to see that the linear term $M_{x_{12}}$ in~\eqref{eq:L vs fL} is given by:
\begin{equation}\label{eq:BCgens}
  M_{x_{12}} =
    \left(\begin{BMAT}[5pt]{c|c|c}{c|c|c}
      -x_{12}-\tfrac{\NK}{2}+1-\wp\wm & M_{[12]} & 0 \\
      -\wm & \wm\wp-\idb_{\NK-2}\wp^T\wm^T\idb_{\NK-2}-\id_{\NK-2} & M_{[23]} \\
      0 & \wm^T\idb_{\NK-2} & x_{12}+\tfrac{\NK}{2}-1+\wp\wm \\
    \end{BMAT}\right) \,,
\end{equation}
with
\begin{align}
  M_{[12]}&=\left(x_{12}+\sfrac{\NK}{2}-2+\wp\wm\right)\wp-\sfrac{1}{2}\wp\idb_{\NK-2}\wp^T\wm^T\idb_{\NK-2} \,, \\
  M_{[23]}&=-\left(x_{12}+\sfrac{\NK}{2}-2+\wp\wm\right)\idb_{\NK-2}\wp^T+\sfrac{1}{2}\wp\idb_{\NK-2}\wp^T\cdot \wm \,,
\end{align}
while the free term $G_{x_{12}}$ in~\eqref{eq:L vs fL} is expressed via the linear term $M_{x_{12}}$ as follows:
\begin{equation}
  G_{x_{12}}=\sfrac{1}{2}M_{x_{12}}^2+\sfrac{1}{4}(\NK-2)M_{x_{12}}+\sfrac{1}{4}(\NK-3-x_{12}^2)\id_{\NK} \,.
\end{equation}

\medskip
\noindent
As a direct consequence of the RTT relation~\eqref{eq:RTT}, we can identify the generators of $\sso_{\NK}$ through:
\begin{equation}\label{eq:F-for-BD}
  \CF_{ij}=\left(M_{1-\st-\frac{\NK}{2}}\right)_{ji} \,.
\end{equation}
In particular, we have:
\begin{equation}\label{eq:cartanD2}
  \CF_{11}=\st-\sum_{k=2}^{\NK-1} \oad_k \oa_k \,, \qquad
  \CF_{ii}=\oad_i\oa_{i}-\oad_{i'}\oa_{i'} \quad \mathrm{for} \quad 1<i\leq r \,.
\end{equation}
As before, let $\Fock$ denote the Fock module of $\CA$, generated by the Fock vacuum $|0\rangle\in \Fock$ satisfying:
\begin{equation}
  \oa_{i} |0\rangle = 0 \,, \qquad 1<i<1' \,.
\end{equation}
Then, the Fock vacuum $|0\rangle$ is obviously a highest weight state of the resulting $\sso_{\NK}$-action:
\begin{equation}
  \CF_{ij}|0\rangle=0 \,, \qquad 1\leq i<j\leq \NK \,,
\end{equation}
with the highest weight $\lambda=t\omega_1=(\st,\underbrace{0,\ldots,0}_{r-1})$, that is:
\begin{equation}
  \CF_{ii} |0\rangle = \st\delta_{i}^1 |0\rangle \,, \qquad 1\leq i\leq r \,.
\end{equation}
Completely similarly to Lemmas~\ref{lem:A-Fock as par-Verma},~\ref{lem:C-Fock as par-Verma},~
\ref{lem:Dspin-Fock as par-Verma},~\ref{lem:Dspin-Fock-2nd as par-Verma}, we have:

\medskip

\begin{Lem}\label{lem:BD-Fock as par-Verma}
There is an $\sso_{\NK}$-module isomorphism:
\begin{equation}\label{eq:BD-Fock-as-verma}
  \Fock \simeq \left(M^{\fp_{\{2,\ldots,r\}}}_{\st\omega_1}\right)^\vee \,.
\end{equation}
\end{Lem}

\medskip
\noindent
Completely similarly to Corollaries~\ref{cor:t-special A},~\ref{cor:t-special C},~\ref{cor:t-special D},
~\ref{cor:t-special D-2nd}, we thus get:

\medskip

\begin{Cor}\label{cor:t-special BD}
(a) For $\st\notin 4-\NK+\sfrac{1}{2}\BN$, the $\sso_{\NK}$-module $\Fock$ is irreducible.

\medskip
\noindent
(b) For $\st\in \BN$, the Fock vacuum $|0\rangle$ generates an irreducible
finite-dimensional $\sso_{\NK}$-module $L_{\st\omega_1}$.
\end{Cor}

\medskip


\subsection{More oscillator realizations in types BD via underlying symmetries}
\label{ssec BD all osc}
$\ $

Consider the following $2r$ endomorphisms of $\BC^{\NK}$:
\begin{equation}\label{eq:hat-B}
  \hat{B}_{k}=
    \begin{cases}
      \sum_{j=1}^{\NK} e_{jj} & \text{for } k=1 \\
      e_{1k}+e_{k1}+e_{1'k'}+e_{k'1'}+\sum_{1<j<1'}^{j\ne k,k'} e_{jj}
        & \text{for } 1<k\leq r \\[0.2cm]
      e_{1k}+e_{k1}+e_{1'k'}+e_{k'1'}+\sum_{1< j< 1'}^{j\ne k,k'} e_{jj'}
        & \text{for } r'\leq k<1' \\[0.2cm]
      \sum_{j=1}^{\NK} e_{jj'} & \text{for } k=1'
  \end{cases} \,.
\end{equation}
Since the $R$-matrix~\eqref{eq:BCD-Rmatrix} is invariant under such transformations, cf.~\eqref{eq:R-inv}:
\begin{equation}
  \Big[R(x),\hat{B}_{k}\otimes \hat{B}_{k}\Big]=0
  \,, \qquad \forall\, k\in \{1,\ldots,r\}\cup \{r',\ldots,1'\} \,,
\end{equation}
we can generate more solutions to the RTT relation~\eqref{eq:RTT} from the Lax matrix~\eqref{eq:L vs fL} via:
\begin{equation}\label{eq:LaxBDmod}
  \CL_{k}(x)=\left. \hat{B}_k L_{x_{12}=1-\st-\frac{\NK}{2}}(x)\hat{B}_k^{-1} \right|_{p.h.}
  \,, \qquad \forall\, k\in \{1,\ldots,r\}\cup \{r',\ldots,1'\} \,.
\end{equation}
Here, we apply the following particle-hole automorphism of $\CA$ (denoted \emph{p.h.}):
\begin{equation}\label{eq:pt-transformBD}
\begin{split}
  & \oad_{j}\mapsto -\oa_{j} \,, \quad \oa_{j}\mapsto\oad_{j} \, \quad
    \text{for}\quad 1<j\leq k \qquad \mathrm{if}\quad 1\leq k\leq r \,, \\
  & \oad_{j}\mapsto -\oa_{j} \,, \quad \oa_{j}\mapsto\oad_{j} \, \quad
    \text{for}\quad k'< j<1' \quad\ \mathrm{if}\quad r'\leq k\leq 1' \,,
\end{split}
\end{equation}
uniquely chosen to insure that the Fock vacuum $|0\rangle$ remains to be an $\sso_{\NK}$ highest weight state.

\medskip
\noindent
The resulting $\sso_{\NK}$ generators are read off the linear term of~\eqref{eq:LaxBDmod} in the spectral parameter:
\begin{equation}\label{eq:F-for-BD general}
  \CF_{ij}^{k}=
  \left(\left.\hat B_k M_{1-\st-\frac{\NK}{2}}\hat B_k^{-1}\right|_{p.h.}\right)_{ji}
  \,, \qquad k\in \{1,\ldots,r\}\cup \{r',\ldots,1'\} \,.
\end{equation}
This makes the Fock module $\Fock$ into an $\sso_{\NK}$-module, denoted by $M^+_{k,\st}$.
For $1\leq i\leq r$, we get:
\begin{equation}\label{eq:F-for-BD action}
  \CF^{k}_{ii} |0\rangle=
    \begin{cases}
      \Big((\st+k-1)\delta_{i}^{k}-\delta_{i<k}\Big) |0\rangle & \text{for } 1\leq k\leq r \\
      \Big((-\st-k+2)\delta_{i}^{k'}-\delta_{i<k'}\Big) |0\rangle & \text{for } r'\leq k\leq 1'
    \end{cases} \,,
\end{equation}
thus the corresponding highest weight of $|0\rangle \in M^+_{k,\st}$ is:
\begin{equation}\label{eq:replBD}
\begin{split}
  & (\underbrace{-1,\ldots,-1}_{k-1},\st+k-1,\underbrace{0,\ldots,0}_{r-k})
    \qquad \mathrm{for}\quad 1\leq k\leq r \,, \\
  & (\underbrace{-1,\ldots,-1}_{k'-1},-\st-k+2,\underbrace{0,\ldots,0}_{r-k})
    \qquad \mathrm{for}\quad r'\leq k \leq 1' \,.
\end{split}
\end{equation}

\medskip
\noindent
We shall now compare the above modules $M^+_{k,\st}$'s with those from the Introduction.
To this end, let us consider the parabolic $\fp_S\subset \sso_{\NK}$ corresponding to $S=\{2,\ldots,r\}$
with $r=\lfloor \frac{\NK}{2}\rfloor$, see~\S\ref{ssec truncated-BGG}. The Weyl group of $\sso_{\NK}$ can be identified
with $W\simeq \{\pm 1\}^r\rtimes S_r$ for $\NK=2r+1$ or $W\simeq \{\pm 1\}^r_{+}\rtimes S_r$ for $\NK=2r$,
cf.~\eqref{eq:even-odd}. The Weyl group of the Levi subalgebra $\fl\simeq \sso_{\NK-2}\oplus \gl_1$ is
$W_{\fl}\simeq \{\pm 1\}^{r-1}\rtimes S_{r-1}$ for $\NK=2r+1$ or $W\simeq \{\pm 1\}^{r-1}_{+}\rtimes S_{r-1}$
for $\NK=2r$, consisting of those $(\vec\mu,\sigma)\in W$ such that $\mu_1=+1$ and $\sigma(1)=1$.
Equivalently, $W_{\fl}$ is the stabilizer of $1$ for the natural transitive action of the Weyl group $W$
on the set $\{1\ldots,r\} \cup \{r',\ldots,1'\}$. Thus, we have a set bijection
\begin{equation}\label{eq:set-bij BD}
  \pi \colon W/W_{\fl} \ \iso \ \{1,\ldots,r\} \cup \{r',\ldots,1'\} \,,
\end{equation}
cf.~(\ref{eq:set-bij},~\ref{eq:set-bij C},~\ref{eq:set-bij D+},~\ref{eq:set-bij D-}).
For any $1\leq k\leq r$, we define the permutation $\sigma_k\in S_r$ as $\sigma_{\{k\}}$ of~\eqref{eq:special permutation}:
\begin{equation}\label{eq:k-to-sigma 1}
  \sigma_k(1)=k \,,\ \sigma_k(2)=1 \,,\ \ldots \,, \ \sigma_k(k)=k-1 \,,\
  \sigma_k(k+1)=k+1 \,,\ \ldots \,,\ \sigma_k(r)=r \,,
\end{equation}
and further consider $w_{k}\in W$ given by:
\begin{equation}\label{eq:k-to-w 1}
  w_{k}=\Big( (+1,\ldots,+1),\sigma_k \Big) \,, \qquad 1\leq k\leq r \,.
\end{equation}
The element $w_{k}\in \pi^{-1}(k)$ can be characterized similarly to
Lemmas~\ref{lem:shortest decsription C},~\ref{lem:shortest decsription D+},~\ref{lem:shortest decsription D-}:

\medskip

\begin{Lem}\label{lem:shortest decsription BD 1}
$w_k$ is the shortest representative of the left coset $\pi^{-1}(k)$, for any $1\leq k\leq r$.
\end{Lem}

\medskip
\noindent
Likewise, for $1\leq k\leq r$, we also define $w_{k'}\in W$ via:
\begin{equation}\label{eq:k-to-w 2}
  w_{k'}=\left(\mu(k),\sigma_{k}\right) \,, \qquad 1\leq k\leq r \,,
\end{equation}
with $\sigma_{k}\in S_r$ as in~\eqref{eq:k-to-sigma 1} and $\mu(k)\in \{\pm 1\}^r$ having $-1$ components only at the:
\begin{enumerate}
  \item[(1)] $k$-th spot, if $\NK=2r+1$;
  \item[(2)] $k$-th and $r$-th spots, if $\NK=2r$ and $k<r$;
  \item[(3)] $(r-1)$-th and $r$-th spots, if $\NK=2r$ and $k=r$.
\end{enumerate}
Then, similarly to Lemma~\ref{lem:shortest decsription BD 1}, we have the following characterization of such elements:

\medskip

\begin{Lem}\label{lem:shortest decsription BD 2}
$w_k$ is the shortest representative of the left coset $\pi^{-1}(k)$, for any $r'\leq k\leq 1'$.
\end{Lem}

\medskip
\noindent
Combining Lemmas~\ref{lem:shortest decsription BD 1} and~\ref{lem:shortest decsription BD 2} with
the set bijection~\eqref{eq:set-bij BD} and~\eqref{eq:shortest left equiv}, we get:

\medskip

\begin{Cor}\label{BD-left-coset}
${}^{\fl}W = \Big\{w_{k} \, \Big|\, k\in \{1,\ldots,r\}\cup \{r',\ldots, 1'\} \Big\}$.
\end{Cor}

\medskip
\noindent
Combining now the formula~\eqref{eq:replBD} with the definition of $w_{k}\in W$,
see~(\ref{eq:k-to-w 1},~\ref{eq:k-to-w 2}), we conclude that the highest weight of the Fock vacuum
$|0\rangle\in M^+_{k,\st}$ coincides with the highest weight of our key modules
$M'_{w_{k}\, \cdot \, \st\omega_1}$ introduced in~\eqref{eq:hard modules}, see Corollary~\ref{BD-left-coset},
for any $k\in \{1,\ldots,r\} \cup \{r',\ldots,1'\}$.
Furthermore, $M'_{w_k\, \cdot \, \st\omega_1}$ has the same character as $M^+_{k,\st}$
(according to Lemma~\ref{lem:M-char}) and is irreducible for $\st\notin \BZ$ (as follows from~\cite{j}).
Therefore, similarly to Propositions~\ref{lem:A-Fock as M-mod},~\ref{lem:C-Fock as M-mod},~\ref{lem:Dspin-Fock as M-mod},
we obtain:

\medskip

\begin{Prop}\label{lem:BD-Fock as M-mod}
For $k\in  \{1,\ldots,r\} \cup \{r',\ldots,1'\}$ and $\st\notin \BZ$, we have $\sso_{\NK}$-module isomorphisms:
\begin{equation}\label{eq:BD-Fock-as-M}
  M^+_{k,\st} \simeq M'_{w_{k}\, \cdot \, \st\omega_1} \simeq \left(M'_{w_{k}\, \cdot \, \st\omega_1}\right)^\vee \,.
\end{equation}
\end{Prop}

\medskip

\begin{Rem}\label{rem:BD-details}
Let us point out the key difference between Proposition~\ref{lem:BD-Fock as M-mod} and Lemma~\ref{lem:BD-Fock as par-Verma}:

\medskip
\noindent
(a) For $k=1$, we actually have $M^+_{k,\st} \simeq (M'_{w_{k}\, \cdot \, \st\omega_1})^\vee$
for any $\st\in \BC$, due to Lemma~\ref{lem:BD-Fock as par-Verma}.

\medskip
\noindent
(b) Likewise, for $k=1'$, we have $M^+_{k,\st} \simeq M'_{w_{k}\, \cdot \, \st\omega_1}$ for any $\st\in \BC$.

\medskip
\noindent
(c) For other values of $k$, $M^+_{k,\st}$ is \underline{not isomorphic} to either of
$M'_{w_{k}\, \cdot \, \st\omega_1}$ or $(M'_{w_{k}\, \cdot \, \st\omega_1})^\vee$ at certain
$\st\in \BZ$ (but is expected to be isomorphic to one of the twisted Verma modules in the sense of~\cite{al}).
\end{Rem}

\medskip
\noindent
Evoking the above bijection $\{1,\ldots,r\} \cup \{r',\ldots,1'\}\ni k \leftrightarrow w_{k}\in {}^{\fl}W$
of Corollary~\ref{BD-left-coset}, we define:
\begin{equation}\label{eq:BDM-modified}
  M^\vee_{k,\st}=\left(M'_{w_{k}\, \cdot \, \st\omega_1}\right)^\vee \,, \qquad \forall\, \st\in \BC \,.
\end{equation}
Then, Proposition~\ref{lem:BD-Fock as M-mod} can be recast as the isomorphism of the following $\sso_{\NK}$-modules:
\begin{equation}\label{BD-generic-coincidence}
  M^+_{k,\st}\simeq M^\vee_{k,\st} \,, \qquad \forall \, \st\in \BC \setminus \BZ \,.
\end{equation}

\medskip
\noindent
For the key results of the following subsection, let us record the lengths of the above elements:
\begin{equation}\label{eq:k-length BD}
  l(w_k)=
  \begin{cases}
    k-1 \qquad & \text{for} \ 1\leq k \leq r \\
    k-2 \qquad & \text{for} \ r'\leq k\leq 1' \\
  \end{cases} \,,
\end{equation}
which follows from~\eqref{eq:length-interpretation} and the explicit description of the set $\Delta^+$
of positive roots of $\sso_{\NK}$.

\medskip


\subsection{Type BD transfer matrices}
$\ $

Recall the notion of transfer matrices $\{T_W(x)\}_{W\in \mathrm{Rep}\, Y(\sso_{\NK})}$,
as discussed in Subsection~\ref{ssec R-matrix}.
In particular, we shall consider the following explicit infinite-dimensional transfer matrices:
\begin{equation}\label{eq:BD-inf-transfer}
  T_{k,\st}^+(x)=\tr \prod_{i=1}^r \tau_i^{\CF^{k}_{ii}} \underbrace{\CL_{k}(x) \otimes \dots \otimes \CL_{k}(x)}_{N} \,,
\end{equation}
corresponding to $M^+_{k,\st}$. For $\st\in \BN$, we also consider the finite-dimensional
transfer matrices $T_{1,\st}(x)$ corresponding to the modules $L_{\st\omega_1}$ in the auxiliary space:
they are defined similarly to~\eqref{eq:BD-inf-transfer}, but with the trace taken over the
finite-dimensional submodule $L_{\st\omega_1}$ of $M^+_{1,\st}$, see Corollary~\ref{cor:t-special BD}(b).

\medskip
\noindent
Using the notation~\eqref{eq:BDM-modified} and the formula~\eqref{eq:k-length BD}, let us recast
the resolution~\eqref{eq:geometric resolution 1}, dual to~\eqref{eq:conjectured resolution 1}:
\begin{equation}\label{eq:D-first-resolution}
  D_r\colon \
  0\to L_{\st\omega_1}\to M^\vee_{1,\st} \to \cdots \to M^\vee_{r-1,\st}\to M^\vee_{r,\st}\oplus M^\vee_{r',\st}
  \to M^\vee_{(r-1)',\st}\to \cdots \to M^\vee_{1',\st}\to 0 \,,
\end{equation}
\begin{equation}\label{eq:B-resolution}
  B_r\colon \
  0\to L_{\st\omega_1}\to M^\vee_{1,\st} \to \cdots \to M^\vee_{r-1,\st}\to M^\vee_{r,\st}\to M^\vee_{r',\st}
  \to M^\vee_{(r-1)',\st}\to \cdots \to M^\vee_{1',\st}\to 0 \,,
\end{equation}
for any $\st\in \BN$. Combining them with~\eqref{BD-generic-coincidence} and the fact that
\underline{the transfer matrices~\eqref{eq:BD-inf-transfer} depend} \underline{continuously on $\st\in \BC$}
(as so do the Lax matrices $\CL_{k}(x)$), we obtain the key result of this section:

\medskip

\begin{Thm}\label{thm:Main-BD}
(a) For $\st\in \BN$, we have the following equality of $D_r$-type transfer matrices:
\begin{equation}\label{eq:transfer-D-first}
  T_{1,\st}(x)\, =  \sum_{k=1}^r (-1)^{k-1} T_{k,\st}^+(x) \, + \, \sum_{k=1}^r (-1)^{k-1} T_{k',\st}^+(x) \,.
\end{equation}

\medskip
\noindent
(b) For $\st\in \BN$, we have the following equality of $B_r$-type transfer matrices:
\begin{equation}\label{eq:transfer-B-first}
  T_{1,\st}(x)\, = \sum_{k=1}^r (-1)^{k-1} T_{k,\st}^+(x) \, + \, \sum_{k=1}^r (-1)^{k} T_{k',\st}^+(x) \,.
\end{equation}
\end{Thm}

\medskip
\noindent
The character limit of~\eqref{eq:transfer-D-first} expresses the character
of the $\sso_{2r}$-modules $\{L_{\st \omega_1}\}_{\st\in \BN}$ defined as
\begin{equation}\label{eq:Dfirst-char-def}
  \ch_{1,\st}=\ch_{1,\st}(\tau_1, \ldots, \tau_r) := \tr_{L_{\st\omega_1}} \prod_{i=1}^r \tau_i^{\CF_{ii}} \,,
\end{equation}
that is the length $N=0$ case of $T_{1,\st}(x)$, via:
\begin{equation}\label{eq:char-D-first}
\begin{split}
  \ch_{1,\st} =
  & \sum_{k=1}^r (-1)^{k-1}
    \frac{\tau_1^{-1}\cdots \tau_{k-1}^{-1}\tau_k^{\st+k-1}}
         {\prod_{1\leq \ell<k}\left(1-\frac{\tau_k}{\tau_\ell}\right)
          \prod_{k< \ell \leq r}\left(1-\frac{\tau_\ell}{\tau_k}\right)
          \prod_{\ell\ne k}\left(1-\frac{1}{\tau_k\tau_\ell}\right)} \, + \\
  & \sum_{k=1}^r (-1)^{k-1}
    \frac{\tau_1^{-1}\cdots \tau_{k-1}^{-1}\tau_k^{k+1-2r-\st}}
         {\prod_{1\leq \ell<k}\left(1-\frac{\tau_k}{\tau_\ell}\right)
          \prod_{k< \ell \leq r}\left(1-\frac{\tau_\ell}{\tau_k}\right)
          \prod_{\ell\ne k}\left(1-\frac{1}{\tau_k\tau_\ell}\right)}  \,.
\end{split}
\end{equation}
Likewise, the character limit of~\eqref{eq:transfer-B-first} expresses the character
of the $\sso_{2r+1}$-modules $\{L_{\st \omega_1}\}_{\st\in \BN}$
\begin{equation}\label{eq:Bfirst-char-def}
  \ch_{1,\st}=\ch_{1,\st}(\tau_1, \ldots, \tau_r) := \tr_{L_{\st\omega_1}} \prod_{i=1}^r \tau_i^{\CF_{ii}} \,,
\end{equation}
that is the length $N=0$ case of $T_{1,\st}(x)$, via:
\begin{equation}\label{eq:char-B-first}
\begin{split}
  \ch_{1,\st} =
  & \sum_{k=1}^r (-1)^{k-1}
    \frac{\tau_1^{-1}\cdots \tau_{k-1}^{-1}\tau_k^{\st+k-1}}
         {\left(1-\frac{1}{\tau_k}\right)
          \prod_{1\leq \ell<k}\left(1-\frac{\tau_k}{\tau_\ell}\right)
          \prod_{k< \ell \leq r}\left(1-\frac{\tau_\ell}{\tau_k}\right)
          \prod_{\ell\ne k}\left(1-\frac{1}{\tau_k\tau_\ell}\right)} \, + \\
  & \sum_{k=1}^r (-1)^{k}
    \frac{\tau_1^{-1}\cdots \tau_{k-1}^{-1}\tau_k^{k-2r-\st}}
         {\left(1-\frac{1}{\tau_k}\right)
          \prod_{1\leq \ell<k}\left(1-\frac{\tau_k}{\tau_\ell}\right)
          \prod_{k< \ell \leq r}\left(1-\frac{\tau_\ell}{\tau_k}\right)
          \prod_{\ell\ne k}\left(1-\frac{1}{\tau_k\tau_\ell}\right)} \,.
\end{split}
\end{equation}
Here, the $k$-th summand in the first (resp.\ second) sums in the right-hand side
of~(\ref{eq:char-D-first},~\ref{eq:char-B-first})
is equal to the character of $M^+_{k,\st}$ (resp.\ $M^+_{k',\st}$), up to a sign.

\medskip

\begin{Rem}
For the physics' reader who skipped Section~\ref{sec: truncated BGG resolutions}, let us present a concise proof
of~\eqref{eq:char-B-first} (the proof of~\eqref{eq:char-D-first} is completely analogous).
Let us identify the set $\Delta^+$ of positive roots of $\fg=\sso_{2r+1}$ with
  $\Delta^+ = \{\epsilon_i\pm \epsilon_j\}_{1\leq i<j\leq r} \cup \{\epsilon_i\}_{i=1}^r$,
so that $\rho=(r-\frac{1}{2},r-\frac{3}{2},\ldots,\frac{1}{2})$ and the Weyl group $W$ is
$W\simeq (\BZ/2\BZ)^r\rtimes S_r\simeq \{\pm 1\}^r\rtimes S_r$.
According to the Weyl character formula, we have:
\begin{equation}\label{eq:Weyl B}
  \ch_{L_{\st\omega_1}}\, =
  \sum_{(\vec{\mu},\sigma)\in \{\pm 1\}^r\rtimes S_r} (-1)^{l(\vec\mu,\sigma)}
  \frac{e^{(\vec\mu,\sigma)(\st\omega_1+\rho)-\rho}}
       {\prod_{1\leq i<j\leq r} (1-e^{\epsilon_j-\epsilon_i}) (1-e^{-\epsilon_j-\epsilon_i})
        \prod_{i=1}^r (1-e^{-\epsilon_i})} \,.
\end{equation}
Following~\S\ref{ssec BD all osc}, let us consider the parabolic subalgebra
$\fp_{\{2,\ldots,r\}}\subset\fg$ whose Levi subalgebra is $\fl\simeq \sso_{2r-1}\oplus \gl_1$ and
the Weyl group is $W_{\fl}\simeq \{\pm 1\}^{r-1}\rtimes S_{r-1}$ consisting of
$(\vec\mu,\sigma)\in \{\pm 1\}^{r}\rtimes S_{r}=W$ such that $\mu_1=1$ and $\sigma(1)=1$.
The assignment $W\ni (\vec\mu,\sigma)\mapsto (\mu_1,\sigma(1))\in \{\pm 1\}\times \{1,\ldots,r\}$
gives rise to a set bijection $\pi\colon W/W_{\fl} \, \iso \, \{\pm 1\}\times \{1,\ldots,r\}$,
cf.~\eqref{eq:set-bij BD}. For any $(\mu,k)\in \{\pm 1\}\times \{1,\ldots,r\}$, we consider
$((\mu,1,\ldots,1),\sigma_k)\in \pi^{-1}(\mu,k)$ with $\sigma_k\in S_r$ as in~\eqref{eq:k-to-sigma 1}.
We can rewrite~\eqref{eq:Weyl B} as:
\begin{equation}\label{eq:Weyl B doubled}
  \ch_{L_{\st\omega_1}}\, =
  \sum_{(\mu,k)\in \{\pm 1\}\times \{1,\ldots,r\}} \sum_{(\vec\nu,\tau)\in W_{\fl}}
  \frac{(-1)^{l((\mu,\vec\nu),\sigma_k\tau)} e^{((\mu,\vec\nu),\sigma_k\tau)(\st\omega_1+\rho)-\rho}}
       {\prod_{1\leq i<j\leq r}(1-e^{\epsilon_j-\epsilon_i})(1-e^{-\epsilon_j-\epsilon_i})
        \prod_{i=1}^r (1-e^{-\epsilon_i})} \,,
\end{equation}
where $(\mu,\vec\nu)\in \{\pm 1\}^r$ is obtained by attaching $\mu\in \{\pm 1\}$ on the left to $\vec\nu\in\{\pm 1\}^{r-1}$.
The key step is to simplify the inner sum of~\eqref{eq:Weyl B doubled} using the Weyl denominator formula for $\fl$:
\begin{equation}\label{eq:Weyl denominator B}
  \sum_{(\vec\nu,\tau) \in W_{\fl}} (-1)^{l(\vec\nu,\tau)} e^{(\vec\nu,\tau)(\rho_\fl)-\rho_\fl} \, =
  \prod_{\alpha\in \Delta^+_\fl} (1-e^{-\alpha}) \,,
\end{equation}
where $\Delta^+_{\fl}=\{\epsilon_i\pm \epsilon_j\}_{2\leq i<j\leq r} \cup \{\epsilon_i\}_{i=2}^{r}\subset \Delta^+$
consists of positive roots of $\fl$,
  $\rho_{\fl}=\frac{1}{2}\sum_{\alpha\in \Delta^+_{\fl}} \alpha = (0,r-\frac{3}{2},\ldots,\frac{1}{2})$.
As $(\vec\nu,\tau)(\rho)-\rho=(\vec\nu,\tau)(\rho_{\fl})-\rho_{\fl}$, $(\vec\nu,\tau)(\omega_1)=\omega_1$
for any $(\vec\nu,\tau)\in W_{\fl}\subset W$, we~get:
\begin{multline}
  \sum_{(\vec\nu,\tau)\in W_{\fl}} (-1)^{l(\vec\nu,\tau)}
  \frac{e^{(\vec\nu,\tau)(\st\omega_1+\rho)-\rho}}
       {\prod_{1\leq i<j\leq r}(1-e^{\epsilon_j-\epsilon_i})(1-e^{-\epsilon_j-\epsilon_i})
        \prod_{i=1}^r (1-e^{-\epsilon_i})}\, = \\
  \frac{e^{\st\epsilon_1}}
       {(1-e^{-\epsilon_1})\prod_{j=2}^r (1-e^{\epsilon_j-\epsilon_1})(1-e^{-\epsilon_j-\epsilon_1})} \,.
\end{multline}
Thus, the inner sum of~\eqref{eq:Weyl B doubled} indexed by $(\mu,k)=(1,1)$ gives rise to the $k=1$ summand
of the first sum in~\eqref{eq:char-B-first}. Likewise, we claim that the inner sum of~\eqref{eq:Weyl B doubled}
indexed by $(\mu,k)$ with $\mu=1$ (resp.\ $\mu=-1$) precisely recovers the $k$-th summand of the first (resp.\ second)
sum in~\eqref{eq:char-B-first}, which amounts to the proof of~\eqref{eq:char derivation B+} below (and its $\mu=-1$ counterpart).
To prove this claim for $\mu=1$, let us apply $((+1,\ldots,+1),\sigma_k)\in W$ to
both sides of the equality~\eqref{eq:Weyl denominator B}:
\begin{multline}
  \sum_{(\vec\nu,\tau)\in W_{\fl}}
  (-1)^{l((1,\vec\nu),\sigma_k\tau)} e^{((1,\vec\nu),\sigma_k\tau)(\rho)-\rho} \, =\\
  (-1)^{l(\sigma_k)}e^{\sigma_k(\rho)-\rho}
  \prod_{1\leq i\leq r}^{i\ne k} (1-e^{-\epsilon_i})
  \prod_{1\leq i<j\leq r}^{i,j\ne k} (1-e^{\epsilon_j-\epsilon_i})(1-e^{-\epsilon_j-\epsilon_i}) \,.
\end{multline}
Combining this with the formulas
\begin{equation*}
  \sigma_k(\rho)-\rho=(k-1)\epsilon_k - (\epsilon_1+\cdots+\epsilon_{k-1}) \,, \qquad
  l(\sigma_k)=k-1 \,, \qquad \sigma_k(\omega_1)=\epsilon_k \,,
\end{equation*}
we obtain the desired equality:
\begin{multline}\label{eq:char derivation B+}
  \sum_{(\vec\nu,\tau)\in W_{\fl}}
  (-1)^{l((1,\vec\nu),\sigma_k\tau)}
  \frac{e^{((1,\vec\nu),\sigma_k\tau)(\st\omega_1+\rho)-\rho}}
       {\prod_{1\leq i<j\leq r}(1-e^{\epsilon_j-\epsilon_i})(1-e^{-\epsilon_j-\epsilon_i})
        \prod_{i=1}^r (1-e^{-\epsilon_i})} = \\
  (-1)^{k-1}
  \frac{e^{-\epsilon_1}\cdots e^{-\epsilon_{k-1}}e^{(\st+k-1)\epsilon_k}}
       {(1-e^{-\epsilon_k}) \prod_{\ell=1}^{k-1}(1-e^{\epsilon_k-\epsilon_\ell})
        \prod_{\ell=k+1}^{r}(1-e^{\epsilon_\ell-\epsilon_k})
        \prod_{1\leq \ell\leq r}^{\ell\ne k} (1-e^{-\epsilon_k-\epsilon_\ell})} \,.
\end{multline}
The proof of the above claim for $\mu=-1$ is completely analogous with the only difference that:
\begin{equation*}
  ((-1,+1,\ldots,+1),\sigma_k)(\rho)-\rho = (k-2r)\epsilon_k - (\epsilon_1+\cdots+\epsilon_{k-1}) \,, \qquad
  ((-1,+1,\ldots,+1),\sigma_k)(\omega_1)=-\epsilon_k \,.
\end{equation*}
This completes our direct proof of the character formula~\eqref{eq:char-B-first}
(see~Remark~\ref{rem:math-to-physics} for more details in regards to
perceiving $\ch_{1,\st}$ of~\eqref{eq:Bfirst-char-def} as a specialization of $\ch_{L_{\st\omega_1}}$).
\end{Rem}

\medskip
\noindent
Let us note that the formulas~(\ref{eq:char-D-first}) and~(\ref{eq:char-B-first}) allow to
\underline{analytically continue} the character $\ch_{1,\st}$
of~(\ref{eq:Dfirst-char-def}) and~(\ref{eq:Bfirst-char-def}), from the discrete set $\st\in \BN$
to the entire complex plane $\st\in \BC$. With these conventions in mind
and similarly to Lemmas~\ref{lem:t-property C},~\ref{lem:t-property D}, we obtain:

\medskip

\begin{Lem}\label{lem:t-property BD}
(a) $\ch_{1,\st}=(-1)^{\NK}\ch_{1,2-\NK-\st}$ for any $\st\in \BC$.

\medskip
\noindent
(b) $\ch_{1,\st}=0$ for $\st\in \{-1,-2,\ldots,3-\NK\}$.
\end{Lem}

\medskip
\noindent
In a completely similar way, the formulas~(\ref{eq:transfer-D-first},~\ref{eq:transfer-B-first}) allow to
\underline{analytically continue} the transfer matrices $T_{1,\st}(x)$ of the finite-dimensional representations
$L_{\st\omega_1},\st\in\BN$, to the entire complex plane $\st\in \BC$. With this convention in mind, we have the following
generalization of Lemma~\ref{lem:t-property BD}(a):

\medskip

\begin{Prop}\label{prop:t-symmetry BD}
$T_{1,\st}(x)=(-1)^{\NK}T_{1,2-\NK-\st}(x)$ for any $\st\in \BC$.
\end{Prop}

\medskip

\begin{proof}
The proof is completely analogous to that of Proposition~\ref{prop:t-symmetry C} and follows from
the proof of Lemma~\ref{lem:t-property BD}(a) combined with the factorisation~\eqref{eq:T-via-QQ BD-type} of
the transfer matrices $T^+_{k,\st}(x), T^+_{k',\st}(x)$ into the product of two commuting $Q$-operators,
cf.\ Proposition~\ref{prop:main-BD factorized}.
\end{proof}

\medskip
\noindent
We also expect the generalization of Lemma~\ref{lem:t-property BD}(b) to hold: $T_{1,\st}(x)=0$ for
$\st\in \{-1,-2,\ldots,3-\NK\}$. For $D$-type, this was first observed in~\cite{ffk} for small length and rank.

$\ $


\section{Factorisation for linear ACD-types}\label{sec linear factorizations}

In this section, we demonstrate the factorisation of the infinite-dimensional transfer matrices
(\ref{eq:Tpgln},~\ref{eq:C-inf-transfer},~\ref{eq:Dspin-inf-transfer}) into
the products of two Baxter $Q$-operators arising from degenerate Lax matrices (which are \emph{renormalized limits}
of the former), linear in the spectral parameter.
The factorisation formula is universal for all three types $ACD$, and we present it
in full detail for the case of $A$-type.

\medskip


\subsection{General two-term linear factorisation}\label{ssec general two-factorisation}
$\ $

Consider the following two $n\times n$ matrices written in the block form as:
\begin{equation}\label{eq:two matrices}
  L_a(x)=\,
    \left(\begin{BMAT}[5pt]{c:c}{c:c}
      x\ID_{a}-\ap_1\am_1 & \ap_1 \\
      -\am_1 & \ID_{n-a}
    \end{BMAT}\right) \,, \qquad
  \bL_a(y)=\,
    \left(\begin{BMAT}[5pt]{c:c}{c:c}
      \ID_a & \ap_2 \\
      \am_2 & y\ID_{n-a} + \am_2\ap_2
    \end{BMAT}\right) \,,
\end{equation}
with $a\times (n-a)$ upper-right blocks $\ap_1,\ap_2$ and $(n-a)\times a$ lower-left blocks $\am_1,\am_2$.
Then, their product can be factorised as follows:
\begin{equation}\label{eq:linear factorization}
  L_a(x)\bL_a(y)=\,
    \left(\begin{BMAT}[5pt]{c:c}{c:c}
      x\ID_{a}-\ap'_1\am'_1 & \left(y-x+\ap'_1\am'_1\right)\ap'_1 \\
      - \am'_1 & y\ID_{n-a}+\am_1'\ap'_1
    \end{BMAT}\right)\,
    \left(\begin{BMAT}[5pt]{c:c}{c:c}
      \ID_{a} & \ap'_2\\
      0 & \ID_{n-a}
    \end{BMAT}\right) \,,
\end{equation}
where
\begin{equation}\label{eq:tra12}
\begin{split}
  & \am'_1=\am_1-\am_2 \,, \qquad \am'_2=\am_2 \,, \\
  & \ap'_2=\ap_2+\ap_1 \,, \qquad \ap'_1=\ap_1 \,.
\end{split}
\end{equation}
We note that the right-hand side of~\eqref{eq:linear factorization} is independent of $\am'_2$.

\medskip


\subsection{Two-term factorisation in A-type}
\label{ssec two-term A-factorization}
$\ $

Let $\CA$ denote the oscillator algebra generated by $n(n-1)$ pairs of oscillators
$\{(\oa_{j,i},\oad_{i,j})\}_{1\leq i\ne j\leq n}$ subject to~\eqref{eq:oscillator relations}.
For any subset $I\in \CS_a$, see our notation~\eqref{eq:CS-set}, recall the permutation $\sigma_I$ of the set
$\{1,\ldots,n\}$ defined in~\eqref{eq:special permutation} and the corresponding permutation matrix $B_I$
of~\eqref{eq:permutation-matrix}. We define:
\begin{equation}\label{eq:matricesx 1}
  L_{I}(x)=\left. B_IL_{\{1,\ldots,a\}}(x)B_I^{-1} \right|_{p.h.} \,,
\end{equation}
where
\begin{equation}\label{eq:matricesx 2}
  L_{\{1,\ldots,a\}}(x)=\,
    \left(\begin{BMAT}[5pt]{c:c}{c:c}
      x\ID_{a}-\ap\am & \ap \\
      -\am & \ID_{n-a}
    \end{BMAT}\right) =
    \left(\begin{BMAT}[5pt]{c:c}{c:c}
      \ID_{a} & \ap \\
      0 & \ID_{n-a}
    \end{BMAT}\right)
    \left(\begin{BMAT}[5pt]{c:c}{c:c}
      x\ID_{a} & 0 \\
      0 & \ID_{n-a}
    \end{BMAT}\right)
    \left(\begin{BMAT}[5pt]{c:c}{c:c}
      \ID_{a} & 0 \\
      -\am & \ID_{n-a}
    \end{BMAT}\right)
\end{equation}
with
\begin{equation}\label{eq:specific A-blocks1}
  \am=
  \left(\begin{array}{cccc}
    \oa^{}_{a+1,1} & \cdots & \oa^{}_{a+1,a} \\
    \vdots & \iddots & \vdots \\
    \oa^{}_{n,1} & \cdots & \oa^{}_{n,a} \\
  \end{array}\right) \,, \qquad
  \ap=
  \left(\begin{array}{cccc}
    \oad^{}_{1,a+1} & \cdots & \oad^{}_{1,n} \\
    \vdots & \iddots & \vdots \\
    \oad^{}_{a,a+1} & \cdots & \oad^{}_{a,n} \\
  \end{array}\right) \,,
\end{equation}
as in~\eqref{eq:ApAmA}, and the particle-hole transformation (denoted \emph{p.h.}) in~\eqref{eq:matricesx 1}
is chosen as follows:
\begin{equation}\label{eq:I-ph}
\begin{split}
  & \oad_{i,j}\mapsto \oa_{\sigma_I(j),\sigma_I(i)} \,, \quad \oa_{j,i}\mapsto -\oad_{\sigma_I(i),\sigma_I(j)}
    \quad \mathrm{if}\quad  \sigma_I(j)<\sigma_I(i) \,,  \\
  & \oad_{i,j}\mapsto \oad_{\sigma_I(i),\sigma_I(j)} \,, \quad \oa_{j,i}\mapsto \oa_{\sigma_I(j),\sigma_I(i)}
    \qquad \mathrm{if}\quad  \sigma_I(j)>\sigma_I(i) \,.
  \end{split}
\end{equation}
Let us note that the matrix $L_I(x)$ of~\eqref{eq:matricesx 1} depends only on the oscillators
$(\oad_{i,j},\oa_{j,i})\in \CA$ with $i\in I$ and $j\in \bar{I}=\{1,\ldots,n\}\setminus I$, see~\eqref{eq:set-complement},
and can be further factorised similarly to~\eqref{eq:matricesx 2} as:
\begin{equation*}
  L_I(x) =
  \left(\id_n+\sum_{i\in I}^{j\in \bar{I}} \left(\oad_{ij}\delta_{i<j}+\oa_{ji}\delta_{i>j}\right)e_{ij}\right)
  \left(x\sum_{i\in I}e_{ii}+\sum_{j\in \bar{I}}e_{jj}\right)
  \left(\id_n+\sum_{i\in I}^{j\in \bar{I}}\left(\oad_{ij}\delta_{j<i}-\oa_{ji}\delta_{j>i}\right)e_{ji}\right) .
\end{equation*}

\medskip

\begin{Rem}\label{rem:comparing two ph}
We note that the particle-hole~\eqref{eq:I-ph} differs from~\eqref{eqn:A-ph} in two aspects:
(1) a different sign change, (2) relabelling of the oscillator indices to indicate the row and column of their position.
\end{Rem}

\medskip
\noindent
Let us now apply the general factorisation from Subsection~\ref{ssec general two-factorisation} to
the following choice of~\eqref{eq:two matrices}:
\begin{equation}\label{eq:A-choice}
  L_a(x)=L_{\{1,\ldots,a\}}^{ }(x) \,, \qquad \bL_a(y)=L_{\{a+1,\ldots,n\}}(y) \,.
\end{equation}
This fixes $\am_1$ and $\ap_1$ of~\eqref{eq:two matrices} as $\am$ and $\ap$ of~\eqref{eq:specific A-blocks1},
while $\am_2$ and $\ap_2$ are explicitly given by:
\begin{equation}\label{eq:specific A-blocks2}
  \am_2=
  \left(\begin{array}{cccc}
    \oa_{1,a+1} & \cdots & \oa_{a,a+1} \\
    \vdots & \iddots & \vdots \\
    \oa_{1,n} & \cdots & \oa_{a,n} \\
  \end{array}\right) \,, \qquad
  \ap_2=
  \left(\begin{array}{cccc}
    \oad_{a+1,1} & \cdots & \oad_{n,1} \\
    \vdots & \iddots & \vdots \\
    \oad_{a+1,a} & \cdots & \oad_{n,a} \\
  \end{array}\right) \,.
\end{equation}
We note that these two matrices $L_a(x)$ and $\bL_a(y)$ involve non-intersecting sets of oscillators from
the algebra $\CA$, hence, they mutually commute, while the only nontrivial commutators are:
\begin{equation}\label{eq:ambient osc}
  [\oa_{j,i},\oad_{i,j}]=1 \,.
\end{equation}
Thus, the transformation~\eqref{eq:tra12} in this case is in fact induced by the similarity transformation:
\begin{equation}\label{eq:gauge}
\begin{split}
  & \am'_1=\bfS \am_1 \bfS^{-1} \,, \qquad \am'_2=\bfS \am_2 \bfS^{-1} \,, \\
  & \ap'_2=\bfS \ap_2 \bfS^{-1} \,, \qquad \ap'_1=\bfS \ap_1 \bfS^{-1} \,,
\end{split}
\end{equation}
with
\begin{equation}\label{eq:A-similarity}
  \bfS=\exp\left[\sum_{1\leq i\leq a}^{a<j\leq n} \oad_{ij}\oa_{ij}\right] \,,
\end{equation}
where we note that all the summands in~\eqref{eq:A-similarity} pairwise commute.

\medskip
\noindent
Combining the factorisation formula~\eqref{eq:linear factorization} with
the similarity transformation~\eqref{eq:gauge}, we obtain:
\begin{equation}\label{eq:facA}
  L_{\{1,\ldots,a\}}(x+\st) L_{\{a+1,\ldots,n\}}(x-a) = \, \bfS \CL_a(x) \Gop \bfS^{-1} \,,
\end{equation}
where
\begin{equation}
  \Gop =
    \left(\begin{BMAT}[5pt]{c:c}{c:c}
      \ID_a & \ap_2 \\
      0 & \ID_{n-a}
    \end{BMAT}\right) \,,
\end{equation}
$\bfS$ is given by~\eqref{eq:A-similarity}, and $\CL_a(x)$ is precisely
the $\gl_n$-type Lax matrix of~\eqref{eq:LaxRec}.

\medskip
\noindent
Vice versa, the matrices $L_{\{1,\ldots,a\}}(x)$ and $L_{\{a+1,\ldots,n\}}(x)$ can be obtained
from the Lax matrix $\CL_a(x)$ of~\eqref{eq:LaxRec} via the \emph{renormalized limit} procedures
(which clearly preserve the property of being Lax):
\begin{equation}\label{eq:A renormalized limit}
\begin{split}
  & L_{\{1,\ldots,a\}}(x) =
    \lim_{\st\to \infty}\ \Big\{\CL_a(x-\st)\cdot \mathrm{diag}
    \Big(\underbrace{1,\ldots,1}_{a};\underbrace{-\sfrac{1}{\st},\ldots,-\sfrac{1}{\st}}_{n-a}\Big)\Big\} \,, \\
  & L_{\{a+1,\ldots,n\}}(x) = \lim_{\st\to \infty}\ \Big\{\mathrm{diag}
    \Big(\underbrace{\sfrac{1}{\st},\ldots,\sfrac{1}{\st}}_{a};\underbrace{1,\ldots,1}_{n-a}\Big)
    \cdot \CL_a(x+a)\Big\}\Big|_{\oad_{ij}\mapsto -\oad_{ji} \,,\, \oa_{ij}\mapsto -\oa_{ji}}  \,.
\end{split}
\end{equation}

\medskip

\begin{Rem}
This implies that all the matrices $\{L_I(x)\}_{I\in \CS_a}$ of~\eqref{eq:matricesx 1} are in fact Lax, that is,
they satisfy the RTT relation~\eqref{eq:RTT}, which is crucial for the entire analysis of the present section.
\end{Rem}

\medskip
\noindent
Conjugating the factorisation formula~\eqref{eq:facA} by $B_I$ of~\eqref{eq:permutation-matrix},
thus utilizing the Weyl group action, and further performing the particle-hole transformations
in both sets of oscillators, we obtain:
\begin{equation}\label{eq:factor-A}
  L_I(x+\st) L_{\bar I}(x-a) = \, \bfS_I \CL'_I(x) \Gop_I \bfS_I^{-1}
\end{equation}
with the similarity transformation $\bfS_I$ and the matrix $\Gop_I$ specified
in~(\ref{eq:A-similarity I}) and ~\eqref{eq:G-redundancy I} below,
and the $\gl_n$-type Lax matrix $\CL'_I(x)$ obtained from $\CL_I(x)$
of~\eqref{eq:symmetry-conjugate-A} through the sign change and relabelling of the oscillator indices
precisely as in Remark~\ref{rem:comparing two ph}. The factorisation~\eqref{eq:factor-A} above
follows when performing the particle-hole transformation~\eqref{eq:I-ph} for the oscillators contained
in $L_{\{1,\ldots,a\}}(x+\st)$, see~\eqref{eq:matricesx 1}, and further undoing
the particle-hole transformation  for the \emph{creation}/\emph{annihilation} oscillators contained in
$L_{\{a+1,\ldots,n\}}(x-a)$ that get mapped below/above the main diagonal:
\begin{equation}
\begin{split}
  & \oad_{j,i}\mapsto -\oa_{\sigma_I(i),\sigma_I(j)} \,, \quad
    \oa_{i,j}\mapsto \oad_{\sigma_I(j),\sigma_I(i)} \quad
    \mathrm{if}\quad  \sigma_I(j)<\sigma_I(i) \,, \\
  & \oad_{j,i}\mapsto \oad_{\sigma_I(j),\sigma_I(i)} \,, \quad
    \oa_{i,j}\mapsto \oa_{\sigma_I(i),\sigma_I(j)} \qquad
    \mathrm{if}\quad  \sigma_I(j)>\sigma_I(i) \,,
\end{split}
\end{equation}
with the permutation $\sigma_I$ of the set $\{1,\ldots,n\}$ as in~\eqref{eq:special permutation}.
More precisely, the latter particle-hole (denoted ${\overline{p.h.}}$) is chosen so that the following equality holds:
\begin{equation}
  L_{\bar{I}}(x) = \left. B_I L_{\{a+1,\ldots,n\}}(x)B_I^{-1} \right|_{\overline{p.h.}} \,,
\end{equation}
where we note the following natural compatibility between $B_I$ and $B_{\bar{I}}$:
\begin{equation}
  B_I B_{\{a+1,\ldots,n\}} = B_{\bar I} \,, \qquad \forall\, I\in \CS_a \,.
\end{equation}
The remaining ingredients in~\eqref{eq:factor-A} can be written as:
\begin{equation}\label{eq:G-redundancy I}
  \Gop_I=\left. B_I\Gop^{}B_I^{-1}\right|_{\overline{p.h.}}=
  \id_n+\sum_{i\in I}^{j\in \bar{I}} \Big(\oad_{ji}\delta_{i<j}-\oa_{ij}\delta_{i>j}\Big) e_{ij} \,,
\end{equation}
and $\bfS_I$ is obtained from $\bfS$ of~\eqref{eq:A-similarity} by performing both particle-hole
transformations \emph{p.h.}, $\overline{p.h.}$:
\begin{equation}\label{eq:A-similarity I}
   \bfS_I =
   \exp\left[ \sum_{i\in I}^{j\in \bar{I}} \Big(\oad_{ij}\oa_{ij}\delta_{i<j} - \oa_{ji}\oad_{ji}\delta_{i>j}\Big) \right] \,.
\end{equation}

\medskip
\noindent
Following~\cite{Bazhanov:2010jq}, let us now define the $Q$-operators
$\{Q_I(x)\}_{I\in \CS_a}\subset \End(\BC^n)^{\otimes N}$ via:
\begin{equation}\label{eq:qopsA}
  Q_I(x) = \wtr_{D_I}\, \Big( \underbrace{L_I(x)\otimes \dots \otimes L_I(x)}_{N} \Big) \,,
\end{equation}
that is, as the \emph{normalized trace} $\wtr_{D_I}$, defined as in~\eqref{eq:normalized trace},
of the $N$-fold tensor product of $L_I(x)$ from~\eqref{eq:matricesx 1}.
The twist $D_I$ in~\eqref{eq:qopsA} acts only on the Fock space and is defined via:
\begin{equation}\label{eq:twist I}
  D_I =
  \prod_{i\in I,\, j\in\bar I}^{i<j} \left(\frac{\tau_j}{\tau_i}\right)^{\oad_{ij}\oa_{ji}} \,
  \prod_{i\in I,\, j\in\bar I}^{i>j} \left(\frac{\tau_i}{\tau_j}\right)^{\oa_{ji}\oad_{ij}} \,,
\end{equation}
which can be further expressed via $\{\CE^I_{ii}\}_{i=1}^n$ of~\eqref{eq:symmetry-conjugate-A} as
(cf.\ formulas~(\ref{eq:diagonal-A},~\ref{eq:Cartan I})):
\begin{equation}\label{eq:twist-via-Cartan I}
  D_I = \prod_{i\in I} \tau_i^{-\st}  \cdot \, \prod_{i=1}^n \tau_i^{\CE_{ii}^I} \,.
\end{equation}
We note that the action of this twist on the Fock module is uniquely determined (up to a scalar function)
by the same condition as in~\eqref{eq:DiD agrees}:
\begin{equation}\label{eq:twistconj}
  DL_I(x)D^{-1} = D_I^{-1} L_I(x)D_I \,,
\end{equation}
with $D=\diag(\tau_1,\ldots,\tau_n)$ of~\eqref{eq:D-diag}. This ensures the commutativity of $Q_I(x)$ and the
transfer matrix $T_{(1,0,\ldots,0)}(y)$ of the defining fundamental representation in the auxiliary space,
cf.~Remark~\ref{rem:commutativity}.

\medskip

\begin{Rem}\label{rem:generalize partonic}
For $a=1$ and $I=\{i\}$, the Lax matrix~\eqref{eq:matricesx 1} coincides with $L_i(x)$ of~\eqref{eq:partonic-A} and
the twist~\eqref{eq:twist I} coincides with~\eqref{eq:twist-bflms}, hence, the above $Q$-operator~\eqref{eq:qopsA}
recovers $Q_i(x)$ of~\eqref{eq:singleindexQ}.
\end{Rem}

\medskip
\noindent
Building the monodromy matrices (by considering the $N$-fold tensor product on the matrix space)
from~\eqref{eq:factor-A}, taking the normalized trace, and evoking the relation~\eqref{eq:twist-via-Cartan I},
the factorisation formula~\eqref{eq:factor-A} implies the following factorisation formula for
the transfer matrices $T^+_{I,\st}(x)$ of~\eqref{eq:Tpgln}:
\begin{equation}\label{eq:T-via-QQ A-type}
  T_{I,\st}^+(x)=\ch_{I,\st}^+ \cdot\, Q_I(x+\st) Q_{ \bar I}(x-a)
\end{equation}
with
\begin{equation}\label{eq:ch-I}
  \ch^+_{I,\st} = \tr\, \prod_{i=1}^n \tau_i^{\CE^{I}_{ii}} =
  \prod_{i\in I} \tau_{i}^{\st} \prod_{i\in I}^{j\in \bar{I}}\, \frac{(-1)^{\delta_{i>j}}\tau_i}{\tau_i-\tau_j} \,,
\end{equation}
cf.\ the factorisation formula~\eqref{eq:details-fact}, the character formula~\eqref{eq:char-A},
and the details of Remark~\ref{rem:details for A-factor}.

\medskip
\begin{Rem}
Let us stress right away that the transfer matrices constructed from $\CL'_I(x)$ and $\CL_I(x)$
via~\eqref{eq:Tpgln} \underline{do coincide}, as the sign change and the relabelling of oscillators
(see Remark~\ref{rem:comparing two ph}) do not affect the trace.
\end{Rem}

\medskip

\begin{Rem}\label{rem:SDcomRect}
An essential step used in our derivation of~\eqref{eq:T-via-QQ A-type} is the following commutativity:
\begin{equation}\label{eq:comSDrect}
  [\bfS_I, D_{I}D_{\bar{I}}] = 0 \,,
\end{equation}
cf.~\eqref{eq:comSD}. Clearly, it suffices to verify~\eqref{eq:comSDrect} for $I=\{1,\ldots,a\}$.
To this end, we note that:
\begin{equation}
  D_{\{1,\ldots,a\}} D_{\{a+1,\ldots,n\}} \, =
  \prod_{1\leq i\leq a}^{a<j\leq n} \left(\frac{\tau_j}{\tau_i}\right)^{\oad_{ij}\oa_{ji} + \oa_{ij}\oad_{ji}} = \
  \prod_{1\leq i\leq a} \tau_i^{\bN_i} \prod_{a<j\leq n} \tau_j^{\bN_j}
\end{equation}
with
\begin{equation}
  \bN_i= -\sum_{a<j\leq n} \left( \oad_{ij}\oa_{ji} + \oa_{ij}\oad_{ji} \right) \,,\qquad
  \bN_j=  \sum_{1\leq i\leq a} \left( \oad_{ij}\oa_{ji} + \oa_{ij}\oad_{ji} \right) \,,
\end{equation}
while the similarity transformation $\bfS_I$ is given by~\eqref{eq:A-similarity}:
\begin{equation}
  \bfS_{\{1,\ldots,a\}} = \prod_{1\leq i\leq a}^{a<j\leq n} \exp\left[ \oad_{ij}\oa_{ij} \right] \,.
\end{equation}
Thus, the desired commutativity~\eqref{eq:comSDrect} follows from the obvious equalities
\begin{equation}
  [\bN_i,\oad_{k \ell}\oa_{\ell k}] = 0 \,, \qquad
  [\bN_j,\oad_{k \ell}\oa_{\ell k}] = 0 \,,
\end{equation}
for any $1\leq i\leq a<j\leq n$ and $1\leq k\leq a<\ell\leq n$, cf.~(\ref{eq:N-commutativity},~\ref{eq:N-action}).
\end{Rem}

\medskip
\noindent
Combining the factorisation formula~\eqref{eq:T-via-QQ A-type} with Theorem~\ref{thm:Main-A},
we get (see~\eqref{eq:sigma-length} and~\eqref{eq:ch-I}):

\medskip

\begin{Prop}\label{prop:Main-A-recasted}
For any $1\leq a<n$ and $\st\in \BN$, we have:
\begin{equation}\label{eq:transfer-AQQ}
  T_{a,\st}(x)\, = \sum_{I\in \CS_a} (-1)^{l(I)} \, \ch_{I,\st}^+ \cdot\, Q_I(x+\st) Q_{\bar I}(x-a) \,.
\end{equation}
\end{Prop}

\medskip

\begin{Rem}
(a) Such formula first appeared in~\cite[(5.12)]{Bazhanov:2001xm} for the $n=3$ trigonometric case,
while the general rational case goes back to~\cite{Tsuboi:2009ud,Tsuboi:2011iz}. However, our derivation
of~\eqref{eq:transfer-AQQ} from~\eqref{eq:transfer-A} has a benefit of not using the determinant
formula~\eqref{eq:Tdet} that is absent in other types.

\medskip
\noindent
(b) Let us note that~\eqref{eq:transfer-AQQ} is also a consequence of the determinant formula~\eqref{eq:Tdet}
and its analogue expressing any $Q$-operator $Q_I(x)$ of~\eqref{eq:qopsA} in terms of the single-indexed
$Q$-operators~\eqref{eq:singleindexQ}:
\begin{equation}\label{eq:QI}
  Q_I(x)
  =\frac{\det \left\Vert \tau_{i_k}^{-\ell+1} Q_{i_k}(x-\ell+1) \right\Vert_{1\leq k,\ell\leq a}}
        {\det \left\Vert \tau_{i_k}^{-\ell+1} \right\Vert_{1\leq k,\ell\leq a}} \,,
\end{equation}
where $I=\{i_1,\ldots,i_a\}$ (the right-hand side of~\eqref{eq:QI} is clearly independent of the ordering of $i_c$'s).
To establish this formula, one has to consider a family of $\gl_n$-type Lax matrices
$\{\mathbb{L}_I(x)\}_{I\subseteq \{1,\ldots,n\}}$, see~\cite[(2.20)]{Bazhanov:2010jq},
generalizing $L_I(x)$ of \eqref{eq:matricesx 1} by letting their matrix coefficients to take values
in the bigger algebra $\CA\otimes U(\gl_{|I|})[x]$. Explicitly, we set
$\mathbb{L}_{\{1,\ldots,a\}}(x)=L_{\{1,\ldots,a\}}(x)+\sum_{i,j=1}^a e_{ij}E_{ji}$
with $\{E_{ji}\}_{i,j=1}^a$ being the generators of $\gl_a$, while all other $\mathbb{L}_I(x)$ are again obtained
through the similarity and particle-hole transformations precisely as in~\eqref{eq:matricesx 1}, thus resulting in:
\begin{equation}\label{eq:generalized-L}
  \mathbb{L}_I(x)=L_I(x)+\sum_{i,j\in I} E_{\sigma_I^{-1}(j),\sigma_I^{-1}(i)}e_{ij} \,.
\end{equation}
Generalizing $Q_I(x)$ of~\eqref{eq:qopsA}, one defines the $X$-operators
$\{X^+_I(x,\lambda)\}_{I\in \CS_a}^{\lambda\in \BC^{a}}\subset \End(\BC^n)^{\otimes N}$ via:
\begin{equation}\label{eq:X-operators}
  X^+_I(x,\lambda) = \tr_{M^\vee_\lambda} \left\{\prod_{i\in I}\tau_i^{E_{ii}}\,\wtr_{D_I}\,
  \Big( \underbrace{\mathbb{L}_I(x)\otimes \dots \otimes \mathbb{L}_I(x)}_{N} \Big)\right\} \,,
\end{equation}
cf.~\cite[(4.13)]{Bazhanov:2010jq}. Thus, $X^+_I(x,\lambda)$ is the \emph{normalized trace} $\wtr_{D_I}$
of~\eqref{eq:normalized trace} in the Fock module $\Fock$ of $\CA$
followed by the standard trace in the dual Verma module $M^\vee_\lambda$ of $\gl_a$
of the $N$-fold tensor product of $\mathbb{L}_I(x)$ with the twist $D_I$ defined in~\eqref{eq:twist I}.
Likewise, for a dominant integral weight $\lambda$ of $\gl_a$, one defines $X_{I}(x,\lambda)$ with the outer trace
taken over the finite-dimensional $\gl_a$-submodule $L_\lambda$ of $M^\vee_\lambda$.
The latter construction allows to recover back the $Q$-operators via:
\begin{equation}\label{eq:Q-via-X}
  Q_I(x)=X_I\Big(x,(\underbrace{0,\ldots,0}_{a})\Big) \,.
\end{equation}
Then, evoking the BGG resolution of the finite-dimensional $\gl_a$-module $L_\lambda$, we obtain
the following counterpart of Theorem~\ref{thm:baby-bgg} (cf.~\cite[(4.19)]{Bazhanov:2010jq}):
\begin{equation}\label{eq:X-transfer}
  X_I(x,\lambda)=\sum_{\sigma\in S_a} (-1)^{l(\sigma)} X^+_I(x,\sigma \, \cdot \lambda) \,.
\end{equation}
On the other hand, arguing precisely as in our proof of~\eqref{eq:details-fact}, we obtain
the following counterpart of the latter (cf.~\cite[(5.7)]{Bazhanov:2010jq}):
\begin{equation}\label{eq:X-factorisation}
  X^+_I(x,\lambda)=
  \prod_{1\leq k<\ell\leq a}\frac{1}{\tau_{i_\ell}^{-1}-\tau_{i_k}^{-1}} \,
  \prod_{1\leq k\leq a} \tau_{i_k}^{\lambda_k-k+1}
  Q_{i_1}(x+\lambda_1) Q_{i_2}(x+\lambda_2-1) \cdots Q_{i_{a}}(x+\lambda_a-a+1) \,,
\end{equation}
where $I=\{i_1,\ldots,i_a\}\in \CS_a$ and $\lambda=(\lambda_1,\ldots,\lambda_a)\in \BC^a$.
Combining~\eqref{eq:X-factorisation} with~\eqref{eq:X-transfer} and evoking the Vandermonde determinant,
we obtain the following analogue of Theorem~\ref{thm:baby-det}:
\begin{equation}\label{eq:Xdet}
  X_I(x,\lambda) =
  \frac{\det \left\Vert \tau_{i_k}^{\lambda_\ell-\ell+1}
        Q_{i_k}(x+\lambda_\ell-\ell+1) \right\Vert_{1\leq k,\ell\leq a}}
       {\det \left\Vert \tau_{i_k}^{-\ell+1} \right\Vert_{1\leq k,\ell\leq a}} \,.
\end{equation}
Specializing this formula at $\lambda=(0,\ldots,0)$ and evoking~\eqref{eq:Q-via-X}, we recover the desired
formula~\eqref{eq:QI}.

\medskip
\noindent
(c) Conversely, plugging the formula~\eqref{eq:QI} into~\eqref{eq:transfer-AQQ}, we recover
(by expanding the corresponding $n\times n$ determinant with respect to the first $a$ columns)
the determinant formula~\eqref{eq:Tdet} in the particular case of $\lambda=\st\omega_a$, the multiples
of fundamental weights.
\end{Rem}

\medskip
\noindent
For completeness of our exposition, let us conclude with the $QQ$-relations in the present conventions:

\medskip

\begin{Lem}\label{lem:qq}
For any two disjoint subsets $I$ and $\{i,j\}$ of $\{1,\ldots,n\}$, we have:
\begin{equation*}
  Q_{I\sqcup i\sqcup j}(x+\tfrac{1}{2}) Q_I(x-\tfrac{1}{2}) =
  \frac{\tau_j}{\tau_j-\tau_i} Q_{I\sqcup i}(x-\tfrac{1}{2}) Q_{I\sqcup j}(x+\tfrac{1}{2}) -
  \frac{\tau_i}{\tau_j-\tau_i} Q_{I\sqcup j}(x-\tfrac{1}{2}) Q_{I\sqcup i}(x+\tfrac{1}{2}) \,.
\end{equation*}
\end{Lem}

\medskip

\begin{proof}
Let $I=\{i_1,\ldots,i_a\}$ and set $i_0=j,\, i_{a+1}=i$.
Then, the $QQ$-relation stated above follows immediately from the
Desnanot-Jacobi-Dodgson-Sylvester theorem applied to the $(a+2)\times(a+2)$ matrix
$M=\left(\tau_{i_k}^{-\ell+1} Q_{i_k}(x-\ell+\tfrac{3}{2})\right)_{0\leq k\leq a+1}^{1\leq \ell\leq a+2}$
with $Q_i(x)$ defined in~\eqref{eq:singleindexQ}.
\end{proof}

\medskip

\begin{Rem}
Generalizing our earlier Remarks~\ref{rem:our-vs-bflms first} and~\ref{rem:generalize partonic}, we note that the
$Q$-operators $\mathsf{Q}_I(x)$ of~\cite[(4.13,~4.20)]{Bazhanov:2010jq} are related to ours from~\eqref{eq:qopsA} via:
\begin{equation}\label{eq:qopsAtobflms}
  \mathsf{Q}_I(x) = \prod_{i\in I} \tau_i^x \cdot \, Q_I\left(x-\sfrac{n-|I|}{2}\right) \,,
\end{equation}
where the twist parameters and the oscillators are identified via~\eqref{eq:bflms-vs-ous twists}
and~\eqref{eq:v-vs-a osc}, respectively.
Here, the shift of the spectral parameter $x$ by $\sfrac{n-|I|}{2}$ arises when identifying
our Lax matrices~\eqref{eq:matricesx 2} with those of~\cite[(2.20)]{Bazhanov:2010jq}, while the additional
factor $\prod_{i\in I} \tau_i^x$ is due to the conventions of~\cite{Bazhanov:2010jq}.
Thus, our $QQ$-relations of Lemma~\ref{lem:qq} are equivalent to the $QQ$-relations of~\cite[(5.12)]{Bazhanov:2010jq}
(though the latter need to be corrected by a sign as already seen from~\cite[(5.9)]{Bazhanov:2010jq}),
see~\eqref{eq:A-Hirota}:
\begin{equation}\label{eq:qq-bflms}
  \frac{\tau_j-\tau_i}{\sqrt{\tau_i\tau_j}} \cdot \mathsf{Q}_{I\sqcup i\sqcup j}(x) \mathsf{Q}_I(x) =
  \mathsf{Q}_{I\sqcup i}(x-\tfrac{1}{2}) \mathsf{Q}_{I\sqcup j}(x+\tfrac{1}{2}) -
  \mathsf{Q}_{I\sqcup j}(x-\tfrac{1}{2}) \mathsf{Q}_{I\sqcup i}(x+\tfrac{1}{2}) \,.
\end{equation}
\end{Rem}

\medskip

\begin{Rem}
Let us also note that the commutativity of the single-index $Q$-operators $\{Q_i(x)\}_{i=1}^n$,
see Remark~\ref{rem:commutativity}(c), is essential both to the derivation of the determinant
formulas~(\ref{eq:Tdet},~\ref{eq:QI}) as well as to the above proof of Lemma~\ref{lem:qq}.
Furthermore, combining Remark~\ref{rem:commutativity} with the determinant expression~\eqref{eq:QI},
we conclude that all the $Q$-operators $\{Q_I(x)|I\subseteq \{1,\ldots,n\}\}$ commute among themselves as well
as with the transfer matrices $T^+_{I,\st}(x)$ and $T_{a,\st}(x)$ of Subsection~\ref{ssec A-transfer rect}.
\end{Rem}

\medskip


\subsection{Two-term factorisation in C-type}\label{sec:factorisation in C-type}
$\ $

Inspired by~\eqref{eq:oscLaxC}, let us consider the particular example of the general
factorisation~\eqref{eq:linear factorization} applied in the case when $n=2r$, $a=r$, and
the $r\times r$ matrices $\am_1,\ap_1,\am_2,\ap_2$ are explicitly given by:
\begin{equation}\label{eq:specific C-blocks}
\begin{split}
  &
  \am_1=
  \left(\begin{array}{cccc}
    \oa_{r',1} & \cdots & \oa_{r',r-1} & \oa_{r',r} \\
    \vdots & \iddots & \oa_{(r-1)',r-1} & \oa_{r',r-1} \\
    \oa_{2',1} & \oa_{2',2} & \iddots & \vdots \\
    \oa_{1',1} & \oa_{2',1} & \cdots & \oa_{r',1}
  \end{array}\right) \,, \quad
  \ap_1=
  \left(\begin{array}{cccc}
    \oad_{1,r'} & \cdots & \oad_{1,2'} & 2\oad_{1,1'} \\
    \vdots & \iddots & 2\oad_{2,2'} & \oad_{1,2'} \\
    \oad_{r-1,r'} & 2\oad_{r-1,(r-1)'} & \iddots & \vdots \\
    2\oad_{r,r'} & \oad_{r-1,r'} & \cdots & \oad_{1,r'}
  \end{array}\right) \,, \\
  &
  \am_2=
  \left(\begin{array}{cccc}
    \oa_{1,r'} & \cdots & \oa_{r-1,r'} & 2\oa_{r,r'} \\
    \vdots & \iddots & 2\oa_{r-1,(r-1)'} & \oa_{r-1,r'} \\
    \oa_{1,2'} & 2\oa_{2,2'} & \iddots & \vdots \\
    2\oa_{1,1'} & \oa_{1,2'} & \cdots & \oa_{1,r'}
  \end{array}\right) \,, \quad
  \ap_2=
  \left(\begin{array}{cccc}
    \oad_{r',1} & \cdots & \oad_{2',1} & \oad_{1',1} \\
    \vdots & \iddots & \oad_{2',2} & \oad_{2',1} \\
    \oad_{r',r-1} & \oad_{(r-1)',r-1} & \iddots & \vdots \\
    \oad_{r',r} & \oad_{r',r-1} & \cdots & \oad_{r',1}
  \end{array}\right) \,,
\end{split}
\end{equation}
cf.~\eqref{eq:ApAmC}, where the only nontrivial commutators of the above entries are:
\begin{equation}\label{eq:ambient osc C}
  \big[\oa_{j,i},\oad_{i,j}\big]=1\,,\qquad 1\leq i,j\leq 2r \,.
\end{equation}

\medskip
\noindent
In this setup, both matrices $L_{r}(x)$ and $\bL_{r}(y)$ of~\eqref{eq:two matrices} are actually $C_r$-type Lax matrices.
In fact
\begin{equation}\label{eq:laxlin}
  L_{(\footnotesize{\underbrace{+,\ldots,+}_{r}})}(x) = L_{r}(x) =
    \left(\begin{BMAT}[5pt]{c:c}{c:c}
      x\ID_r-\ap_1\am_1 & \ap_1 \\
      -\am_1 & \ID_r
    \end{BMAT}\right)
\end{equation}
appeared in our recent work~\cite[(3.50)]{Frassek:2021ogy}. On the other hand, the Lax matrix $\bL_{r}(y)$ can be obtained
from~\eqref{eq:laxlin} via the Weyl group action followed by a particle-hole transformation:
\begin{equation}\label{eq:opposite deg-Lax-C}
  L_{(\footnotesize{\underbrace{-,\ldots,-}_{r}})}(y)=\bL_{r}(y)=\left. S L_{(+,\ldots,+)}(y)S^{-1} \right|_{P.H.}=
    \left.\left(\begin{BMAT}[5pt]{c:c}{c:c}
      \ID_r & \idb_r\am_1\idb_r \\
      -\idb_r \ap_1 \idb_r & y\ID_r-\idb_r \ap_1\am_1\idb_r
   \end{BMAT}\right)\right|_{P.H.}
\end{equation}
with the $2r\times 2r$ similarity matrix $S$ given by:
\begin{equation}\label{eq:Spm}
   S=
   \left(\begin{BMAT}[5pt]{c:c}{c:c}
     0 & \idb_r \\
     -\idb_r & 0
   \end{BMAT}\right)=B_{(-1,\ldots,-1)} \,,
\end{equation}
cf.\ notation~\eqref{eq:C-type weyl elt}, and the \emph{total particle-hole} transformation (denoted \emph{P.H.}) given by:
\begin{equation}\label{eq:total-ph}
  \oad_{i,j'}\mapsto -\oa_{i,j'} \,, \quad \oa_{j',i}\mapsto \oad_{j',i}
  \, \quad \text{for\ all} \quad 1\leq i\leq j\leq r \,.
\end{equation}

\medskip
\noindent
Let us stress right away that both $L_{(+,\ldots,+)}(x)$ and $L_{(-,\ldots,-)}(x)$ can be obtained
from the Lax matrix $\CL(x)$ of~\eqref{eq:oscLaxC} via the \emph{renormalized limit} procedures
(preserving the property of being Lax):
\begin{equation}\label{eq:C renormalized limit}
\begin{split}
  & L_{(+,\ldots,+)}(x) = \lim_{\st\to \infty}\
    \Big\{\CL(x-\st)\cdot \mathrm{diag}
    \Big(\underbrace{1,\ldots,1}_{r};\underbrace{-\sfrac{1}{2\st},\ldots,-\sfrac{1}{2\st}}_{r}\Big)\Big\} \,,\\
  & L_{(-,\ldots,-)}(x) = \lim_{\st\to \infty}\ \Big\{\mathrm{diag}
    \Big(\underbrace{\sfrac{1}{2\st},\ldots,\sfrac{1}{2\st}}_{r};\underbrace{1,\ldots,1}_{r}\Big)\cdot
    \CL(x+\st+r+1)\Big\} \Big|_{\ap\mapsto -\ap_2 \,,\, \am\mapsto -\am_2} \,.
\end{split}
\end{equation}

\medskip

\begin{Rem}\label{rem:trace-indep}
We note that the commutation relations \eqref{eq:ambient osc C} are invariant under the transformation
$\oa_{j,i}\mapsto -\oa_{i,j},\ \oad_{i,j}\mapsto -\oad_{j,i}$, and furthermore such transformations
do not affect the trace.
\end{Rem}

\medskip
\noindent
Then, the transformation~\eqref{eq:tra12} is again induced by the similarity transformation~\eqref{eq:gauge} with
\begin{equation}\label{eq:C-similarity}
  \bfS=\exp\left[\sum_{1\leq i\leq j\leq r} \left(1+\delta_i^j\right) \oad_{ij'}\oa_{ij'}\right] \,,
\end{equation}
where we note that all the summands in~\eqref{eq:C-similarity} pairwise commute.

\medskip
\noindent
Combining the factorisation formula~\eqref{eq:linear factorization} with the similarity transformation~\eqref{eq:gauge},
we obtain:
\begin{equation}\label{eq:factor-C}
  L_{(+,\ldots,+)}(x+\st) L_{(-,\ldots,-)}(x-\st-r-1) = \, \bfS \CL (x) \Gop \bfS^{-1} \,,
\end{equation}
where
\begin{equation}
  \Gop  =
  \left(\begin{BMAT}[5pt]{c:c}{c:c}
    \ID_r & \ap_2 \\
    0 & \ID_{r}
  \end{BMAT}\right) \,,
\end{equation}
$\bfS$ is given by~\eqref{eq:C-similarity}, and $\CL(x)$ is precisely the $C_r$-type Lax matrix of~\eqref{eq:oscLaxC}.

\medskip
\noindent
Following~\eqref{eq:qopsA}, let us now define the $Q$-operators
$Q_{(+,\ldots,+)}(x), Q_{(-,\ldots,-)}(x)\in \End(\BC^{2r})^{\otimes N}$ via:
\begin{equation}\label{eq:QC+}
  Q_{(+,\ldots,+)}(x)=\wtr_{D_{(+,\ldots,+)}}
  \Big(\underbrace{L_{(+,\ldots,+)}(x) \otimes \cdots \otimes L_{(+,\ldots,+)}(x)}_{N}\Big)
\end{equation}
and
\begin{equation}\label{eq:QC-}
  Q_{(-,\ldots,-)}(x)=\wtr_{D_{(-,\ldots,-)}}
  \Big(\underbrace{L_{(-,\ldots,-)}(x) \otimes \cdots \otimes L_{(-,\ldots,-)}(x)}_{N}\Big) \,,
\end{equation}
cf.~\eqref{eq:normalized trace}, with the twists $D_{(+,\ldots,+)}$ and $D_{(-,\ldots,-)}$ defined in analogy
with~(\ref{eq:twist I})--(\ref{eq:twistconj}) via:
\begin{equation}\label{eq:two C twists}
  D_{(+,\ldots,+)}\, = \prod_{1\leq i\leq j\leq r}\left(\tau_i\tau_{j}\right)^{-\oad_{ij'}\oa_{j'i}} \,, \qquad
  D_{(-,\ldots,-)}\, = \prod_{1\leq i\leq j\leq r}\left(\tau_i\tau_{j}\right)^{-\oad_{j'i}\oa_{ij'}} \,.
\end{equation}
We note that the twist $D_{(+,\ldots,+)}$ can be further expressed via $\{\CF_{ii}\}_{i=1}^r$
of~(\ref{eq:F-generators C-type},~\ref{eq:C-diag}) as:
\begin{equation}\label{eq:twist-via-Cartan C}
  D_{(+,\ldots,+)}=\prod_{i=1}^r \tau_i^{\CF_{ii}-\st} \,.
\end{equation}
Let us stress right away that the actions of the twists $D_{(+,\ldots,+)}$ and $D_{(-,\ldots,-)}$
on the corresponding Fock modules are uniquely determined (up to scalar functions) by the following conditions:
\begin{equation}\label{eq:twistconj CD}
\begin{split}
  & D L_{(+,\ldots,+)}(x) D^{-1} = D^{-1}_{(+,\ldots,+)}L_{(+,\ldots,+)}(x)D_{(+,\ldots,+)}  \,, \\
  & D L_{(-,\ldots,-)}(x) D^{-1} = D^{-1}_{(-,\ldots,-)}L_{(-,\ldots,-)}(x)D_{(-,\ldots,-)} \,,
\end{split}
\end{equation}
with
\begin{equation}\label{eq:const twist CD}
  D=\diag\Big(\tau_1,\ldots,\tau_r,\tau^{-1}_r,\ldots,\tau^{-1}_1\Big) \,,
\end{equation}
cf.~\eqref{eq:twistconj}. The relations~\eqref{eq:twistconj CD} ensure the commutativity of
$Q_{(+,\ldots,+)}(x)$ and $Q_{(-,\ldots,-)}(x)$ defined via~(\ref{eq:QC+},~\ref{eq:QC-}) with the transfer matrix
$T_{(1,0,\ldots,0)}(y)$ of the defining fundamental representation in the auxiliary space, see Remark~\ref{rem:commutativity}(b).

\medskip
\noindent
Building the monodromy matrices (by considering the $N$-fold tensor product on the matrix space)
from~\eqref{eq:factor-C}, taking the normalized trace, and evoking the relation~\eqref{eq:twist-via-Cartan C},
the factorisation formula~\eqref{eq:factor-C} implies the following factorisation for
the transfer matrices $T^+_{(+1,\ldots,+1),\st}(x)$ of~\eqref{eq:C-inf-transfer}:
\begin{equation}\label{eq:T-via-QQ C-type basic}
  T^+_{(+1,\ldots,+1),\st}(x) = \ch_{(+1,\ldots,+1),\st}^+ \cdot\,
  Q_{(+,\ldots,+)}(x+\st) Q_{(-,\ldots,-)}(x-\st-r-1)
\end{equation}
with
\begin{equation}\label{eq:ch-C basic}
  \ch_{(+1,\ldots,+1),\st}^+ = \tr \prod_{i=1}^r \tau_i^{\CF_{ii}} = \,
  \prod_{i=1}^r \tau_i^{\st}  \prod_{1\leq i\leq j\leq r} \frac{1}{1-\tau^{-1}_i\tau^{-1}_{j}} \,.
\end{equation}

\medskip
\noindent
In analogy with Subsection~\ref{ssec two-term A-factorization}, the factorisation formula \eqref{eq:linear factorization}
can be conjugated by $B_{\vec\mu}$ of \eqref{eq:C-type weyl elt} to provide an analogue of \eqref{eq:factor-A} with
the index $I\in\CS_a$ being replaced by $\vec\mu\in \{\pm 1\}^r$.
To keep our presentation short, we shall generate the remaining $Q$-operators directly from
$Q_{(+,\ldots,+)}(x)$ of~\eqref{eq:QC+} using the action of the Weyl group via:
\begin{equation}\label{eq:Qmu}
  Q_{\vec\mu}(x)=
  \left(B_{\vec\mu} \otimes \cdots \otimes B_{\vec\mu}\right) Q_{(+,\ldots,+)}(x)
  \left(B_{\vec\mu} \otimes \cdots \otimes B_{\vec\mu}\right)^{-1}
  \Big|_{\left\{\tau_i\mapsto \tau_i^{-1}\,|\,\mu_i=-1\right\}} \,.
\end{equation}

\medskip
\noindent
To this end, we should stress right away that the action of the Weyl group on the $Q$-operator
$Q_{(-,\ldots,-)}(x)$ of~\eqref{eq:QC-} generates the same $Q$-operators $Q_{\vec\mu}(x)$.
More precisely, we have:
\begin{equation}\label{eq:Qbmu}
  Q_{\vec{\bar\mu}}(x)=
  \left(B_{\vec\mu} \otimes \cdots \otimes B_{\vec\mu}\right) Q_{(-,\ldots,-)}(x)
  \left(B_{\vec\mu} \otimes \cdots \otimes B_{\vec\mu}\right)^{-1}
  \Big|_{\left\{\tau_i\mapsto \tau_i^{-1} \,|\, \mu_i=-1\right\}} \,,
\end{equation}
where we use the notation of~\eqref{eq:opposite-mu}:
\begin{equation}\label{eq:opposite-mu revisited}
  \vec{\bar\mu} = -\vec\mu \,.
\end{equation}

\medskip
\noindent
The above claim, that is the operators~\eqref{eq:Qmu} satisfy~\eqref{eq:Qbmu}, follows from the equality:
\begin{equation}\label{eq:relLbL}
  B_{\vec\mu} {L_{(-,\ldots,-)}(x)}\Big|_{\overline{P.H.}} B_{\vec\mu}^{-1} =
  B_{\vec{\bar\mu}} L_{(+,\ldots,+)}(x)  B_{\vec{\bar\mu}}^{-1}
  \Big|_{\substack{\oa_{j'i}\mapsto \mu_i\mu_j\oa_{j'i} \\ \oad_{ij'}\mapsto \mu_i\mu_j\oad_{ij'}}}
\end{equation}
that relates the Lax matrices from our definition of $Q_{(\pm,\ldots,\pm)}(x)$.
Here, $\overline{P.H.}$ stands for undoing the total particle-hole \eqref{eq:total-ph}.
As mentioned in Remark~\ref{rem:trace-indep}, since only powers of $\oad\oa$ contribute to the trace,
the resulting $Q$-operators are invariant under the transformations
$\oa\rightarrow-\oa, \oad\rightarrow-\oad$, and thus~\eqref{eq:relLbL} indeed implies~\eqref{eq:Qbmu}.
According to~\eqref{eq:opposite deg-Lax-C}, the equality \eqref{eq:relLbL} is equivalent to:
\begin{equation}\label{eq:C-comp tech}
  \Big(B_{\vec\mu}B_{(-1,\ldots,-1)}\Big) L_{(+,\ldots,+)}(x) \Big(B_{\vec\mu}B_{(-1,\ldots,-1)}\Big)^{-1} =
  B_{\vec{\bar\mu}} L_{(+,\ldots,+)}(x) B_{\vec{\bar\mu}}^{-1}
  \Big|_{\substack{\oa_{j'i}\mapsto \mu_i\mu_j\oa_{j'i} \\ \oad_{ij'}\mapsto \mu_i\mu_j\oad_{ij'}}} \,.
\end{equation}
To prove the latter (and thus establish~\eqref{eq:Qbmu}), we note that the endomorphisms~\eqref{eq:C-type weyl elt} satisfy:
\begin{equation}\label{eq:techn C}
  B_{(-1,\ldots,-1)}^{-1}B_{\vec\mu}^{-1}B_{\vec{\bar\mu}}\ =
  \sum_{i=1}^{r} \mu_i \left(e_{ii}+e_{i'i'}\right) \,,
\end{equation}
cf.~Remark~\ref{rem:B-action C-type}. Thus, conjugating $L_{(+,\ldots,+)}(x)$ by
$B_{(-1,\ldots,-1)}^{-1}B_{\vec\mu}^{-1}B_{\vec{\bar\mu}}$ (which is diagonal in the standard basis) and
further applying the above change of oscillators leaves $L_{(+,\ldots,+)}(x)$ invariant:
\begin{equation}
  L_{(+,\ldots,+)}(x)=
  \left.\left(\sum_{i=1}^{r} \mu_i\left(e_{ii}+e_{i'i'}\right)\right)
  L_{(+,\ldots,+)}(x) \left(\sum_{i=1}^{r} \mu_i\left(e_{ii}+e_{i'i'}\right)\right)^{-1}
  \right|_{\substack{\oa_{j'i}\mapsto \mu_i\mu_j\oa_{j'i} \\ \oad_{ij'}\mapsto \mu_i\mu_j\oad_{ij'}}} \,.
\end{equation}

\medskip
\noindent
Evoking the action of $B_{\vec\mu}$ on the Lax matrices $\CL_{\vec\mu}(x)$ and the behaviour of the twists,
as discussed in Subsection~\ref{ssec C all osc}, we see that conjugating~\eqref{eq:factor-C} with $B_{\vec\mu}$
and subsequently interchanging the twists, we obtain the following generalization of the
factorisation~\eqref{eq:T-via-QQ C-type basic} for any $\vec\mu\in \{\pm 1\}^r$:
\begin{equation}\label{eq:T-via-QQ C-type}
  T^+_{\vec\mu,\st}(x) = \ch_{\vec\mu,\st}^+ \cdot\, Q_{\vec\mu}(x+\st) Q_{\vec{\bar\mu}}(x-\st-r-1)
\end{equation}
with
\begin{equation}\label{eq:ch-C}
  \ch_{\vec\mu,\st}^+ = \tr \prod_{i=1}^r \tau_i^{\CF^{\vec\mu}_{ii}} = \,
  \frac{\prod_{i=1}^{r} \tau_i^{\mu_i\left(\st+(r-i+1)\delta_{\mu_i}^{-}+\sum_{k=1}^i \delta_{\mu_k}^{-}\right)}}
       {\prod_{1\leq i\leq j\leq r}\left(1-\tau_i^{-1}\tau_j^{-\mu_i \mu_j}\right)} \,,
\end{equation}
cf.~(\ref{eq:replC},~\ref{eq:char-C},~\ref{eq:ch-plus}).

\medskip
\noindent
Combining the factorisation formula~\eqref{eq:T-via-QQ C-type} with Theorem~\ref{thm:Main-C}, we get
(cf.~\eqref{eq:sigma-length C} and~\eqref{eq:ch-C}):

\medskip

\begin{Prop}\label{prop:main-C factorized}
For any $\st\in \BN$, we have:
\begin{equation}\label{eq:transfer-CQQ}
  T_{r,\st}(x)\, = \sum_{\vec\mu\in \{\pm 1\}^r} (-1)^{\mathsf{l}(\vec{\mu})}\, \ch_{\vec\mu,\st}^+ \cdot\,
  Q_{\vec\mu}(x+\st)Q_{\vec{\bar\mu}}(x-\st-r-1) \,.
\end{equation}
\end{Prop}

\medskip


\subsection{Two-term factorisation in D-type}
$\ $

Let us now discuss the straightforward $D$-type version of the results from the previous subsection
(going back to~\cite[\S5.1]{f}). Consider the particular example of the general
factorisation~\eqref{eq:linear factorization} applied in the case when $n=2r$, $a=r$,
and the $r\times r$ matrices $\am_1,\ap_1,\am_2,\ap_2$ are explicitly given by:
\begin{equation}\label{eq:specific D-blocks}
\begin{split}
  \am_1=
  \left(\begin{array}{cccc}
    \oa_{r',1} & \cdots & \oa_{r',r-1} & 0 \\
    \vdots & \iddots & 0 & -\oa_{r',r-1} \\
    \oa_{2',1} & 0 & \iddots & \vdots \\
    0 & -\oa_{2',1} & \cdots & -\oa_{r',1}
  \end{array}\right) \,, \quad
  \ap_1=
  \left(\begin{array}{cccc}
    \oad_{1,r'} & \cdots & \oad_{1,2'} & 0 \\
    \vdots & \iddots & 0 & -\oad_{1,2'} \\
    \oad_{r-1,r'} & 0 & \iddots & \vdots \\
    0 & -\oad_{r-1,r'} & \cdots & -\oad_{1,r'}
  \end{array}\right) \,, \\
  \am_2=
  \left(\begin{array}{cccc}
    \oa_{1,r'} & \cdots & \oa_{r-1,r'} & 0 \\
    \vdots & \iddots & 0 & -\oa_{r-1,r'} \\
    \oa_{1,2'} & 0 & \iddots & \vdots \\
    0 & -\oa_{1,2'} & \cdots & -\oa_{1,r'}
  \end{array}\right) \,, \quad
  \ap_2=
  \left(\begin{array}{cccc}
    \oad_{r',1} & \cdots & \oad_{2',1} & 0 \\
    \vdots & \iddots & 0 & -\oad_{2',1} \\
    \oad_{r',r-1} & 0 & \iddots & \vdots \\
    0 & -\oad_{r',r-1} & \cdots & -\oad_{r',1}
  \end{array}\right) \,,
\end{split}
\end{equation}
cf.~\eqref{eq:ApAmD}, where the only nontrivial commutators of the above entries are given by~\eqref{eq:ambient osc C}.

\medskip
\noindent
In this setup, both matrices $L_{r}(x)$ and $\bL_{r}(y)$ of~\eqref{eq:two matrices} are actually
$D_r$-type Lax matrices. In fact, the Lax matrix $L_{(+,\ldots,+)}(x)=L_{r}(x)$, see~\eqref{eq:laxlin},
appeared first in~\cite[\S4.1]{f}, cf.~\cite[(2.231)]{Frassek:2021ogy}. On the other hand, the Lax matrix
$L_{(-,\ldots,-)}(y)=\bL_{r}(y)$ can be obtained from~\eqref{eq:laxlin} via the Weyl group
action~\eqref{eq:Spm} followed by the total particle-hole transformation~\eqref{eq:total-ph},
exactly as in~\eqref{eq:opposite deg-Lax-C}.

\medskip
\noindent
Similarly to~\eqref{eq:C renormalized limit}, we note that both $L_{(+,\ldots,+)}(x)$ and $L_{(-,\ldots,-)}(x)$
can be obtained from the Lax matrix $\CL(x)$ of~\eqref{eq:linear-Lax-D} via the \emph{renormalized limit} procedures
(preserving the property of being Lax):
\begin{equation}\label{eq:D renormalized limit}
\begin{split}
  & L_{(+,\ldots,+)}(x) = \lim_{\st\to \infty}\
    \Big\{\CL(x-\st)\cdot \mathrm{diag}
    \Big(\underbrace{1,\ldots,1}_{r};\underbrace{-\sfrac{1}{2\st},\ldots,-\sfrac{1}{2\st}}_{r}\Big)\Big\} \,, \\
  & L_{(-,\ldots,-)}(x) = \lim_{\st\to \infty}\ \Big\{\mathrm{diag}
    \Big(\underbrace{\sfrac{1}{2\st},\ldots,\sfrac{1}{2\st}}_{r};\underbrace{1,\ldots,1}_{r}\Big)\cdot
    \CL(x+\st+r-1)\Big\} \Big|_{\ap\mapsto -\ap_2 \,,\, \am\mapsto -\am_2} \,.
\end{split}
\end{equation}

\medskip
\noindent
Then, the transformation~\eqref{eq:tra12} is again induced by the similarity transformation~\eqref{eq:gauge} with
\begin{equation}\label{eq:D-similarity}
  \bfS=\exp\left[\sum_{1\leq i< j\leq r} \oad_{ij'}\oa_{ij'}\right] \,,
\end{equation}
where we note that all the summands in~\eqref{eq:D-similarity} pairwise commute.

\medskip
\noindent
Combining the factorisation formula~\eqref{eq:linear factorization} with the similarity transformation~\eqref{eq:gauge},
we obtain:
\begin{equation}\label{eq:factor-D}
  L_{(+,\ldots,+)}(x+\st) L_{(-,\ldots,-)}(x-\st-r+1) = \, \bfS \CL (x) \Gop \bfS^{-1} \,,
\end{equation}
where
\begin{equation}
  \Gop =
  \left(\begin{BMAT}[5pt]{c:c}{c:c}
    \ID_r & \ap_2 \\
    0 & \ID_{r}
  \end{BMAT}\right) \,,
\end{equation}
$\bfS$ is given by~\eqref{eq:D-similarity}, and $\CL(x)$ is precisely the $D_r$-type Lax matrix of~\eqref{eq:linear-Lax-D}.

\medskip
\noindent
Following~(\ref{eq:qopsA},~\ref{eq:QC+},~\ref{eq:QC-}), we define the $Q$-operators
$Q_{(+,\ldots,+)}(x), Q_{(-,\ldots,-)}(x)\in \End(\BC^{2r})^{\otimes N}$~via:
\begin{equation}\label{eq:QD+}
  Q_{(+,\ldots,+)}(x)=\wtr_{D_{(+,\ldots,+)}}
  \Big(\underbrace{L_{(+,\ldots,+)}(x) \otimes \cdots \otimes L_{(+,\ldots,+)}(x)}_{N}\Big)
\end{equation}
and
\begin{equation}\label{eq:QD-}
  Q_{(-,\ldots,-)}(x)=\wtr_{D_{(-,\ldots,-)}}
  \Big(\underbrace{L_{(-,\ldots,-)}(x) \otimes \cdots \otimes L_{(-,\ldots,-)}(x)}_{N}\Big) \,,
\end{equation}
cf.~\eqref{eq:normalized trace}, with the twists $D_{(+,\ldots,+)}$ and $D_{(-,\ldots,-)}$
defined in analogy with~(\ref{eq:twist I})--(\ref{eq:twistconj}) via:
\begin{equation}
  D_{(+,\ldots,+)}\, = \prod_{1\leq i< j\leq r}\left(\tau_i\tau_{j}\right)^{-\oad_{ij'}\oa_{j'i}} \,, \qquad
  D_{(-,\ldots,-)}\, = \prod_{1\leq i<j\leq r}\left(\tau_i\tau_{j}\right)^{-\oad_{j'i}\oa_{ij'}} \,,
\end{equation}
cf.~\eqref{eq:two C twists}. Similarly to $C$-type considered in the previous subsection, we note that
the actions of these twists on the Fock modules are uniquely determined (up to scalar functions) by the
condition~\eqref{eq:twistconj CD}. Crucially, the twist $D_{(+,\ldots,+)}$ can be further expressed
via $\{\CF_{ii}\}_{i=1}^r$ of~(\ref{eq:F-generators D-type},~\ref{eq:D-diagonal})~as:
\begin{equation}\label{eq:twist-via-Cartan D}
  D_{(+,\ldots,+)}=\prod_{i=1}^r \tau_i^{\CF_{ii}-\st} \,.
\end{equation}

\medskip
\noindent
Building the monodromy matrices (by considering the $N$-fold tensor product on the matrix space)
from~\eqref{eq:factor-D}, taking the normalized trace, and evoking the relation~\eqref{eq:twist-via-Cartan D},
the factorisation formula~\eqref{eq:factor-D} implies the following factorisation for
the transfer matrices $T^+_{(+1,\ldots,+1),\st}(x)$ of~\eqref{eq:Dspin-inf-transfer}:
\begin{equation}\label{eq:T-via-QQ D-type basic}
  T^+_{(+1,\ldots,+1),\st}(x) = \ch_{(+1,\ldots,+1),\st}^+ \cdot\,
  Q_{(+,\ldots,+)}(x+\st) Q_{(-,\ldots,-)}(x-\st-r+1)
\end{equation}
with
\begin{equation}\label{eq:ch-D basic}
  \ch_{(+1,\ldots,+1),\st}^+ = \tr \prod_{i=1}^r \tau_i^{\CF_{ii}} = \,
  \prod_{i=1}^r \tau_i^{\st} \prod_{1\leq i<j\leq r} \frac{1}{1-\tau^{-1}_i\tau^{-1}_{j}} \,.
\end{equation}

\medskip
\noindent
Similarly to our treatment of $C$-type in the previous subsection, we shall generate
the remaining $Q$-operators directly from $Q_{(+,\ldots,+)}(x)$ of~\eqref{eq:QD+} via:
\begin{equation}\label{eq:QmuDspinor}
  Q_{\vec\mu}(x)=
  \left(B_{\vec\mu} \otimes \cdots \otimes B_{\vec\mu}\right) Q_{(+,\ldots,+)}(x)
  \left(B_{\vec\mu} \otimes \cdots \otimes B_{\vec\mu}\right)^{-1}
  \Big|_{\Big\{\tau_i\mapsto \tau_i^{-1} \,| \,\mu_i=-1 \Big\}} \,.
\end{equation}
As in type $C$, let us note the following key compatibility of this construction:
\begin{equation}\label{eq:QbmuDspinor}
  Q_{\vec{\bar\mu}}(x)=
  \left(B_{\vec\mu} \otimes \cdots \otimes B_{\vec\mu}\right) Q_{(-,\ldots,-)}(x)
  \left(B_{\vec\mu} \otimes \cdots \otimes B_{\vec\mu}\right)^{-1}
  \Big|_{\Big\{\tau_i\mapsto \tau_i^{-1} \,|\, \mu_i=-1 \Big\}} \,.
\end{equation}
The latter is a consequence of the following analogue of~\eqref{eq:relLbL}:
\begin{equation}\label{eq:relLbLD}
  B_{\vec\mu} L_{(-,\ldots,-)}(x)\Big|_{\overline{P.H.}} B_{\vec\mu}^{-1} =
  B_{\vec{\bar\mu}} L_{(+,\ldots,+)}(x)  B_{\vec{\bar\mu}}^{-1}
\end{equation}
(where $\overline{P.H.}$ stands for undoing the total particle-hole \eqref{eq:total-ph}),
which follows from the natural compatibility among the endomorphisms~\eqref{eq:fullBD}, cf.~\eqref{eq:techn C}:
\begin{equation}\label{eq:compatDspin}
  B_{(-1,\ldots,-1)}^{-1}B_{\vec\mu}^{-1}B_{\vec{\bar\mu}}=\id_{2r} \,.
\end{equation}

\medskip
\noindent
Thus, similarly to $C$-type, the factorisation~\eqref{eq:T-via-QQ D-type basic} admits the generalization
for any $\vec\mu\in \{\pm 1\}^r$:
\begin{equation}\label{eq:T-via-QQ D-type}
  T^+_{\vec\mu,\st}(x) = \ch_{\vec\mu,\st}^+ \cdot\, Q_{\vec\mu}(x+\st) Q_{\vec{\bar\mu}}(x-\st-r+1)
\end{equation}
with
\begin{equation}\label{eq:ch-D}
  \ch_{\vec\mu,\st}^+ = \tr \prod_{i=1}^r \tau_i^{\CF^{\vec\mu}_{ii}} = \,
  \frac{\prod_{i=1}^{r} \tau_i^{\mu_i\left(\st+(r-i-1)\delta_{\mu_i}^{-}+\sum_{k=1}^i \delta_{\mu_k}^{-}\right)}}
       {\prod_{1\leq i<j\leq r}\left(1-\tau_i^{-1}\tau_j^{-\mu_i \mu_j}\right)} \,,
\end{equation}
cf.~\eqref{eq:replD} and~\eqref{eq:char-Dspin}.

\medskip
\noindent
Combining the factorisation formula~\eqref{eq:T-via-QQ D-type} with Theorem~\ref{thm:Main-D}, we get
(cf.~(\ref{eq:even-odd},~\ref{eq:sigma-length D}) and~\eqref{eq:ch-D})):

\medskip

\begin{Prop}\label{prop:main-D factorized}
For any $\st\in \frac{1}{2}\BN$, we have:
\begin{equation}\label{eq:transfer-DQQ}
  T^\pm_{\st}(x)\, = \sum_{\vec\mu\in \{\pm 1\}^r_{\pm}} (-1)^{\mathsf{l}(\vec{\mu})}\, \ch_{\vec\mu,\st}^+ \cdot\,
  Q_{\vec\mu}(x+\st)Q_{\vec{\bar\mu}}(x-\st-r+1) \,.
\end{equation}
\end{Prop}

$\ $


\section{Factorisation for quadratic BD-types}\label{sec:facquad}

In this section, we factorise the infinite-dimensional quadratic $BD$-type transfer matrices~\eqref{eq:BD-inf-transfer}
into the products of two Baxter $Q$-operators arising from degenerate Lax matrices, alike~Section~\ref{sec linear factorizations}.

\medskip
\noindent
Consider the following two $\NK\times \NK$ matrices written in the block form as:
\begin{equation}\label{eq:quadlaxl 1}
{\small
  L_1(x)=
  \left(\begin{BMAT}[5pt]{c|c|c}{c|c|c}
    x^2+x\left(2-\sfrac{\NK}{2}-\wp_1\wm_1\right)+\sfrac{1}{4}\wp_1\idb\wp_1^T\wm_1^T\idb\wm_1 &
      x\wp_1-\sfrac{1}{2}\wp_1\idb\wp_1^T\wm_1^T\idb & -\sfrac{1}{2}\wp_1\idb\wp_1^T \\\
    -x\wm_1+\sfrac{1}{2}\idb\wp_1^T\wm_1^T\idb\wm_1 & x\id-\idb\wp_1^T\wm_1^T\idb & -\idb\wp_1^T \\
    -\sfrac{1}{2}\wm_1^T\idb\wm_1 & \wm_1^T\idb & 1 \\
  \end{BMAT}\right) }
\end{equation}
and
\begin{equation}\label{eq:quadlax2}
{\small
  L_{\NK}(y)=
  \left(\begin{BMAT}[5pt]{c|c|c}{c|c|c}
    1 & \wp_2 & -\sfrac{1}{2}\wp_2\idb\wp_2^T \\
    \wm_2 & y\id+\wm_2\wp_2 &
    -y\idb\wp_2^T-\sfrac{1}{2}\wm_2\wp_2\idb\wp_2^T \\
    -\sfrac{1}{2}\wm_2^T\idb\wm_2 &
     - y\wm_2^T\idb-\sfrac{1}{2}\wm_2^T\idb\wm_2\wp_2 &
       y^2+y\left(2-\sfrac{\NK}{2}+\wm_2^T\wp_2^T\right)+\sfrac{1}{4}\wm_2^T\idb\wm_2\wp_2\idb\wp_2^T  \\
  \end{BMAT} \right)}
\end{equation}
with $\id=\id_{\NK-2}$, $\idb=\idb_{\NK-2}$, while the length $\NK-2$ rows $\wp_1,\wp_2$ and columns $\wm_1,\wm_2$
are given by:
\begin{equation}
\begin{split}
  & \wm_1=(\oa_{2,1},\ldots,\oa_{\NK-1,1})^T \,, \qquad \wp_1=(\oad_{1,2},\ldots,\oad_{1,\NK-1}) \,,\\
  & \wm_2=(\oa_{1,2},\ldots,\oa_{1,\NK-1})^T \,, \qquad \wp_2=(\oad_{2,1},\ldots,\oad_{\NK-1,1}) \,,
\end{split}
\end{equation}
cf.~(\ref{eq:D-osc},~\ref{eq:B-osc}), where the only nontrivial commutators of the above entries are
given by~\eqref{eq:ambient osc C}.

\medskip
\noindent
Then, their product can be factorised as follows:
\begin{equation}\label{eq:BDgenfac}
  L_1\left(x-1+\sfrac{t}{2}+\sfrac{\NK}{4}\right)
  L_\NK\left(x-\sfrac{t}{2}-\sfrac{\NK}{4}\right) = \CL_{1}'(x) \Gop' \,,
\end{equation}
where
\begin{equation}\label{eq:gop'}
  \Gop'=
  \left(\begin{BMAT}[5pt]{c|c|c}{c|c|c}
    1 & \wp_{2}' & -\sfrac{1}{2}\wp_{2}'\idb(\wp_{2}')^T \\\
    0 & \id&-\idb(\wp'_{2})^T \\
    0 & 0 & 1 \\
  \end{BMAT} \right)
\end{equation}
and $\CL_{1}'(x)$ is the $\sso_\NK$-type Lax matrix obtained from $\CL_1(x)$ of~\eqref{eq:LaxBDmod}
by replacing $\wp\mapsto \wp'_1, \wm\mapsto \wm'_1$, with the following transformation in place:
\begin{equation}\label{eq:tra-quadr}
\begin{split}
  \wm_{1}' & =\wm_{1}-\wm_{2} \,, \qquad  \wm_{2}'=\wm_{2} \,, \\
  \wp_{2}' & =\wp_{2}+\wp_{1}\,, \qquad \wp_{1}'=\wp_{1} \,,
\end{split}
\end{equation}
cf.~\eqref{eq:tra12}. We note that the right-hand side of~\eqref{eq:BDgenfac} is independent of $\wm'_2$.

\medskip
\noindent
Actually, both $L_{1}(x)$, $L_{\NK}(y)$ of~(\ref{eq:quadlaxl 1},~\ref{eq:quadlax2})
are $\sso_{\NK}$-type Lax matrices. In fact, the Lax matrix $L_1(x)$ appeared first in~\cite[\S4.2]{f},
cf.~\cite[(2.237) and \S4.3]{Frassek:2021ogy}. On the other hand, the Lax matrix $L_{\NK}(y)$ can be obtained
from~\eqref{eq:quadlaxl 1} via the Weyl group action followed by a particle-hole transformation:
\begin{equation}
   L_{\NK}(x) = \hat{B}_{1'} L_{1}(x) \hat{B}_{1'}^{-1}\Big|_{P.H.}
\end{equation}
with the similarity matrix $\hat{B}_{1'}=\idb_{\NK}$, cf.~\eqref{eq:hat-B}, and
the following particle-hole transformation~\emph{P.H.}:
\begin{equation}
   \wp_1\mapsto -\wm_2^T \,, \qquad  \wm_1\mapsto \wp_2^T \,.
\end{equation}

\medskip
\noindent
We also note that the transformation~\eqref{eq:tra-quadr} is in fact induced by the similarity transformation:
\begin{equation}\label{eq:gauge quadratic}
\begin{split}
  \wm'_1=\bfS \wm_{1} \bfS^{-1} \,, \qquad  \wm'_2=\bfS \wm_{2} \bfS^{-1} \,, \\
  \wp'_2=\bfS \wp_{2} \bfS^{-1} \,, \qquad  \wp'_1=\bfS \wp_{1} \bfS^{-1}  \,,
\end{split}
\end{equation}
with
\begin{equation}\label{eq:BD-similarity}
  \bfS=\exp\left[\, \sum_{\ell=2}^{\NK-1} \oad_{1 \ell}\oa_{1 \ell}\, \right] \,,
\end{equation}
where all the summands in the right-hand side of~\eqref{eq:BD-similarity} pairwise commute,
cf.~(\ref{eq:gauge},~\ref{eq:A-similarity}).

\medskip
\noindent
Thus, combining the factorisation formula~\eqref{eq:BDgenfac} with  the similarity transformation~\eqref{eq:gauge quadratic},
we get:
\begin{equation}\label{eq:factor-BD}
  L_1\left(x-1+\sfrac{t}{2}+\sfrac{\NK}{4}\right)L_{\NK}\left(x-\sfrac{t}{2}-\sfrac{\NK}{4}\right) =
  \bfS \CL_{1}(x) \Gop \bfS^{-1} \,,
\end{equation}
where $\CL_{1}(x)$ is precisely the $\sso_\NK$-type Lax matrix of~\eqref{eq:LaxBDmod}, the matrix $\Gop$ is obtained
from~\eqref{eq:gop'} by replacing $\wp'_2,\wm'_2$ with $\wp_2,\wm_2$, and $\bfS$ is given by~\eqref{eq:BD-similarity}.

\medskip

\begin{Rem}
In type $D$, the factorisation formula~\eqref{eq:factor-BD} first appeared in~\cite[(5.30)]{f}.
\end{Rem}

\medskip
\noindent
Let us stress right away that both $L_{1}(x)$ and $L_{\NK}(x)$ can be obtained from the Lax matrix $\CL_1(x)$
of~\eqref{eq:LaxBDmod} via the \emph{renormalized limit} procedures
(which clearly preserve the property of being Lax):
\begin{equation}\label{eq:BD renormalized limit}
\begin{split}
  & L_{1}(x) = \lim_{\st\to \infty}\
    \Big\{\CL_1\left(x+1-\sfrac{\st}{2}-\sfrac{\NK}{4}\right)\cdot \mathrm{diag}
    \Big(1;\underbrace{-\sfrac{1}{\st},\ldots,-\sfrac{1}{\st}}_{\NK-2};\sfrac{1}{\st^2}\Big)\Big\}
    \Big|_{\wp\mapsto \wp_1 \,,\, \wm\mapsto \wm_1} \,,\\
  & L_{\NK}(x) = \lim_{\st\to \infty}\
    \Big\{\mathrm{diag}
    \Big(\sfrac{1}{\st^2};\underbrace{\sfrac{1}{\st},\ldots,\sfrac{1}{\st}}_{\NK-2};1\Big)\cdot
    \CL_1\left(x+\sfrac{\st}{2}+\sfrac{\NK}{4}\right)\Big\} \Big|_{\wp\mapsto -\wp_2 \,,\, \wm\mapsto -\wm_2}  \,.
\end{split}
\end{equation}

\medskip
\noindent
Following~(\ref{eq:qopsA},~\ref{eq:QC+},~\ref{eq:QC-},~\ref{eq:QD+},~\ref{eq:QD-}),
we define the $Q$-operators $Q_1(x), Q_{\NK}(x)\in \End(\BC^{\NK})^{\otimes N}$ via:
\begin{equation}
  Q_{1}(x)=\wtr_{D_{1}} \Big( \underbrace{L_{1}(x) \otimes \cdots \otimes L_{1}(x)}_{N} \Big)
\end{equation}
and
\begin{equation}\label{eq:QNK}
  Q_{\NK}(x)=\wtr_{D_{\NK}} \Big( \underbrace{L_{\NK}(x) \otimes \cdots \otimes L_{\NK}(x)}_{N} \Big) \,,
\end{equation}
cf.~\eqref{eq:normalized trace}, with the twists $D_{1}$ and $D_{\NK}$
defined in analogy with~(\ref{eq:twist I})--(\ref{eq:twistconj}) via:
\begin{equation}
  D_{1}=\tau_1^{-\sum_{\ell=2}^{\NK-1} \oad_{1 \ell}\oa_{\ell 1}}
    \prod_{i=2}^r\tau_i^{\oad_{1i}\oa_{i1}-\oad_{1i'}\oa_{i'1}} \,, \qquad
  D_{\NK}=\tau_1^{-\sum_{\ell=2}^{\NK-1} \oad_{\ell 1}\oa_{1 \ell}}
    \prod_{i=2}^r\tau_i^{\oad_{i1}\oa_{1i}-\oad_{i'1}\oa_{1i'}} \,.
\end{equation}
We note that the twist $D_{1}$ can be further expressed via $\{\CF_{ii}\}_{i=1}^r$ of~(\ref{eq:F-for-BD},~\ref{eq:cartanD2}) as:
\begin{equation}\label{eq:twist-via-Cartan BD}
  D_1=\tau_1^{-\st} \prod_{i=1}^r \tau_i^{\CF_{ii}} \,.
\end{equation}
Similarly to~\eqref{eq:twistconj} and~\eqref{eq:twistconj CD}, we should stress right away that the actions of the twists
$D_1$ and $D_{\NK}$ on the corresponding Fock modules are uniquely determined (up to scalars) by the conditions:
\begin{equation}\label{eq:twistconj BD}
  D L_{1}(x) D^{-1} = D^{-1}_{1}L_{1}(x)D_{1}  \,, \qquad
  D L_{\NK}(x) D^{-1} = D^{-1}_{\NK}L_{\NK}(x)D_{\NK}  \,,
\end{equation}
with $D$ given by:
\begin{equation}\label{eq:const twist BD}
\begin{split}
  & D_r\mathrm{-type}\colon \quad D=\diag\Big(\tau_1,\ldots,\tau_r,\tau^{-1}_r,\ldots,\tau^{-1}_1\Big) \,, \\
  & B_r\mathrm{-type}\colon \quad D=\diag\Big(\tau_1,\ldots,\tau_r,1,\tau^{-1}_r,\ldots,\tau^{-1}_1\Big) \,.
\end{split}
\end{equation}

\medskip
\noindent
Building the monodromy matrices (by considering the $N$-fold tensor product on the matrix space)
from~\eqref{eq:factor-BD}, taking the normalized trace, and evoking the relation~\eqref{eq:twist-via-Cartan BD},
the factorisation formula~\eqref{eq:factor-BD} implies the following factorisation formula for
the transfer matrices $T^+_{1,\st}(x)$ of~\eqref{eq:BD-inf-transfer}:
\begin{equation}\label{eq:T-via-QQ BD-type basic}
  T^+_{1,\st}(x) = \ch_{1,\st}^+ \cdot\,
  Q_1\left(x-1+\sfrac{t}{2}+\sfrac{\NK}{4}\right) {Q}_\NK\left(x-\sfrac{t}{2}-\sfrac{\NK}{4}\right)
\end{equation}
with the character $\ch^+_{1,\st}= \tr \prod_{i=1}^r \tau_i^{\CF_{ii}}$ given explicitly by:
\begin{equation}\label{eq:ch-BD basic}
\begin{split}
  & D_r\mathrm{-type}\colon \quad
    \ch_{1,\st}^+ = \, \tau_1^{\st}  \prod_{1<\ell\leq r}
    \frac{1}{\left(1-\frac{1}{\tau_1\tau_{\ell}}\right)\left(1-\frac{\tau_\ell}{\tau_1}\right)} \,,\\
  & B_r\mathrm{-type}\colon \quad
    \ch_{1,\st}^+ = \, \frac{\tau_1^{\st}}{\left(1-\frac{1}{\tau_1}\right)} \prod_{1<\ell\leq r}
    \frac{1}{\left(1-\frac{1}{\tau_1\tau_{\ell}}\right)\left(1-\frac{\tau_\ell}{\tau_1}\right)} \,.
\end{split}
\end{equation}

\medskip
\noindent
Following Section~\ref{sec linear factorizations}, we define $2r$ $Q$-operators using the action
of the operators $\hat B_k$ in~\eqref{eq:hat-B} as:
\begin{equation}
  Q_{k}(x)=
  \begin{cases}
    \left.\left(\hat{B}_k\otimes \cdots \otimes \hat{B}_k\right) Q_{1}(x)
    \left(\hat{B}^{-1}_k\otimes \cdots \otimes \hat{B}^{-1}_k\right) \right|_{\tau_{1}\leftrightarrow \tau_{k}}
       & \text{for } 1\leq k\leq r \\[0.3cm]
    \left.\left(\hat{B}_k\otimes \cdots \otimes \hat{B}_k\right) Q_{1}(x)|_{\tau_i\mapsto \tau_i^{-1}}
    \left(\hat{B}^{-1}_k\otimes \cdots \otimes \hat{B}^{-1}_k\right)\right|_{\tau_{1}\leftrightarrow \tau_{k}}
    & \text{for } r'\leq  k\leq 1'
  \end{cases} \,.
\end{equation}
The compatibility condition for the $Q_{\NK}(x)$ defined as in~\eqref{eq:QNK} and the above
$Q$-operators follows by employing the natural relation satisfied by the endomorphisms of~\eqref{eq:hat-B}
(cf.~\eqref{eq:compatDspin}):
\begin{equation}
  \hat{B}_k=\hat{B}_{k'} \hat{B}_{1'} \,, \qquad k\in \{1,\ldots,r\} \cup \{r',\ldots,1'\} \,.
\end{equation}

\medskip
\noindent
Acting with the endomorphisms $\hat{B}_k$'s on~\eqref{eq:factor-BD} and subsequently interchanging the twists,  we get:
\begin{equation}\label{eq:T-via-QQ BD-type}
\begin{split}
  & T^+_{k,\st}(x) = \ch_{k,\st}^+ \cdot\,
    Q_{k}\left(x-1+\sfrac{t}{2}+\sfrac{\NK}{4}\right) Q_{k'}\left(x-\sfrac{t}{2}-\sfrac{\NK}{4}\right) \,, \\
  & T^+_{k',\st}(x) = \ch_{k',\st}^+ \cdot\,
    Q_{k'}\left(x-1+\sfrac{t}{2}+\sfrac{\NK}{4}\right) Q_{k}\left(x-\sfrac{t}{2}-\sfrac{\NK}{4}\right)\,,
\end{split}
\end{equation}
for any $1\leq k\leq r$. Here, the characters read
\begin{equation*}
  \ch^+_{k,\st} = \tr \prod_{i=1}^r \tau_i^{\CF^{k}_{ii}} \,, \qquad
  \ch^+_{k',\st} = \tr \prod_{i=1}^r \tau_i^{\CF^{k'}_{ii}} \,,
\end{equation*}
see~(\ref{eq:F-for-BD general},~\ref{eq:F-for-BD action}), and are explicitly given by:
\begin{equation}\label{eq:ch-BD}
\begin{split}
  D_r\mathrm{-type}\colon \quad
  & \ch_{k,\st}^+ = \, \frac{\tau_1^{-1}\cdots \tau_{k-1}^{-1}\tau_k^{\st+k-1}}
         {\prod_{1\leq \ell<k}\left(1-\frac{\tau_k}{\tau_\ell}\right)
          \prod_{k< \ell \leq r}\left(1-\frac{\tau_\ell}{\tau_k}\right)
          \prod_{\ell\ne k}\left(1-\frac{1}{\tau_k\tau_\ell}\right)} \,,\\
  & \ch_{k',\st}^+ = \, \frac{\tau_1^{-1}\cdots \tau_{k-1}^{-1}\tau_k^{k+1-2r-\st}}
         {\prod_{1\leq \ell<k}\left(1-\frac{\tau_k}{\tau_\ell}\right)
          \prod_{k< \ell \leq r}\left(1-\frac{\tau_\ell}{\tau_k}\right)
          \prod_{\ell\ne k}\left(1-\frac{1}{\tau_k\tau_\ell}\right)} \,,\\
  B_r\mathrm{-type}\colon \quad
  & \ch_{k,\st}^+ = \, \frac{\tau_1^{-1}\cdots \tau_{k-1}^{-1}\tau_k^{\st+k-1}}
         {\left(1-\frac{1}{\tau_k}\right)
          \prod_{1\leq \ell<k}\left(1-\frac{\tau_k}{\tau_\ell}\right)
          \prod_{k< \ell \leq r}\left(1-\frac{\tau_\ell}{\tau_k}\right)
          \prod_{\ell\ne k}\left(1-\frac{1}{\tau_k\tau_\ell}\right)} \,,\\
  & \ch_{k',\st}^+ = \, \frac{\tau_1^{-1}\cdots \tau_{k-1}^{-1}\tau_k^{k-2r-\st}}
         {\left(1-\frac{1}{\tau_k}\right)
          \prod_{1\leq \ell<k}\left(1-\frac{\tau_k}{\tau_\ell}\right)
          \prod_{k< \ell \leq r}\left(1-\frac{\tau_\ell}{\tau_k}\right)
          \prod_{\ell\ne k}\left(1-\frac{1}{\tau_k\tau_\ell}\right)} \,,
\end{split}
\end{equation}
cf.~(\ref{eq:F-for-BD action},~\ref{eq:replBD}) and~(\ref{eq:char-D-first},~\ref{eq:char-B-first}).

\medskip
\noindent
Combining the factorisation formula~\eqref{eq:T-via-QQ BD-type} with Theorem~\ref{thm:Main-BD},
we arrive at the following result:

\medskip

\begin{Prop}\label{prop:main-BD factorized}
(a) In type $D_r$, for $\st\in \BN$ we have:
\begin{equation}\label{eq:transfer-DfirstQQ}
\begin{split}
  T_{1,\st}(x) =
  & \sum_{k=1}^r (-1)^{k-1} \ch_{k,\st}^+ \cdot\,
    Q_{k}\left(x-1+\sfrac{t}{2}+\sfrac{\NK}{4}\right) Q_{k'}\left(x-\sfrac{t}{2}-\sfrac{\NK}{4}\right) + \\
  & \sum_{k=1}^r (-1)^{k-1} \ch_{k',\st}^+ \cdot\,
    Q_{k'}\left(x-1+\sfrac{t}{2}+\sfrac{\NK}{4}\right) Q_{k}\left(x-\sfrac{t}{2}-\sfrac{\NK}{4}\right) \,.
\end{split}
\end{equation}

\medskip
\noindent
(b) In type $B_r$, for $\st\in \BN$ we have:
\begin{equation}\label{eq:transfer-BfirstQQ}
\begin{split}
  T_{1,\st}(x) =
  & \sum_{k=1}^r (-1)^{k-1} \ch_{k,\st}^+ \cdot\,
    Q_{k}\left(x-1+\sfrac{t}{2}+\sfrac{\NK}{4}\right) Q_{k'}\left(x-\sfrac{t}{2}-\sfrac{\NK}{4}\right) + \\
  & \sum_{k=1}^r (-1)^{k} \ch_{k',\st}^+ \cdot\,
    Q_{k'}\left(x-1+\sfrac{t}{2}+\sfrac{\NK}{4}\right) Q_{k}\left(x-\sfrac{t}{2}-\sfrac{\NK}{4}\right) \,.
\end{split}
\end{equation}
\end{Prop}

$\ $

\noindent
In the case of $B$-type, these results are in agreement with the functional relations found by Tsuboi
in~\cite{Tsuboi:2011iz,Tsuboi:2021xzl}. In the case of $D$-type, they appeared in~\cite{ffk}; see also
\cite{esv} where similar relations were derived from the ODE/IM correspondence \cite{Dorey:2007zx}.

$\ $


\section{Further generalizations}\label{sec:further}

One may wonder to which extent our key resolution~\eqref{eq:conjectured resolution 1} may be further
utilized to study spin chains.


\medskip
\noindent
\underline{\emph{Exceptional types}}

\noindent
While the present paper covers all examples of $\fg$-resolutions~\eqref{eq:conjectured resolution 1}
with $\lambda=\st\omega_i$ which can be further viewed as resolutions of $Y(\fg)$-modules for $\fg$ of classical types,
there are also three more examples for the case of exceptional $\fg$ as follows from \cite{cgy}
(corresponding to the vertices $i$ of the Dynkin diagram of $\fg$ with the label $1$,
see our Subsection~\ref{ssec BCD-overview}):
\begin{enumerate}

\item [$\bullet$]
$E_6$-type: vertices $i=1,5$

\item [$\bullet$]
$E_7$-type: vertices $i=1$

\end{enumerate}

\medskip
\noindent
\textbf{Type $E_6$:} According to~\eqref{eq:t-via-qq honest}, the transfer matrices $T_{1,\st}(x)$ and $T_{5,\st}(x)$
of the corresponding finite-dimensional representations of the highest weights $\st\omega_1$ and $\st\omega_5$,
respectively, may be written as alternating sums of $27$ transfer matrices associated to infinite-dimensional
representations $M'_{\bullet}$. Their highest weights are of the form $\st\omega+w_\omega(\rho)-\rho$, where $\omega$
runs through the set of weights of $L_{\omega_1}$ or $L_{\omega_5}$, respectively, and $w_\omega\in W_{E_6}$ are
the shortest elements in the Weyl group of $E_6$ such that $w_\omega(\omega_1)=\omega$ or $w_\omega(\omega_5)=\omega$,
respectively (note that the actions of $W_{E_6}$ on the sets of weights of $L_{\omega_1}$ and $L_{\omega_5}$
are transitive, as both representations are minuscule, and furthermore the stabilizers of each weight are
isomorphic to $W_{D_5}=(\BZ/2\BZ)^4\rtimes S_5$, the Weyl group of $D_5$).

\medskip
\noindent
\textbf{Type $E_7$:} According to~\eqref{eq:t-via-qq honest}, the transfer matrices $T_{1,\st}(x)$ of the corresponding
finite-dimensional representations of the highest weight $\st\omega_1$ may be written as alternating sums of $56$ transfer
matrices associated to infinite-dimensional representations $M'_{\bullet}$. The highest weights of these modules are of the form
$\st\omega+w_\omega(\rho)-\rho$, where $\omega$ runs through the set of weights of $L_{\omega_1}$ and $w_\omega\in W_{E_7}$
are the shortest elements in the Weyl group of $E_7$ such that $w_\omega(\omega_1)=\omega$ (note that $W_{E_7}$ acts
transitively on the set of the weights of $L_{\omega_1}$, with a stabilizer of each weight  isomorphic to $W_{E_6}$).

\medskip
\noindent
It is an interesting problem to realize the aforementioned $Y(\fg)$-representations $M'_{\bullet}$ through
the corresponding explicit Lax matrices at generic values of $\st\in \BC$. To this end, one should construct:
\begin{enumerate}

\item [$\bullet$]
  two polynomial oscillator-type Lax matrices $\CL_{1}(x),\CL_5(x)$ of size $27\times 27$ in type $E_6$

\item [$\bullet$]
  one polynomial oscillator-type Lax matrix of size $56\times 56$ in type $E_7$

\end{enumerate}
Not only those Lax matrices are presently unknown (see~\cite[\S7.7]{cgy} for their semiclassical limit),
but also explicit formulas for the corresponding $R$-matrices seem to be missing in the literature.


\medskip
\noindent
\underline{\emph{More general weights}}

\noindent
While in the present paper we only treated the examples with $\lambda$ being a multiple of a fundamental weight
from~\eqref{eq:KR-classification}, thus being exactly in the framework of the weights considered in~\cite{cgy},
we should stress that there do exist other cases when the $\fg$-module
resolution~\eqref{eq:conjectured resolution 1} does become a resolution of $Y(\fg)$-modules.
In particular, the oscillator-type Lax matrices of~\cite[(5.24)]{f} give rise to the explicit action of $Y(\sso_{2r})$
on the $\sso_{2r}$-modules $L_{\st_1\omega_1+\st_2\omega_r}$ and $L_{\st_1\omega_1+\st_2\omega_{r-1}}$
for any $\st_1,\st_2\in \BN$.


\medskip
\noindent
\underline{\emph{Trigonometric version}}

\noindent
The most interesting continuation of the current work, which shall be discussed elsewhere,
is the generalization of the present results to the trigonometric spin chains. To this end, let us recall
that the representation theory of the finite quantum group $U_q(\fg)$, over the field $\BC(q)$, is
equivalent to that of the underlying Lie algebra $\fg$. In particular, the description of the weights of
singular vectors in the Verma modules over $U_q(\fg)$ is precisely the same as for the Verma modules over $\fg$.
Thus, the modules $M'_{w\,\cdot \lambda}$ of~\eqref{eq:hard modules} and our key
resolutions~\eqref{eq:conjectured resolution 1} admit $q$-analogues. It would be interesting to construct those
in a self-contained way. For $(\fg,i)$ as in~\eqref{eq:KR-classification}, we therefore get resolutions
over the quantum loop algebra $U_q(L\fg)$, giving rise to the trigonometric version
of~\eqref{eq:t-via-qq honest} that provides a BGG-type formula for the transfer matrices of
finite-dimensional $U_q(L\fg)$-representations $L_{\st\omega_i}$
(the~length $N=0$ case of which recovers an interesting identity on the corresponding $q$-characters).


\medskip
\noindent
\underline{\emph{Supersymmetric version}}

\noindent
In view of the increasing interest in the integrable structures of supersymmetric gauge theories
as well as recent studies of quantum affine superalgebras and Yangians associated to Lie superalgebras,
it is desirable to generalize the results of the present paper to the rational (as well as trigonometric)
spin chains of super-type. We plan to return to this question in the follow-up work.

$\ $


%
%


\end{document}